\documentclass{emsprocart}

\usepackage{enumerate,graphicx,tikz,stmaryrd}
\usetikzlibrary{arrows,decorations,decorations.pathmorphing,positioning}
\usepackage{mdframed}
\usepackage{mdwlist}
\usepackage[pdftex]{hyperref}
\usepackage{amscd,mathrsfs,epic,empheq,float}
\usepackage{bbm}
\usepackage{cases}
 \usepackage{todonotes}
\usepackage[all]{xy}
\usepackage[nottoc]{tocbibind}
\usepackage{arydshln}
\tikzset{anchorbase/.style={baseline={([yshift=-0.5ex]current bounding box.center)}}}
\usetikzlibrary{decorations.markings}
\usetikzlibrary{decorations.pathreplacing}
\usetikzlibrary{arrows,shapes,positioning}
\tikzstyle directed=[postaction={decorate,decoration={markings,
    mark=at position #1 with {\arrow{>}}}}]
\tikzstyle rdirected=[postaction={decorate,decoration={markings,
    mark=at position #1 with {\arrow{<}}}}]
    
\usepackage[vcentermath]{youngtab}
\newcommand{\renc}{\renewcommand}

 \newlength{\baseunit}               
 \newcount{\numlines}                
\newtheorem{theorem}{Theorem}[section]
\newtheorem{lemma}[theorem]{Lemma}
\newtheorem{prop}[theorem]{Proposition}
\newtheorem{corollary}[theorem]{Corollary}
\newtheorem{conjecture}[theorem]{Conjecture}

\theoremstyle{definition}
\newtheorem{definition}[theorem]{Definition}

\newtheorem{remark}[theorem]{Remark}
\newtheorem{notation}[theorem]{Notation}
\newtheorem{example}[theorem]{Example}
\newtheorem{ex}[theorem]{Example}
\newtheorem{exs}[theorem]{Examples}
\newtheorem*{remarks}{Remark}

\newcommand{\de}{\delta}
\newcommand{\Brd}{\op{Br}_d(\de)}
\newcommand{\Bd}{\op{B}_d(\de)}

\newcommand{\cB}{{\mathcal B}}
\newcommand{\cC}{{\mathcal C}}

\newcommand{\cE}{{\mathcal E}}
\newcommand{\cF}{{\mathcal F}}

\newcommand{\cL}{{\mathcal L}}

\newcommand{\cO}{{\mathcal O}}
\newcommand{\cP}{{\mathcal P}}

\newcommand{\cS}{{\mathcal S}}

\newcommand{\ov}{\overline}

\newcommand{\D}{\mathbb{D}_{\Lambda}}

\def\down{\vee}
\def\up{\wedge}

\newcommand{\fin}{\op{fin}}

\newcommand{\mg}{\mathfrak{g}}

\newcommand{\mh}{\mathfrak{h}}

\newcommand{\mC}{\mathbb{C}}

\newcommand{\mZ}{\mathbb{Z}}

\newcommand{\parswitch}{\boldsymbol{\pi}}

\newcommand{\la}{\lambda}

\newcommand{\pla}{\ulcorner{\!\lambda}}
\newcommand{\pga}{\ulcorner{\!\gamma}}
\newcommand{\pmu}{\ulcorner{\!\mu}}

\newcommand{\op}{\operatorname}

\newcommand{\surj}{\mbox{$\rightarrow\!\!\!\!\!\rightarrow$}}

\DeclareMathOperator{\Hom}{Hom}

\definecolor{myred}{rgb}{1,.1,.1}

\def\eps{{\varepsilon}}
\newcommand{\un}{\underline}

\newcommand{\OSPrn}{\op{OSp}(r|2n)}
\newcommand{\SOSP}{\op{SOSp}}
\newcommand{\SOSPrn}{\op{SOSp}(r|2n)}
\newcommand{\osprn}{\mathfrak{osp}(r|2n)}
\newcommand{\glmn}{\mathfrak{gl}(m|n)}
\newcommand{\OSP}{\op{OSp}}

\newcommand{\osp}{\mathfrak{osp}}
\newcommand{\Brgraded}{\mathrm{Br}_d^{\mathrm gr}(\de)}
\newcommand{\Bgraded}{\mathrm{B}_d^{\mathrm gr}(\de)}

\contact[mehrig@math.uni-bonn.de]{Michael Ehrig, Mathematisches Institut, Universit\"at Bonn, Endenicher Allee 60, 53115 Bonn, Germany}
\contact[stroppel@math.uni-bonn.de]{Catharina Stroppel, Mathematisches Institut, Universit\"at Bonn, Endenicher Allee 60, 53115 Bonn, Germany}

\title[Finite-dimensional representations of $\OSPrn$]{On the category of finite-dimensional representations of $\OSPrn$: Part  I}

\author[Michael Ehrig, Catharina Stroppel]{Michael Ehrig, Catharina Stroppel\thanks{M.E. was financed by the DFG Priority program 1388. C.S. was partially supported by the Max-Planck institute in Bonn.}}

\begin{document}
\maketitle

\begin{abstract}
We study the combinatorics of the category $\cF$ of finite-dimensional 
modules for the orthosymplectic Lie supergroup $\OSP(r\, |\, 2n)$. 
In particular we present a positive counting formula for the dimension 
of the space of homomorphism between two projective modules. This 
refines earlier results of Gruson and Serganova. For each block $\cB$ we construct 
an algebra $A_\cB$ whose module category shares the combinatorics with 
$\cF$. It arises as a subquotient of a suitable limit of type 
${\rm D}$ Khovanov algebras. It turns out that $A_\cB$ is isomorphic to 
the endomorphism algebra of a minimal projective generator of $\cF$. This provides a direct link from $\cF$ to parabolic categories $\mathcal{O}$ of type ${\rm B}/{\rm D}$, with maximal parabolic of type ${\rm A}$, to the geometry of 
isotropic Grassmannians of types ${\rm B /} {\rm D}$ and to Springer fibres of type ${\rm C}/{\rm D}$. We also indicate why $\cF$ is not highest weight in general.
\end{abstract}

\tableofcontents

\renewcommand\thesection{\Roman{section}}
\section{Introduction} \label{sec:intro}

\renc{\thetheorem}{\Alph{theorem}}
Fix as ground field the complex numbers $\mathbb{C}$. This is the first part of a series of three papers, where we describe the category $\cC$ of {\it finite-dimensional representations} of the {orthosymplectic Lie supergroup} $G=\OSPrn$ equivalently the finite-dimensional {\it integrable representations} of the orthosymplectic Lie superalgebra $\mg=\osprn$. In particular we are  interested in the combinatorics and the structure of the locally finite endomorphism ring of a projective generator of this category. (To be more precise: a projective generator only exists as a pro-object, but we still call it a projective generator and refer to \cite[Theorem~2.4]{BD} for a detailed treatment of such a situation.)

Our main result is an {\it explicit description of the endomorphism ring of a minimal projective generator} for any block $\cB$ in $\cC$. We first describe in detail the underlying vector space in Theorem~\ref{A}, and then formulate the endomorphism theorem in Theorem~\ref{main}. As a consequence we deduce that the endomorphism algebra can be equipped with a $\mZ$-grading. The definitions and results are illustrated by several examples. Theorem~\ref{A} provides an elementary way to compute dimensions of homomorphism spaces between projective objects, and Theorem~\ref{main} allows a concrete description of the corresponding categories. In small examples, we provide a description of the category $\cC$ in terms of a quiver with relations.   

The proof of Theorem~\ref{main} will appear in Part~II, but we explain here the main ideas of the proof and the important and new phenomena which appear on the way. We believe that they are interesting on their own and also provide a conceptual explanation for the lack of desired properties of the category $\cC$ (in comparison to the type $A$ case). The arguments required for the complete proof of Theorem~\ref{main} will appear in Part~II, together with several applications to the representation theory.  We also defer to Part~II the proof of  Lemma~\ref{lemplusminus}, which is a rather easy observation as soon as the theory of Jucys-Murphys elements for Brauer algebras is available (which will be the case in Part~II).

Understanding the representation theory of algebraic supergroups and in particular their category $\cC$ of finite-dimensional representations  is an interesting and difficult task with several major developments in recent years. We refer to the articles \cite{SergICM}, \cite{BrundanICM}, \cite{MW} for a nice description and overview of the state of art. Despite these remarkable results, in particular for the general linear case, but also for the category $\cO$ for classical Lie superalgebras, there is still an amazingly poor understanding of the category $\cC$ outside of type $\rm A$. 

At least for the orthosymplectic case we can provide here some new insights into the structure of these categories by giving a construction of endomorphism algebras of projective objects in $\cC$.

Our results are in spirit analogous to \cite{BS4} and many of the applications deduced there for the general linear Lie superalgebra can be deduced here as well (investigated in detail in Parts II and III).  The orthosymplectic case however requires new arguments and a {\it totally new line of proof}. There are several subtle differences which make the case treated here substantially harder, the proofs more involved and conceptually different. The categories are much less well behaved than in type~$\rm A$.  To prove the main Theorem~\ref{main} we first need to develop the basic underlying combinatorics for the orthosymplectic case, make it accessible for explicit calculations and also for categorification methods, then use non-trivial results from the representation theory of Brauer algebras and the Schur-Weyl duality for orthosymplectic Lie supergroups, and finally connect both with the theory of Khovanov algebras of type $\rm D$. On the way we explain why (and to which extent) these categories are not highest weight, but we still manage to describe their combinatorics in terms of certain maximal parabolic Kazhdan-Lusztig polynomials of type $\rm D$ (or equivalently $\rm B$ by \cite[9.7]{ESperv}).

\subsection*{The main results and the idea of the proof} 
To explain our results in more detail, fix $r,n\in\mathbb{Z}_{\geq0}$ and consider a {\it vector superspace}, that is a $\mathbb{Z}_2$-graded vector space, $V=V_0\oplus V_1$ of superdimension $(r|2n)$ with its Lie superalgebra $\mathfrak{gl}(V)$ of endomorphisms, see Section~\ref{super} for a precise definition. Then $\mg=\osprn$ is the Lie super subalgebra of  $\mathfrak{gl}(V)$ which leaves invariant a fixed  non-degenerate super-symmetric bilinear form $\beta$ on $V$ (that is a form of degree zero, symmetric on $V_0$ and antisymmetric on $V_1$), and $G = \OSPrn$ is the corresponding supergroup of automorphisms preserving this form. In particular, the extremal cases $r=0$ respectively $n=0$ give the classical simple Lie algebras $\mathfrak{so}(r)$ respectively $\mathfrak{sp}(2n)$ with the corresponding orthogonal and symplectic groups. 

\begin{enumerate}
\renewcommand{\theenumi}{($\star$)}
\renewcommand{\labelenumi}{\theenumi}
\item \label{star} \begin{center} {\it For simplicity, we restrict ourselves in this introduction to the case where $r=2m+1$ is odd.} \end{center}
\end{enumerate}

Now consider the category $\cC'$ of finite-dimensional representations of the supergroup $G'=\SOSPrn$, or equivalently of finite-dimensional representations for its Lie algebra $\mg$ in the sense of \cite{Vera}, \cite{SergICM}. Like in the ordinary semisimple Lie algebra case, simple objects in $\cC'$  are, up to a parity shift $\parswitch$, the highest weight modules $L^\mg(\la)$ which arise as quotients of Verma modules whose highest weights $\la$ are integral and dominant. Hence for each such $\la$ we have two irreducible representations, $L^\mg(\la)$ and $\parswitch L^\mg(\la)$ in the category $\cC'$. 
More precisely $\cC'$ decomposes into a sum of two equivalent categories $\cC'=\cF'\oplus \parswitch(\cF')$, such that the simple objects in $\cF'$ are labelled by integral dominant weights.   In particular, it suffices to study the category $\cF'$. Similarly to \ref{star} we obtain the categories $\cC$ and $\cF$ if we work with $G=\OSPrn$. Under our assumption, an irreducible object in $\cF$ is just an irreducible object in $\cF'$ together with an action of the nontrivial element $\sigma\in G$, not contained in $G'$, by $\pm 1$ (see Proposition \ref{def:labelssimples1}). Let $X^+(G)$ be the labelling set of irreducible objects in $\cF$.

In contrast to the ordinary semsimple Lie algebra case, finite-dimensional representations of $\mg$ are in general {\it not completely reducible}. Already the tensor products $V^{\otimes d}$ of the natural representations $V$ need not be\footnote{They are in fact semsimple in case of the general linear Lie superalgebra by the Schur-Weyl duality theorem of Sergeev \cite{Ser} and Berele-Regev \cite{BR}, see \cite[Theorem~7.5]{BS5}, but in general not semisimple for $\osprn$, see \cite[(1.1), Remark~3.3]{ESSchurWeyl}.}. One goal of our series of papers is to understand possible extensions between simple modules and the decomposition of $V^{\otimes d}$.

The category $\cF$ is an interesting abelian tensor category with enough projective and injective modules (which in fact coincide, \cite[Proposition 2.2.2]{BKN}). We have therefore a non-semisimple $0$-Calabi-Yau category which has additionally a monoidal structure. 

The indecomposable projective modules are precisely the projective covers $P(\la)$ of the simple objects $L(\la)$ for $\la \in X^+(G)$.  Given a block $\cB$ of $\cC$ there is the notion of {\it atypicality} or {\it defect}, $\op{def}(\cB)$, which measures the non-semisimplicity of the block. In case the atypicality is zero, the block is semisimple. In general, our Theorem~\ref{main} implies that the Loewy length of any projective module in $\cB$ equals $2\op{def}(\cB)+1$. We expect that, up to equivalence, the block $\cB$ is determined by its atypicality, \cite[Theorem~2]{GS1}, see also Remark~\ref{remarkatyp}.

\begin{remarks}
Our assumption \ref{star} on $r$ simplifies the setup in this introduction,  since usually people (including also the above cited references)  would consider the category of finite-dimensional representations for the group $G'=\SOSPrn$ instead of the group $G=\OSPrn$. In case $r$ is odd, this makes no difference, since via the isomorphism of groups \eqref{easyodd} below, the representation theory does not change in the sense that any block for $G'$ gives rise to two equivalent blocks for $G$, each of which is equivalent to the original block for $G'$, see Section~\ref{sec:oddcase}. In the even case the interplay is more involved, see Section \ref{sec:evencase}. We prefer to work with $\OSPrn$ instead of $\SOSPrn$, for instance because it allows us to make a connection to Deligne categories \cite{Deligne2}, \cite{CH} and Brauer algebras \cite{Brauer}.
\end{remarks}
\subsection*{Dimension formula}
To access the dimension of $\Hom_{\cF}(P(\la),P(\mu))$, for $\la, \mu \in X^+(G)$, we encode the highest weights $\la$ and $\mu$ in terms of {\it diagrammatic weights} $\la$ and $\mu$ in the spirit of \cite{BS1}, see Definition~\ref{superweight}. Such a diagrammatic weight, see Definition~\ref{def:diagrammaticweight}, is a  certain \textit{infinite} sequence of symbols from $\{\times,\circ,\up,\down\}$, with the property that two weights $\la$ and $\mu$ are in the same block (abbreviating that $L(\la)$ and $L(\mu)$ are in the same block), if and only if the {\it core symbols} $\times$ and $\circ$ of the associated diagrammatic weights are at the same positions and the parities of $\up$'s agree, see also Proposition~\ref{blocksindiagramsodd} for a more precise statement. From Proposition~\ref{blocksindiagramsodd} it also follows that the set $\Lambda(\cB)$ of diagrammatic weights attached to a block $\cB$ is contained in {a \it diagrammatic block} ${\Lambda}$ in the sense of \cite[2.2]{ESperv}.

Following \cite{ESperv} we attach to  the diagrammatic weights $\la$ and $\mu$ via Definition~\ref{decoratedcups} a pair of  {\it cup diagrams} $\underline{\la}$, $\underline{\mu}$. If they have the same core symbols one can put the second on top of the first to obtain a {\it circle diagram}  $\underline{\la}\overline{\mu}$.  Our main combinatorial result (Theorem~\ref{mainprop}) is a {\it counting formula} for the dimensions: 

\begin{theorem}
\label{A}
The dimension of $\Hom_\cF(P(\la),P(\mu))$ equals the number of orientations $\underline{\la}\nu\overline{\mu}$ of $\underline{\la}\overline{\mu}$ if the circle diagram $\underline{\la}\overline{\mu}$ is defined and contains no non-propagating line, and the dimension is zero otherwise. 
\end{theorem}
By an {\it orientation} we mean another diagrammatic weight $\nu$ from the same block which, when putting it into the middle of the circle diagram, makes it oriented in the sense of Definition~\ref{orientedcircle}.  In other words, we factorize the symmetric Cartan matrix $C$ (see Theorem~\ref{Cartanmatrix}) into a product $C=AA^T$ with {\it positive integral} entries. 

In \cite[6.1]{ESperv} it was explained how to introduce an algebra structure $\mathbb{D}_{{\Lambda}}$ on the vector space with basis all oriented circle diagrams $\underline{\la}\nu\overline{\mu}$, where $\la,\mu,\nu \in \Lambda$. This algebra is called the {\it Khovanov algebra of type\footnote{Some readers might prefer to see here Khovanov algebras of type $B$ appearing, but as shown in \cite[9.7]{ESperv}, this is just a matter of perspective: a Khovanov algebra of type $B_n$ is isomorphic to one of type ${\rm D}_{n+1}.$} $\rm D$}  attached to the (diagrammatic) block $\Lambda$. By  \cite[Theorem 6.2]{ESperv} it restricts to an algebra structure on the vector space $\mathbb{D}_{\Lambda(\cB)}$ spanned by all circle diagrams $\underline{\la}\nu\overline{\mu}$ with $\la,\mu\in\Lambda(\cB)$ via the obvious idempotent truncation. Let $\mathbbm{1}_{\cB}$ be the corresponding idempotent projecting onto this subalgebra and consider the idempotent truncation  $\mathbbm{1}_{\cB} \mathbb{D}_{{\Lambda}} \mathbbm{1}_{\cB}$. To make the connection with the combinatorics of Theorem~\ref{A}, we prove in Proposition~\ref{isideal} that its oriented circle diagrams which contain at least one non-propagating line, span an ideal $\mathbb{I}$ in $\mathbbm{1}_{\cB} \mathbb{D}_{{\Lambda}} \mathbbm{1}_{\cB}$. We call this the {\it  nuclear ideal} and its elements {\it nuclear morphisms}. \\

Now our main theorem is the following, where $P=\oplus_{\la\in\Lambda(\cB)} P(\la)$ is a minimal projective generator of the chosen block $\cB$.
\begin{theorem}
\label{main}
There is an isomorphism of algebras 
\begin{eqnarray*}
\mathbbm{1}_{\cB} \mathbb{D}_{{\Lambda}} \mathbbm{1}_{\cB}/\mathbb{I}&\cong&\op{End}^{\fin}_\cF(P). 
\end{eqnarray*}
\end{theorem}
Here, $\op{End}^{\fin}_\cF(P)=\oplus_{\la\in\Lambda(\cB)}\op{Hom}_\cF(P(\la),P)$ denotes the locally finite endomorphism ring of $P$. This locally finiteness adjustment is necessary, since the labelling set $\Lambda(\cB)$ of the indecomposable projective modules in $\cB$ is infinite, and so we have to work with infinite blocks of diagrammatic weights. But we like to stress that for any chosen finite sum $\oplus_{\la\in J\subset \Lambda(\cB)} P(\la)$, the corresponding (ordinary) endomorphism ring is automatically finite dimensional. In practise, the endomorphism ring can then be computed in a quotient of an appropriate Khovanov algebra (of type ${\rm D}$ or equivalently of type ${\rm B}$) attached to a finite diagrammatic block.

Since $\mathbb{D}_{{\Lambda}}$ is by construction a (non-negatively) $\mZ$-graded algebra, and $\mathbb{I}$ is a graded ideal, we deduce that 

\begin{corollary}
\label{grading}
$\mathbbm{1}_{\cB} \mathbb{D}_{{\Lambda}} \mathbbm{1}_{\cB}/\mathbb{I}\cong\op{End}^{\fin}_\cF(P)$ is a graded algebra.
\end{corollary}
In analogy to the general linear supergroup case, \cite{BS4}, it is natural to expect that 
this grading is in fact a Koszul grading in the sense of \cite{MOS} which is a version of \cite{BGS} for locally finite algebras with infinitely many idempotents. This expectation is easy to verify for $\op{OSp}(3|2)$ using the description from Section~\ref{Koszulity}. 

\begin{conjecture} \label{conj:koszul}
The algebra $\op{End}^{\fin}_\cF(P)$ is Koszul.
\end{conjecture}

The Khovanov algebras $\mathbb{D}_{\Lambda}$ of type ${\rm D}$ for {\it finite} diagrammatic blocks arose originally from classical highest weight Lie theory, since they describe blocks of parabolic category $\cO$ of type ${\rm D}$ or equivalently of type ${\rm B}$ with maximal parabolic of type ${\rm A}$, see \cite[Theorem 9.1 and Theorem 9.22]{ESperv}, and hence describe the category of perverse sheaves on isotropic Grassmannians. They also have an interpretation in the context of the geometry of the Springer fibers of type ${\rm D}$ or ${\rm C}$ for nilpotent elements corresponding to two-row partitions, \cite{ES1}, \cite{Wilbert}. 

Our infinite diagrammatic weights $\Lambda$ can be interpreted as elements in an appropriate limit of a sequence of finite diagrammatic weights. As in \cite{BS1} the resulting algebras $\mathbb{D}_{\Lambda}$  could then also be viewed as a limit of algebras $\mathbb{D}_{\Lambda_n}$ for certain finite blocks $\Lambda_n$. Hence, up to the ideal $\mathbb{I}$,  our main theorem connects the category $\cF$ to classical (that means {\it non-super}) infinite-dimensional highest weight Lie theory and classical (i.e. non-super) geometry in an appropriate limit. This is similar to the result for the general linear supergroups, \cite[Theorem~1.2]{BS4}. It is also a shadow of the so-called super duality conjectures \cite{superduality}, but in a subtle variation, since we deal here with finite-dimensional representations instead of the highest weight category $\cO$. Moreover, taking this limit for type ${\rm D}$ Khovanov algebras is technically more difficult than in type ${\rm A}$, since the (naive) parallel construction mimicking the type ${\rm A}$ case would produce infinite weights with infinite defect. To circumvent this problem we apply a diagrammatic trick and introduce so-called {\it frozen vertices} which force our infinite cup diagramms to have a finite number of cups, see Definition \ref{superweight}, which means the defect stays finite. This procedure crucially depends on $r$ and $n$. We expect that this diagrammatic trick also provides the passage between the limit categories introduced in \cite{SergICM} and the category $\cF$.

\subsection*{Gruson-Serganova combinatorics}
The proof of Theorem~\ref{A} is heavily based on the main combinatorial results of Gruson and Serganova, \cite{GS1} and \cite{GS2}, who also introduced a version of cup diagram combinatorics for $\SOSPrn$ very similar to ours. An explicit translation between the two set-ups is given below in \eqref{T2}. There are however some small, but important differences in our approaches:
\begin{itemize}
\item Gruson and Serganova work with certain natural, but {\it virtual} modules in the Grothendieck group (the Euler characteristics $\mathcal{E}(\la)$), whereas our combinatorics relies on {\it actual filtrations} of the projective modules with the subquotients being shadows of cell modules for the Brauer algebra.
\item Gruson and Serganova's formulas are {\it alternating} summation formulas, while ours are {\it positive  counting formulas}.
\item Gruson and Serganova work with the {\it special} orthosymplectic group, whereas we work with the orthosymplectic group,  which is better adapted to the diagram combinatorics and connects directly to the representation theory of Brauer algebras via \cite[Theorem~3.4]{SergICM}, \cite[Theorem~5.6]{LZg}.
\item Gruson and Serganova's cup diagram combinatorics unfortunately does not give a direct connection to the theory of Hecke algebras and Kazhdan-Lusztig polynomials, whereas our Khovanov algebra of type ${\rm D}$ is built from the Kazhdan-Lusztig combinatorics of the hermitian symmetric pair $(\rm{D}_n,\rm{A}_{n-1})$,  see \cite{LS}, \cite{ESperv}.
\end{itemize}

Comparing Theorem~\ref{A} with \cite[(5.15)]{BS3} and \cite[Theorem 2.1]{BS4}, our formulas indicate that one could expect some highest weight structure or at least some cellularity of each block $\cB$ of $\cF$ explaining our positive counting formulas and appearance of Kazhdan-Lusztig polynomials. But blocks of $\cF$ are not highest weight and not even cellular in general, as the example from Section~\ref{32} illustrates, and there are no obvious candidates for cell modules. This is a huge difference to the 
case  of $\glmn$, where parabolic induction of a finite-dimensional representation of  the Levi subalgebra $\glmn_0=\mathfrak{gl}(m)\oplus\mathfrak{gl}(n)$ produces a finite-dimensional {\it Kac module}. These modules are the standard modules for the highest weight structure of the category of integrable finite-dimensional representations in that case, see \cite[Theorem 4.47]{Brundan} or \cite[Theorem 1.1]{BS4}. Such a parabolic subalgebra, and hence such a class of modules is however not available for $\mg=\osprn$ if $r\geq 2, n\geq 1$.  Nevertheless, we claim that our counting formula arises from some natural filtrations on projective objects, whose origin we like to explain now.

\subsection*{Tensor spaces and Brauer algebras}
The tensor spaces $V^{\otimes d}$ for $d\geq 0$ from above already contain in some sense the complete information about the category $\cF$. Namely, each indecomposable projective $P(\la)$ occurs in $V^{\otimes d}$ for some large enough $d$, see e.g. \cite[Lemma 7.5]{CH}. By weight considerations and the 
action of $\sigma$ one can easily check that 
$\Hom_G(V^{\otimes d}, V^{\otimes d'})=\{0\}$ if $d$ and $d'$ have different parity (see also Remark~\ref{rem:brauercat}). Hence to understand the spaces of morphisms between projective modules in a fixed block $\cB$ of $\cF$, it suffices to consider the tensor spaces for each parity of $d$ separately.  Moreover, since the trivial representation appears as a quotient of $V\otimes V$ (via the pairing given by $\beta$), we have a surjection $P\otimes V\otimes V\surj  P\otimes\mC=P$ which splits if $P$ is projective. 
Thus we obtain the following:
\begin{lemma}
\label{B}
Let $J\subset\Lambda(\cB)$ be a finite subset of weights such that all $P(\la)$ are in the same block $\cB$ of $\cF$. Then $P'=\oplus_{\la\in J} P(\la)$ appears as a direct summand of $V^{\otimes d}$ for some large enough $d$.
\end{lemma}

To achieve our goal (to determine the endomorphism ring of all such $P'$) we first consider endomorphisms of these tensor spaces $V^{\otimes d}$. For this we use a super analogue of a result from classical invariant theory of the semisimple orthogonal and symplectic Lie algebras studied by Brauer in \cite{Brauer}.  

For fixed  $d\in\mathbb{Z}_{\geq0}$ and $\de\in\mC$, the {\it Brauer algebra} $\Brd$  is an algebra structure on the vector space with basis all equivalence classes of Brauer diagrams for $d$. A {\it Brauer diagram for $d$} is a partitioning of the set $\{\pm1,\pm2,\ldots, \pm d\}$ into two element subsets. One can display this by  identifying $\pm  j$ with the point $(j,\pm 1)$ in the plane and connect two points in the same subset by an arc inside the rectangle $[1,d]\times [-1,1]$. Here is an example of a Brauer diagram for $d=11$:
 

\abovedisplayskip0.25em
\belowdisplayskip0.25em
 \begin{equation}
\label{diagram}
\begin{tikzpicture}[thick,>=angle 90]
\begin{scope}[xshift=8cm]
\draw (0,0) -- +(0,1);
\draw (.6,0) -- +(.6,1);
\draw (1.2,0) -- +(-.6,1);
\draw (1.8,0) to [out=90,in=-180] +(.9,.5) to [out=0,in=90] +(.9,-.5);
\draw (1.8,1) to [out=-90,in=-180] +(.6,-.4) to [out=0,in=-90] +(.6,.4);
\draw (2.4,0) -- +(0,1);
\draw (3,0) -- +(.6,1);
\draw (4.2,0) -- +(0,1);
\draw (4.8,0) to [out=90,in=-180] +(.3,.3) to [out=0,in=90] +(.3,-.3);
\draw (4.8,1) to [out=-90,in=-180] +(.3,-.3) to [out=0,in=-90] +(.3,.3);
\draw (6,0) -- +(0,1);
\end{scope}
\end{tikzpicture}
\end{equation}
 
Given two Brauer diagrams $D_1$ and $D_2$ we can stack $D_2$ on top of $D_1$. The result is again a Brauer diagram $D$ after we removed  possible internal loops and the process is independent of the chosen visualization. Setting $D_1D_2=\de^cD$, where $c$ is the number of internal loops removed, defines the associative algebra structure  $\op{Br}_d(\de)$ on the vector space with basis given by Brauer diagrams. Here is an example of the product of two basis vectors:

\abovedisplayskip0.25em
\belowdisplayskip0.25em
\begin{equation}
\label{multiplication}
\begin{tikzpicture}[thick,>=angle 90]
\begin{scope}
\draw (0,0) -- +(0,1);
\draw (.6,0) -- +(.6,1);
\draw (1.2,0) to [out=90,in=-180] +(.3,.3) to [out=0,in=90] +(.3,-.3);
\draw (.6,1) to [out=-90,in=-180] +(.6,-.4) to [out=0,in=-90] +(.6,.4);
\node at (2.4,.5) {$\bullet$};
\end{scope}
\begin{scope}[xshift=3cm]
\draw (0,1) to [out=-90,in=-180] +(.3,-.3) to [out=0,in=-90] +(.3,.3);
\draw (1.2,1) to [out=-90,in=-180] +(.3,-.3) to [out=0,in=-90] +(.3,.3);
\draw (0,0) to [out=90,in=-180] +(.6,.4) to [out=0,in=90] +(.6,-.4);
\draw (.6,0) to [out=90,in=-180] +(.6,.4) to [out=0,in=90] +(.6,-.4);
\node at (2.4,.5) {$=$};
\end{scope}
\begin{scope}[xshift=6.5cm]
\node at (-.5,.5) {$\de$};
\draw (0,1) to [out=-90,in=-180] +(.6,-.4) to [out=0,in=-90] +(.6,.4);
\draw (.6,1) to [out=-90,in=-180] +(.6,-.4) to [out=0,in=-90] +(.6,.4);
\draw (0,0) to [out=90,in=-180] +(.6,.4) to [out=0,in=90] +(.6,-.4);
\draw (.6,0) to [out=90,in=-180] +(.6,.4) to [out=0,in=90] +(.6,-.4);
\end{scope}
\end{tikzpicture}
\end{equation}

  We use the following important result.
\begin{prop}[{\cite[Theorem~3.4]{SergICM}, \cite[Theorem 5.6]{LZfirst}}] \label{prop:Brauersurj}
 Let $\de=r-2n$.  Then the canonical algebra homomorphism
 \begin{eqnarray}
\label{surj}
\op{Br}_d(\de)&\surj& \op{End}_{\OSP(r|2n)}(V^{\otimes d})
\end{eqnarray}
is surjective.
 \end{prop}   

Hereby a Brauer diagram $D$ acts on a tensor product $v_1\otimes v_2\otimes\cdots\otimes v_d$ as follows: We identify the $d$ tensor factors with the bottom points of the diagrams. Whenever there is a cap (connecting horizontally two bottom points) we pair the corresponding vectors using $\beta$ and obtain a scalar multiple of the vector $v_{i_1}\otimes \cdots \otimes v_{i_t}$, where $t$ equals $d$ minus twice the number of caps and $v_{i_j}=v_{k}$ if the $j$th top point not connected to another top point (by a cup) is connected with the $k$th point at the bottom.  Finally we insert for each cup a pair of new factors arising as the image of $1 \in \mathbb{C}$ under the counit map $\mC\mapsto V\otimes V$, see e.g. \cite[(3.3)]{SergICM} or \cite[Theorem~3.11]{ESSchurWeyl} for details.

We like to stress that the map \eqref{surj} fails to be surjective in general if we work with $G=\SOSP(2m|2n)$ or its Lie algebra $\osp(2m|2n)$, see  e.g. \cite{LZg} and  \cite[Remark 5.8]{ESSchurWeyl} as well as Remark~\ref{obsVera}.

Because of Proposition~\ref{prop:Brauersurj}, we chose to work with $G=\OSPrn$ instead of the more commonly studied supergroup $\SOSPrn$. This requires then however a translation and adaption of the results from the literature (including \cite{GS1}, \cite{GS2}) to $\OSPrn$. In case $r=2m+1$ is odd this is an easy task, since we have 
 \begin{eqnarray}
 \label{easyodd}
 \op{OSp}(2m+1|2n)&\cong&\op{SOSp}(2m+1|2n)\times \mZ /2\mZ,
 \end{eqnarray}
   where the generator of the cyclic group is minus the identity. If $r=2m$ is even, we only have a semidirect product
    \begin{eqnarray}
    \label{semidirect}
   \op{OSp}(2m|2n)\cong\op{SOSp}(2m|2n)\rtimes \mZ /2\mZ,
    \end{eqnarray}
 and the situation is rather involved. A larger part of the present paper is devoted to this problem. We believe that, in contrast to the case of $\SOSPrn$, the blocks for $\OSPrn$ are completely determined by their atypicality, see Remark~\ref{remarkatyp}. 

\begin{remarks}
\label{rem:brauercat}
Instead of considering only single tensor product spaces $V^{\otimes d}$ as in \eqref{surj}, one might prefer to  work with the  tensor subcategory $(V,\otimes)$ of $\cF(\OSP(r|2n))$ generated by $V$ (for any fixed nonnegative integers $r$, $n$). Then the surjection \eqref{surj} can in fact be extended to a full monoidal functor from the Brauer {\it category} $\op{Br}(\de)$ to $(V,\otimes)$, see e.g. \cite{CW}. An object in the Brauer category (which is just a natural number $d$) is sent to $V^{\otimes d}$ and a basis morphism (that is a Brauer diagram as in \eqref{diagram} but not necessarily with the same number of bottom and top points) is sent to the corresponding intertwiner. Hence the Brauer category controls all intertwiners. Again, this statement is not true for the special orthosymplectic groups, not even for the odd cases $\op{SOSp}(2m+1|2n)$, since one can find some integer $d$ with a non-trivial morphism from $V^{\otimes d}$ to $V^{\otimes (d+1)}$, see Remark~\ref{obsVera}. Such a morphism however can not come from a morphism in the Brauer category, since for a diagram in the Brauer category the number of dots on the top and on the bottom of the diagram have the same parity. 
The Karoubi envelope of the additive closure of the Brauer category can also be identified with Deligne's universal symmetric category $\underline{\op{Rep}}(O_\delta)$, \cite{Deligne2}, as used e.g. in \cite{CH}, \cite{SergICM}.
\end{remarks}

Using Proposition~\ref{prop:Brauersurj} and Lemma~\ref{B}, we can find an idempotent $e=e_{d,\delta}$ in $\Brd$ such that the following holds:

\begin{prop}
\label{D}
Let $J$ and $P'=\oplus_{\la\in J} P(\la)$ be as in Lemma~\ref{B}. Then there is an idempotent $e=e_{d,\delta}$ in $\Brd$ together with a surjective algebra homomorphism 
\begin{eqnarray}
\label{eBe}
\Phi=\Phi_{d,\delta}:\,\, e\Brd e&\surj&\op{End}_\cF(P')
\end{eqnarray}
identifying the primitive idempotents in both algebras. In particular every idempotent in $\op{End}_\cF(P')$ lifts.
\end{prop}

Comes and Heidersdorf obtain in \cite[Theorem 7.3]{CH} a classification of the indecomposable summands in $V^{\otimes d}$ in terms of idempotents of the Brauer algebra and our (yet another) cup diagram combinatorics for the Brauer algebra developed in \cite{ES2}. They moreover prove in \cite[Lemma 7.15]{CH} that the indecomposable {\it projective} summands $P(\la)$ correspond to cup diagrams with maximal possible number, namely $\op{min}\{m,n\}-\op{rk}(\la)$, of cups. Here $\op{rk}(\la)$ denotes the rank of $\lambda$ which is a combinatorially defined nonnegative number. Unfortunately their theorem provides no way to read off the weight $\la$ from the cup diagram. In part II we will show that a diagrammatic trick as in \cite[Lemma 8.18]{BS5} for the walled Brauer algebra can be applied in our set-up as well and provides a correspondence that allows to read off the head of each indecomposable summand, and in particular of the projective summand.

More precisely,  let $c$ be the cup diagram corresponding to a projective summand in $V^{\otimes d}$ via \cite[Lemma 7.15]{CH}. 
 Let $\nu$ be the corresponding diagrammatic weight, that is the unique diagrammatic weight $\nu$ such that $c=\underline{\nu}$, see Remark~\ref{weightforc}. Now let $\nu^\dagger$ be the diagrammatic weight obtained from $\nu$ by changing all labels attached to rays in $c$ from $\up$ into $\down$ and from $\down$ into $\up$. Then the correspondence is given by the following:

\begin{prop}
\label{PropE}
In the set-up from above we have $\nu^\dagger=\la^\infty$, with $\la^\infty$ the infinite diagrammatic weight attached to $\la$ via \eqref{Sla}.
\end{prop}
For examples see Section~\ref{sec:examples}. Note that for  $P'$ (as in Lemma~\ref{B}) Proposition~\ref{PropE} provides a description of the idempotent $e$ in \eqref{eBe}.

\subsection*{The shadow of a quasi-hereditary or cellular structure}
\label{shadow}
Fortunately, the (complex) representation theory of $\Brd$ for arbitrary $\de\in\mZ$ is by now reasonably well understood thanks to the results in \cite{PaulMartin},  \cite{CDM}, \cite{CoxdVi}, \cite{ES2}, \cite{ESKoszul}. In particular it is known that  $\Brd$ is a quasi-hereditary algebra if $\de\not=0$ and still cellular in case $\de=0$,  \cite{PaulMartin}, see also \cite[Theorem~5.4]{ES2}.
 For the sake of simplicity let us assume for the next paragraph that $\de\not=0$. Let $\cP_d$ be the usual labelling set of simple modules for $\Brd$  by partitions, see \cite{PaulMartin}, \cite{CoxdVi}, and denote by $L_d(\alpha)$,  $P_d(\alpha)$,  and $\Delta_d(\alpha)$ the simple module, its projective cover and the corresponding standard module respectively attached to $\alpha\in \cP_d$. Standard properties for quasi-hereditary algebras, the BGG-reciprocity (see \cite[A2.2~(iv)]{Donkin}) and the existence of a duality preserving the simple objects, give us that
 \begin{eqnarray}
 \op{dim}\op{Hom}_{\Brd}(P_d(\alpha),P_d(\beta))&=&[P_d(\beta) :L(\alpha)]\quad\nonumber
 \end{eqnarray}
 is equal to 
 \begin{eqnarray}
 \sum_{\eta\in\cP_d}[\Delta_d(\eta) :L(\alpha)](P_d(\beta) :\Delta(\eta))
 &=&\sum_{\eta\in\cP_d}(P_d(\alpha) :\Delta(\eta))(P_d(\beta) :\Delta(\eta))\quad\quad\label{qh}
 \end{eqnarray}
where $[M:L]$ denotes the multiplicity of a simple module $L$ in a Jordan-H\"older series of $M$ and $(P:\Delta)$ denotes the multiplicity of $\Delta$ appearing as a subquotient in a standard filtration of $P$.  As first observed in \cite{PaulMartin}, see also \cite{CoxdVi}, all the occurring multiplicities are either $0$ or $1$ and given by some parabolic Kazhdan-Lusztig polynomial (which is in fact monomial) evaluated at $1$. 

Now since $\Brd$ is quasi-hereditary with standard modules $\Delta_d(\alpha)$, the idempotent truncation $e\Brd e$ from \eqref{eBe} is cellular, with cell modules $\Delta_d(\alpha)e$, see \cite[Proposition~4.3]{KoenigXi}. Hence the endomorphism algebra in question is by Proposition~\ref{D} {\it a quotient} of a cellular algebra. Unfortunately,  we have the following:
\begin{center}
{\it Quotients of cellular algebras need not be cellular}. 
\end{center}

However, there is still some extra structure. Given a projective $e\Brd e$-module $P_d(\la)$ with $\la^\dagger\in J$ and $e$ as in Proposition~\ref{D}, and a fixed filtration with subquotients certain cell modules $\Delta^{e\Brd e}(\nu^\dagger)$, then this filtration induces a filtration\footnote{More generally given a finite-dimensional algebra $A$ and a quotient algebra $A/I$ with surjection $\Psi:A\rightarrow A/I$, any $A$-module filtration of $Ae_\la$ for an idempotent $e_\la$  induces a filtration on $A/I \Psi(e_\la)$ by taking just the image.} of the projective module $P(\la^\dagger)\in\cF$ via the algebra homomorphism $\Phi_{d,\delta}$.
 
The shape of the successive subquotients, $\Delta^{\cF}(\la^\dagger, \nu^\dagger)=\Delta(\la^\dagger, \nu^\dagger)$  does however in general  not only depend on $\nu^\dagger$, but also on $\la^\dagger$, that means on the projective module we chose. (A priori, in case of higher multiplicities,  two subquotients might even differ although they arise from isomorphic cell modules in $P_d(\la)$. But this turns out to be irrelevant for our counting and so we can ignore it.)  It is the multiplicities of these quotients of the cell modules which we count in our main Theorem~\ref{A}. In particular we still have a well-defined {\it positive counting formula} for the multiplicities for each given pair $(\la, \nu)$. Hence, although we do not have standard or cell modules, we still have some control about the structure of projective objects. Moreover,
\begin{center}
{\it The failure of quasi-heriditarity and cellularity of the category $\cF$ is encoded in the kernel of the maps $\Phi_{d,\delta}$, from \eqref{eBe}.}
\end{center}
We need now to connect this information with Theorem~\ref{A} and describe the kernel.

\subsection*{Graded version  $\Brgraded$ of the Brauer algebra $\Brd$}
To determine the number $(P_d(\alpha) :\Delta_d(\eta))$ one can, as in \cite{CoxdVi} or \cite{ES2},  first assign to each of the partitions $\alpha$ and $\eta$ a diagrammatic weight, denoted by the same letter and compute the corresponding cup diagram $\underline{\alpha}$  using the rules in Definition~\ref{decoratedcups}. Then the multiplicity in question is non-zero (and therefore equal to $1$) if and only if $\underline{\alpha}\eta$ is oriented in the sense of Figure~\ref{oriented}, see \cite[(8.64)]{ESperv}. 

Now consider the endomorphism ring 
\[
\Bd:=\op{End}_{\Brd}(\oplus_{\alpha\in\cP_d}P_d(\alpha))
\]
of a minimal projective generator of  $\Brd$. That is $\Bd$ is the basic algebra underlying $\Brd$. 
Then a basis of $\Bd$ can be labelled by pairs of oriented cup diagrams of the form $(\underline{\alpha}\eta, \underline{\beta}\eta)$  or equivalently by oriented circle diagrams  $\underline{\alpha}\eta\overline{\beta}$, where $\alpha, \eta, \beta\in \cP_d$.

By \cite[Section 6.2]{ESperv}, there is an algebra structure  $\Bgraded$ on the vector space spanned by such circle diagrams using the multiplication rules of the type ${\rm D}$ Khovanov algebra from \cite{ESperv}. Using the degree of circle diagrams from Figure~\ref{oriented}, this turns $\Bd$ into a $\mZ$-graded algebra $\Bgraded$. It provides a new realization of the basic Brauer algebra, namely a graded lift of the basic algebra $\Bd$:

\begin{theorem}{\cite[Theorem A]{ESKoszul}}
The algebra $\Bgraded$ is isomorphic to the basic Brauer algebra $\Bd$ as ungraded algebras. 
\end{theorem}
In fact this grading can also be extended to provide a grading on $\Brd$, but for our purposes it suffices to work with the basic algebra $\Bd$. 

\subsection*{Explicit endomorphism algebra}
Given the diagrammatic description  $\Bgraded$ of $\Bd$,  the idempotent truncation  $e\Brgraded e$  (which is by definition a subalgebra) is easily described by only allowing certain cup diagrams depending on $e$, in fact precisely the $\underline{\la}$ corresponding to elements in $J$. However, the description of the kernel of $\Phi_{d,\delta}$ is more tricky. For \eqref{surj} this kernel was described in \cite{LZsecond}, but their description is not very suitable for our purposes. Instead  we obtain a similar result as in \cite[Theorem~8.1 and Corollary~8.2]{BS5} (although the proof is quite different), which will be explained in Part~II. It implies that the kernel is controlled by the ideal $\mathbb{I}$ of nuclear endomorphisms.

To summarize: For any choice of block $\cB$ and set of weights $J$ as in Lemma~\ref{B} and $\Phi_{d,\delta}$ as in Proposition~\ref{D}, we will  map  in Part~II the circle diagrams from $\Brgraded$, picked out by $e\Brgraded e$, to the corresponding basis element of some Khovanov algebra $\mathbb{D}_\Lambda$ using the identification from Proposition~\ref{PropE} and the identification from Definition~\ref{superweight} of integral highest weights with diagrammatic weights. We will show that under this assignment the kernel of $\Phi_{d,\delta}$ restricted to $e\Brd e$ is mapped to the ideal $\mathbb{I}$ of nuclear circle diagrams.  As a result we deduce finally Theorem~\ref{main}. The construction induces a grading on each block $\cB$ which we expect to be Koszul.

\subsection*{Acknowledgements}We are very grateful to Vera Serganova for sharing her insight and pointing out a mistake in a first draft of the paper. We thank  Jonathan Comes, Kevin Coloumbier, Antonio Sartori and Wolfgang Soergel for useful discussions, and Volodymyr Mazorchuk and Michel Van den Bergh for remarks on a first draft. We finally thank the referee for many helpful comments.

\section{An illustrating example: \texorpdfstring{$\cF(\op{SOSp}(3|2))$}{F(SOSp(3,2)}}
\label{32}

Before we start we describe blocks of $\cF(\op{SOSp}(3|2))$ in terms of a quiver with relations using  Theorems~\ref{A} and~\ref{main}, see also Section~\ref{sec:oddcase} for the precise passage to $\cF(\op{OSp}(3|2))$. In this case $m=n=1$ and $\delta=1$. By \cite[Lemma 7 (ii)]{GS1}, all blocks are semisimple or equivalent to the principal block $\cB$ (of atypicality $1$) containing the trivial representation. We therefore restrict ourselves to this block. 
The explicit description of this category is not new, but was obtained already by Germoni in \cite[Theorem 2.1.1]{Germoni}. We reproduce the result here using our diagram algebras. 

\subsection{The indecomposable projectives and the algebra}
By Definition~\ref{dominance}, the block $\cB$ contains the simple modules $L^\mg(\la)$ of (with our choice of Borel) highest weight $\la$,  where $\la\in\{\la_a\mid a\geq0\}$ whose elements written in the standard basis are $\la_0=(0\,|\,0)$ and $\la_a=(a\,|\,a-1)$ for $a>0$. We abbreviate the corresponding module by $L(a) $ and let $P(a)$ be its projective cover.  We assign to $P(a)$ via Definitions~\ref{superweight} and~\ref{decoratedcups} the cup diagram $\underline{\la_a}$ as shown in the second line of Figure~\ref{PIMS} (with infinitely many rays to the right), see also Section~\ref{sec:osp32}.

\special{em:linewidth 0.7pt} \unitlength 0.80mm
\linethickness{1pt}
{\intextsep0.3cm
\begin{figure}[!ht]
\begin{array}[t]{c|c|c|c|c|c}
P(0)&P(1)&P(2)&P(3)&P(4)&\cdots\\
\hline
\hline
&&&&&\\
\begin{tikzpicture}[scale=0.5]
\draw (-5,0) .. controls +(0,-.5) and +(0,-.5) .. +(.5,0);
\fill (-4.75,-0.35) circle(3pt);
\draw (-4,0) -- +(0,-.8);
\fill (-4,-.4) circle(3pt);
\draw (-3.5,0) -- +(0,-.8);
\draw (-3,0) -- +(0,-.8);
\draw (-2.5,0) -- +(0,-.8) node[midway,right] {$\cdots$};
\end{tikzpicture}
&
\begin{tikzpicture}[scale=0.5]
\draw (-.5,0) .. controls +(0,-.5) and +(0,-.5) .. +(.5,0);
\draw (.5,0) -- +(0,-.8);
\draw (1,0) -- +(0,-.8);
\draw (1.5,0) -- +(0,-.8);
\draw (2,0) -- +(0,-.8) node[midway,right] {$\cdots$};
\end{tikzpicture}
&
\begin{tikzpicture}[scale=0.5]
\draw (4,0) -- +(0,-.8);
\draw (4.5,0) .. controls +(0,-.5) and +(0,-.5) .. +(.5,0);
\draw (5.5,0) -- +(0,-.8);
\draw (6,0) -- +(0,-.8);
\draw (6.5,0) -- +(0,-.8) node[midway,right] {$\cdots$};
\end{tikzpicture}
&
\begin{tikzpicture}[scale=0.5]
\draw (8.5,0) -- +(0,-.8);
\draw (9,0) -- +(0,-.8);
\draw (9.5,0) .. controls +(0,-.5) and +(0,-.5) .. +(.5,0);
\draw (10.5,0) -- +(0,-.8);
\draw (11,0) -- +(0,-.8) node[midway,right] {$\cdots$};
\end{tikzpicture}
&
\begin{tikzpicture}[scale=0.5]
\draw (13,0) -- +(0,-.8);
\draw (13.5,0) -- +(0,-.8);
\draw (14,0) -- +(0,-.8);
\draw (14.5,0) .. controls +(0,-.5) and +(0,-.5) .. +(.5,0) ;
\draw (15.5,0) -- +(0,-.8);
\node at (16.3,-.42) {$\cdots$};
\end{tikzpicture}
&\cdots \\ \hline
&&&&&\\

\begin{picture}(20.00,8.00)
\special{em:linewidth 0.4pt} \unitlength 0.80mm
\linethickness{1pt}
\put(7.50,10.00){\makebox(0,0)[cc]{\tiny{$0$}}}
\put(7.50,5.00){\makebox(0,0)[cc]{\tiny{$2$}}}
\put(5.00,0){\makebox(0,0)[cc]{\tiny{$0$}}}
\put(10.00,0){\makebox(0,0)[cc]{\tiny{$3$}}}
\linethickness{0.3pt}
\dashline{1}(5.50,11.50)(5.50,8.50)
\dashline{1}(9.50,11.50)(9.50,8.50)
\dashline{1}(5.50,11.50)(9.50,11.50)
\dashline{1}(5.50,8.50)(9.50,8.50)
\dashline{1}(2.50,-2.00)(12.50,-2.00)
\dashline{1}(2.50,2.00)(2.50,-2.00)
\dashline{1}(12.50,2.00)(12.50,-2.00)
\dashline{1}(12.50,2.00)(7.50,7.50)
\dashline{1}(2.50,2.00)(7.50,7.50)
\end{picture}
&
\begin{picture}(20.00,0.00)
\special{em:linewidth 0.4pt} \unitlength 0.80mm
\linethickness{1pt}
\put(7.50,10.00){\makebox(0,0)[cc]{\tiny{$1$}}}
\put(7.50,5.00){\makebox(0,0)[cc]{\tiny{$2$}}}
\put(7.50,0){\makebox(0,0)[cc]{\tiny{$1$}}}
\linethickness{0.3pt}
\dashline{1}(5.50,11.50)(5.50,3.00)
\dashline{1}(9.50,11.50)(9.50,3.00)
\dashline{1}(5.50,11.50)(9.50,11.50)
\dashline{1}(5.50,3.00)(9.50,3.00)
\dashline{1}(5.50,2.00)(5.50,-2.00)
\dashline{1}(9.50,2.00)(9.50,-2.00)
\dashline{1}(5.50,2.00)(9.50,2.00)
\dashline{1}(5.50,-2.00)(9.50,-2.00)
\end{picture}
&
\begin{picture}(20.00,0.00)
\special{em:linewidth 0.4pt} \unitlength 0.80mm
\linethickness{1pt}
\put(7.50,10.00){\makebox(0,0)[cc]{\tiny{$2$}}}
\put(2.50,5.00){\makebox(0,0)[cc]{\tiny{$0$}}}
\put(7.50,5.00){\makebox(0,0)[cc]{\tiny{$1$}}}
\put(12.50,5.00){\makebox(0,0)[cc]{\tiny{$3$}}}
\put(7.50,0){\makebox(0,0)[cc]{\tiny{$2$}}}
\linethickness{0.3pt}
\dashline{1}(5.50,6.50)(5.50,-2.00)
\dashline{1}(9.50,6.50)(9.50,-2.00)
\dashline{1}(5.50,6.50)(9.50,6.50)
\dashline{1}(5.50,-2.00)(9.50,-2.00)
\dashline{1}(0.50,7.50)(0.50,2.50)
\dashline{1}(4.50,7.50)(4.50,2.50)
\dashline{1}(0.50,2.50)(4.50,2.50)
\dashline{1}(10.50,7.50)(10.50,2.50)
\dashline{1}(14.50,7.50)(14.50,2.50)
\dashline{1}(10.50,2.50)(14.50,2.50)
\dashline{1}(5.50,11.50)(9.50,11.50)
\dashline{1}(4.50,7.50)(10.50,7.50)
\dashline{1}(0.50,7.50)(5.50,11.50)
\dashline{1}(14.50,7.50)(9.50,11.50)
\end{picture}&
\begin{picture}(20.00,8.00)
\special{em:linewidth 0.4pt} \unitlength 0.80mm
\linethickness{1pt}

\put(7.50,10.00){\makebox(0,0)[cc]{\tiny{$3$}}}
\put(5.50,5.00){\makebox(0,0)[cc]{\tiny{$2$}}}
\put(12.50,5.00){\makebox(0,0)[cc]{\tiny{$4$}}}
\put(3.00,0){\makebox(0,0)[cc]{\tiny{$3$}}}
\put(8.00,0){\makebox(0,0)[cc]{\tiny{$0$}}}
\linethickness{0.3pt}
\dashline{1}(5.50,11.50)(5.50,8.50) 
\dashline{1}(5.50,11.50)(9.50,11.50) 
\dashline{1}(5.50,8.50)(10.50,3.00)  %
\dashline{1}(9.50,11.50)(14.50,6.00) %
\dashline{1}(11.50,3.00)(14.50,3.00) %
\dashline{1}(14.50,3.00)(14.50,6.00)

\dashline{1}(0.50,-2.00)(10.50,-2.00)
\dashline{1}(0.50,2.00)(0.50,-2.00)
\dashline{1}(10.50,2.00)(10.50,-2.00)
\dashline{1}(10.50,2.00)(5.50,7.50)
\dashline{1}(0.50,2.00)(5.50,7.50)
\end{picture}
&
\begin{picture}(20.00,8.00)
\special{em:linewidth 0.4pt} \unitlength 0.80mm
\linethickness{1pt}
\put(7.50,10.00){\makebox(0,0)[cc]{\tiny{$4$}}}
\put(2.50,5.00){\makebox(0,0)[cc]{\tiny{$3$}}}
\put(12.50,5.00){\makebox(0,0)[cc]{\tiny{$5$}}}
\put(7.50,0){\makebox(0,0)[cc]{\tiny{$4$}}}
\linethickness{0.3pt}
\dashline{1}(5.50,11.50)(5.50,8.50) 
\dashline{1}(5.50,11.50)(9.50,11.50) 
\dashline{1}(5.50,8.50)(10.50,3.00)  %
\dashline{1}(9.50,11.50)(14.50,6.00) %
\dashline{1}(11.50,3.00)(14.50,3.00) %
\dashline{1}(14.50,3.00)(14.50,6.00)
%
\dashline{1}(0.50,6.50)(0.50,3.50) 
\dashline{1}(0.50,6.50)(4.50,6.50) 
\dashline{1}(0.50,3.50)(5.50,-2.00)  %
\dashline{1}(4.50,6.50)(9.50,1.00) %
\dashline{1}(6.50,-2.00)(9.50,-2.00) %
\dashline{1}(9.50,-2.00)(9.50,1.00)
\end{picture}
&\cdots\\
&&&&&\\ \hline
&&&&&\\

\begin{picture}(20.00,8.00)
\special{em:linewidth 0.4pt} \unitlength 0.80mm
\linethickness{1pt}
\put(7.50,10.00){\makebox(0,0)[cc]{\tiny{$0$}}}
\put(7.50,5.00){\makebox(0,0)[cc]{\tiny{$2$}}}
\put(7.50,0){\makebox(0,0)[cc]{\tiny{$0$}}}
\end{picture}
&
\begin{picture}(20.00,0.00)
\special{em:linewidth 0.4pt} \unitlength 0.80mm
\linethickness{1pt}
\put(7.50,10.00){\makebox(0,0)[cc]{\tiny{$1$}}}
\put(7.50,5.00){\makebox(0,0)[cc]{\tiny{$2$}}}
\put(7.50,0){\makebox(0,0)[cc]{\tiny{$1$}}}
\linethickness{0.3pt}
\end{picture}
&
\begin{picture}(20.00,0.00)
\special{em:linewidth 0.4pt} \unitlength 0.80mm
\linethickness{1pt}
\put(7.50,10.00){\makebox(0,0)[cc]{\tiny{$2$}}}
\put(2.50,5.00){\makebox(0,0)[cc]{\tiny{$0$}}}
\put(7.50,5.00){\makebox(0,0)[cc]{\tiny{$1$}}}
\put(12.50,5.00){\makebox(0,0)[cc]{\tiny{$3$}}}
\put(7.50,0){\makebox(0,0)[cc]{\tiny{$2$}}}
\end{picture}&
\begin{picture}(20.00,8.00)
\special{em:linewidth 0.4pt} \unitlength 0.80mm
\linethickness{1pt}
\put(7.50,10.00){\makebox(0,0)[cc]{\tiny{$3$}}}
\put(2.50,5.00){\makebox(0,0)[cc]{\tiny{$2$}}}
\put(12.50,5.00){\makebox(0,0)[cc]{\tiny{$4$}}}
\put(7.50,0){\makebox(0,0)[cc]{\tiny{$3$}}}
\linethickness{0.3pt}
\end{picture}
&
\begin{picture}(20.00,8.00)
\special{em:linewidth 0.4pt} \unitlength 0.80mm
\linethickness{1pt}
\put(7.50,10.00){\makebox(0,0)[cc]{\tiny{$4$}}}
\put(2.50,5.00){\makebox(0,0)[cc]{\tiny{$3$}}}
\put(12.50,5.00){\makebox(0,0)[cc]{\tiny{$5$}}}
\put(7.50,0){\makebox(0,0)[cc]{\tiny{$4$}}}
\end{picture}
&\cdots \\
&&&&\\
\hline
\end{array}
\abovecaptionskip0.25cm
\caption{Indecomposable projectives in $\mathbbm{1}_\cB \mathbb{D}_\Lambda \mathbbm{1}_\cB $  versus $\mathbbm{1}_\cB \mathbb{D}_\Lambda \mathbbm{1}_\cB/\mathbb{I}.$ }
\label{PIMS}
\end{figure}}

The oriented circle diagrams built from the given cup diagrams are displayed in Figure~\ref{basis23}. They are obtained by putting one of the cup diagrams upside down on top of another one, and then orienting the result in the sense of Figure~\ref{oriented}.

{\intextsep0.2cm
\begin{figure}[!h]
\begin{eqnarray*}
\arraycolsep=.5pt
\begin{array}[t]{c|ccccc}
\op{deg} = 0\;&
\begin{tikzpicture}[anchorbase,thick,scale=0.32]
\node at (0,0) {$\scriptstyle \up$};
\node at (1,0) {$\scriptstyle \up$};
\node at (2,0) {$\scriptstyle \up$};
\node at (3,0) {$\scriptstyle \down$};
\node at (4,0) {$\scriptstyle \down$};
\node at (5,0) {$\scriptstyle \cdots$};
\draw (0,0) .. controls +(0,-1) and +(0,-1) .. +(1,0);
\fill (0.5,-.75) circle(4pt);
\draw (2,0) -- +(0,-1.5);
\fill (2,-.75) circle(4pt);
\draw (3,0) -- +(0,-1.5);
\draw (4,0) -- +(0,-1.5);

\draw (0,0) .. controls +(0,1) and +(0,1) .. +(1,0);
\fill (0.5,.75) circle(4pt);
\draw (2,0) -- +(0,1.5);
\fill (2,.75) circle(4pt);
\draw (3,0) -- +(0,1.5);
\draw (4,0) -- +(0,1.5);

\node at (2.5,-2.25) {$\mathbbm{1}_0$};
\end{tikzpicture}
&
\begin{tikzpicture}[anchorbase,thick,scale=0.32]
\node at (0,0) {$\scriptstyle \down$};
\node at (1,0) {$\scriptstyle \up$};
\node at (2,0) {$\scriptstyle \down$};
\node at (3,0) {$\scriptstyle \down$};
\node at (4,0) {$\scriptstyle \down$};
\node at (5,0) {$\scriptstyle \cdots$};
\draw (0,0) .. controls +(0,-1) and +(0,-1) .. +(1,0);
\draw (2,0) -- +(0,-1.5);
\draw (3,0) -- +(0,-1.5);
\draw (4,0) -- +(0,-1.5);

\draw (0,0) .. controls +(0,1) and +(0,1) .. +(1,0);
\draw (2,0) -- +(0,1.5);
\draw (3,0) -- +(0,1.5);
\draw (4,0) -- +(0,1.5);

\node at (2.5,-2.25) {$\mathbbm{1}_1$};
\end{tikzpicture}
&
\begin{tikzpicture}[anchorbase,thick,scale=0.32]
\node at (0,0) {$\scriptstyle \down$};
\node at (1,0) {$\scriptstyle \down$};
\node at (2,0) {$\scriptstyle \up$};
\node at (3,0) {$\scriptstyle \down$};
\node at (4,0) {$\scriptstyle \down$};
\node at (5,0) {$\scriptstyle \cdots$};
\draw (0,0) -- +(0,-1.5);
\draw (1,0) .. controls +(0,-1) and +(0,-1) .. +(1,0);
\draw (3,0) -- +(0,-1.5);
\draw (4,0) -- +(0,-1.5);

\draw (0,0) -- +(0,1.5);
\draw (1,0) .. controls +(0,1) and +(0,1) .. +(1,0);
\draw (3,0) -- +(0,1.5);
\draw (4,0) -- +(0,1.5);

\node at (2.5,-2.25) {$\mathbbm{1}_2$};
\end{tikzpicture}
&
\begin{tikzpicture}[anchorbase,thick,scale=0.32]
\node at (0,0) {$\scriptstyle \down$};
\node at (1,0) {$\scriptstyle \down$};
\node at (2,0) {$\scriptstyle \down$};
\node at (3,0) {$\scriptstyle \up$};
\node at (4,0) {$\scriptstyle \down$};
\node at (5,0) {$\scriptstyle \cdots$};
\draw (0,0) -- +(0,-1.5);
\draw (1,0) -- +(0,-1.5);
\draw (2,0) .. controls +(0,-1) and +(0,-1) .. +(1,0);
\draw (4,0) -- +(0,-1.5);

\draw (0,0) -- +(0,1.5);
\draw (1,0) -- +(0,1.5);
\draw (2,0) .. controls +(0,1) and +(0,1) .. +(1,0);
\draw (4,0) -- +(0,1.5);

\node at (2.5,-2.25) {$\mathbbm{1}_3$};
\end{tikzpicture}
&\\ \hline
&&&&&\\
\op{deg} = 1\;&
\begin{tikzpicture}[anchorbase,thick,scale=0.32]
\node at (0,0) {$\scriptstyle \down$};
\node at (1,0) {$\scriptstyle \down$};
\node at (2,0) {$\scriptstyle \up$};
\node at (3,0) {$\scriptstyle \down$};
\node at (4,0) {$\scriptstyle \down$};
\node at (5,0) {$\scriptstyle \cdots$};
\draw (0,0) .. controls +(0,-1) and +(0,-1) .. +(1,0);
\fill (0.5,-.75) circle(4pt);
\draw (2,0) -- +(0,-1.5);
\fill (2,-.75) circle(4pt);
\draw (3,0) -- +(0,-1.5);
\draw (4,0) -- +(0,-1.5);

\draw (0,0) -- +(0,1.5);
\draw (1,0) .. controls +(0,1) and +(0,1) .. +(1,0);
\draw (3,0) -- +(0,1.5);
\draw (4,0) -- +(0,1.5);

\node at (2.5,-2.25) {$f_0$};
\end{tikzpicture}
&
\begin{tikzpicture}[anchorbase,thick,scale=0.32]
\node at (0,0) {$\scriptstyle \down$};
\node at (1,0) {$\scriptstyle \up$};
\node at (2,0) {$\scriptstyle \down$};
\node at (3,0) {$\scriptstyle \down$};
\node at (4,0) {$\scriptstyle \down$};
\node at (5,0) {$\scriptstyle \cdots$};
\draw (0,0) .. controls +(0,-1) and +(0,-1) .. +(1,0);
\draw (2,0) -- +(0,-1.5);
\draw (3,0) -- +(0,-1.5);
\draw (4,0) -- +(0,-1.5);

\draw (0,0) -- +(0,1.5);
\draw (1,0) .. controls +(0,1) and +(0,1) .. +(1,0);
\draw (3,0) -- +(0,1.5);
\draw (4,0) -- +(0,1.5);

\node at (2.5,-2.25) {$f_1$};
\end{tikzpicture}
&
\begin{tikzpicture}[anchorbase,thick,scale=0.32]
\node at (0,0) {$\scriptstyle \down$};
\node at (1,0) {$\scriptstyle \down$};
\node at (2,0) {$\scriptstyle \up$};
\node at (3,0) {$\scriptstyle \down$};
\node at (4,0) {$\scriptstyle \down$};
\node at (5,0) {$\scriptstyle \cdots$};
\draw (0,0) -- +(0,-1.5);
\draw (1,0) .. controls +(0,-1) and +(0,-1) .. +(1,0);
\draw (3,0) -- +(0,-1.5);
\draw (4,0) -- +(0,-1.5);

\draw (0,0) -- +(0,1.5);
\draw (1,0) -- +(0,1.5);
\draw (2,0) .. controls +(0,1) and +(0,1) .. +(1,0);
\draw (4,0) -- +(0,1.5);

\node at (2.5,-2.25) {$f_2$};
\end{tikzpicture}
&
\begin{tikzpicture}[anchorbase,thick,scale=0.32]
\node at (0,0) {$\scriptstyle \down$};
\node at (1,0) {$\scriptstyle \down$};
\node at (2,0) {$\scriptstyle \down$};
\node at (3,0) {$\scriptstyle \up$};
\node at (4,0) {$\scriptstyle \down$};
\node at (5,0) {$\scriptstyle \cdots$};
\draw (0,0) -- +(0,-1.5);
\draw (1,0) -- +(0,-1.5);
\draw (2,0) .. controls +(0,-1) and +(0,-1) .. +(1,0);
\draw (4,0) -- +(0,-1.5);

\draw (0,0) -- +(0,1.5);
\draw (1,0) -- +(0,1.5);
\draw (2,0) -- +(0,1.5);
\draw (3,0) .. controls +(0,1) and +(0,1) .. +(1,0);

\node at (2.5,-2.25) {$f_3$};
\end{tikzpicture}\\
&&&&&\\
&
\begin{tikzpicture}[anchorbase,thick,scale=0.32]
\node at (0,0) {$\scriptstyle \down$};
\node at (1,0) {$\scriptstyle \down$};
\node at (2,0) {$\scriptstyle \up$};
\node at (3,0) {$\scriptstyle \down$};
\node at (4,0) {$\scriptstyle \down$};
\node at (5,0) {$\scriptstyle \cdots$};
\draw (0,0) .. controls +(0,1) and +(0,1) .. +(1,0);
\fill (0.5,.75) circle(4pt);
\draw (2,0) -- +(0,1.5);
\fill (2,.75) circle(4pt);
\draw (3,0) -- +(0,1.5);
\draw (4,0) -- +(0,1.5);

\draw (0,0) -- +(0,-1.5);
\draw (1,0) .. controls +(0,-1) and +(0,-1) .. +(1,0);
\draw (3,0) -- +(0,-1.5);
\draw (4,0) -- +(0,-1.5);

\node at (2.5,-2.25) {$g_0$};
\end{tikzpicture}
&
\begin{tikzpicture}[anchorbase,thick,scale=0.32]
\node at (0,0) {$\scriptstyle \down$};
\node at (1,0) {$\scriptstyle \up$};
\node at (2,0) {$\scriptstyle \down$};
\node at (3,0) {$\scriptstyle \down$};
\node at (4,0) {$\scriptstyle \down$};
\node at (5,0) {$\scriptstyle \cdots$};
\draw (0,0) .. controls +(0,1) and +(0,1) .. +(1,0);
\draw (2,0) -- +(0,1.5);
\draw (3,0) -- +(0,1.5);
\draw (4,0) -- +(0,1.5);

\draw (0,0) -- +(0,-1.5);
\draw (1,0) .. controls +(0,-1) and +(0,-1) .. +(1,0);
\draw (3,0) -- +(0,-1.5);
\draw (4,0) -- +(0,-1.5);

\node at (2.5,-2.25) {$g_1$};
\end{tikzpicture}
&
\begin{tikzpicture}[anchorbase,thick,scale=0.32]
\node at (0,0) {$\scriptstyle \down$};
\node at (1,0) {$\scriptstyle \down$};
\node at (2,0) {$\scriptstyle \up$};
\node at (3,0) {$\scriptstyle \down$};
\node at (4,0) {$\scriptstyle \down$};
\node at (5,0) {$\scriptstyle \cdots$};
\draw (0,0) -- +(0,1.5);
\draw (1,0) .. controls +(0,1) and +(0,1) .. +(1,0);
\draw (3,0) -- +(0,1.5);
\draw (4,0) -- +(0,1.5);

\draw (0,0) -- +(0,-1.5);
\draw (1,0) -- +(0,-1.5);
\draw (2,0) .. controls +(0,-1) and +(0,-1) .. +(1,0);
\draw (4,0) -- +(0,-1.5);

\node at (2.5,-2.25) {$g_2$};
\end{tikzpicture}
&
\begin{tikzpicture}[anchorbase,thick,scale=0.32]
\node at (0,0) {$\scriptstyle \down$};
\node at (1,0) {$\scriptstyle \down$};
\node at (2,0) {$\scriptstyle \down$};
\node at (3,0) {$\scriptstyle \up$};
\node at (4,0) {$\scriptstyle \down$};
\node at (5,0) {$\scriptstyle \cdots$};
\draw (0,0) -- +(0,1.5);
\draw (1,0) -- +(0,1.5);
\draw (2,0) .. controls +(0,1) and +(0,1) .. +(1,0);
\draw (4,0) -- +(0,1.5);

\draw (0,0) -- +(0,-1.5);
\draw (1,0) -- +(0,-1.5);
\draw (2,0) -- +(0,-1.5);
\draw (3,0) .. controls +(0,-1) and +(0,-1) .. +(1,0);

\node at (2.5,-2.25) {$g_3$};
\end{tikzpicture}\\
\hline
&&&&&\\
\op{deg} = 2\;&\begin{tikzpicture}[anchorbase,thick,scale=0.32]
\node at (0,0) {$\scriptstyle \down$};
\node at (1,0) {$\scriptstyle \down$};
\node at (2,0) {$\scriptstyle \up$};
\node at (3,0) {$\scriptstyle \down$};
\node at (4,0) {$\scriptstyle \down$};
\node at (5,0) {$\scriptstyle \cdots$};
\draw (0,0) .. controls +(0,-1) and +(0,-1) .. +(1,0);
\fill (0.5,-.75) circle(4pt);
\draw (2,0) -- +(0,-1.5);
\fill (2,-.75) circle(4pt);
\draw (3,0) -- +(0,-1.5);
\draw (4,0) -- +(0,-1.5);

\draw (0,0) .. controls +(0,1) and +(0,1) .. +(1,0);
\fill (0.5,.75) circle(4pt);
\draw (2,0) -- +(0,1.5);
\fill (2,.75) circle(4pt);
\draw (3,0) -- +(0,1.5);
\draw (4,0) -- +(0,1.5);

\node at (2.5,-2.25) {$g_0 \circ f_0$};
\end{tikzpicture}
&
\begin{tikzpicture}[anchorbase,thick,scale=0.32]
\node at (0,0) {$\scriptstyle \up$};
\node at (1,0) {$\scriptstyle \down$};
\node at (2,0) {$\scriptstyle \down$};
\node at (3,0) {$\scriptstyle \down$};
\node at (4,0) {$\scriptstyle \down$};
\node at (5,0) {$\scriptstyle \cdots$};
\draw (0,0) .. controls +(0,-1) and +(0,-1) .. +(1,0);
\draw (2,0) -- +(0,-1.5);
\draw (3,0) -- +(0,-1.5);
\draw (4,0) -- +(0,-1.5);

\draw (0,0) .. controls +(0,1) and +(0,1) .. +(1,0);
\draw (2,0) -- +(0,1.5);
\draw (3,0) -- +(0,1.5);
\draw (4,0) -- +(0,1.5);

\node at (2.5,-2.25) {$g_1 \circ f_1$};
\end{tikzpicture}
&
\begin{tikzpicture}[anchorbase,thick,scale=0.32]
\node at (0,0) {$\scriptstyle \down$};
\node at (1,0) {$\scriptstyle \up$};
\node at (2,0) {$\scriptstyle \down$};
\node at (3,0) {$\scriptstyle \down$};
\node at (4,0) {$\scriptstyle \down$};
\node at (5,0) {$\scriptstyle \cdots$};
\draw (0,0) -- +(0,-1.5);
\draw (1,0) .. controls +(0,-1) and +(0,-1) .. +(1,0);
\draw (3,0) -- +(0,-1.5);
\draw (4,0) -- +(0,-1.5);

\draw (0,0) -- +(0,1.5);
\draw (1,0) .. controls +(0,1) and +(0,1) .. +(1,0);
\draw (3,0) -- +(0,1.5);
\draw (4,0) -- +(0,1.5);

\node at (2.5,-2.25) {$\scriptstyle f_0\circ g_0= g_2 \circ f_2 = f_1 \circ g_1$};
\end{tikzpicture}
&
\begin{tikzpicture}[anchorbase,thick,scale=0.32]
\node at (0,0) {$\scriptstyle \down$};
\node at (1,0) {$\scriptstyle \down$};
\node at (2,0) {$\scriptstyle \up$};
\node at (3,0) {$\scriptstyle \down$};
\node at (4,0) {$\scriptstyle \down$};
\node at (5,0) {$\scriptstyle \cdots$};
\draw (0,0) -- +(0,-1.5);
\draw (1,0) -- +(0,-1.5);
\draw (2,0) .. controls +(0,-1) and +(0,-1) .. +(1,0);
\draw (4,0) -- +(0,-1.5);

\draw (0,0) -- +(0,1.5);
\draw (1,0) -- +(0,1.5);
\draw (2,0) .. controls +(0,1) and +(0,1) .. +(1,0);
\draw (4,0) -- +(0,1.5);

\node at (2.5,-2.25) {$g_3 \circ f_3 = f_2 \circ g_2$};
\end{tikzpicture}\\
&&&&&\vspace{-2mm}\\
&
\fbox{\begin{tikzpicture}[anchorbase,thick,scale=0.32]
\node at (0,0) {$\scriptstyle \down$};
\node at (1,0) {$\scriptstyle \down$};
\node at (2,0) {$\scriptstyle \up$};
\node at (3,0) {$\scriptstyle \down$};
\node at (4,0) {$\scriptstyle \down$};
\node at (5,0) {$\scriptstyle \cdots$};
\draw (0,0) .. controls +(0,-1) and +(0,-1) .. +(1,0);
\fill (0.5,-.75) circle(4pt);
\draw (2,0) -- +(0,-1.5);
\fill (2,-.75) circle(4pt);
\draw (3,0) -- +(0,-1.5);
\draw (4,0) -- +(0,-1.5);

\draw (0,0) -- +(0,1.5);
\draw (1,0) -- +(0,1.5);
\draw (2,0) .. controls +(0,1) and +(0,1) .. +(1,0);
\draw (4,0) -- +(0,1.5);

\node at (2.5,-2.25) {$f_2 \circ f_0$};
\end{tikzpicture}}
&
\fbox{\begin{tikzpicture}[anchorbase,thick,scale=0.32]
\node at (0,0) {$\scriptstyle \down$};
\node at (1,0) {$\scriptstyle \down$};
\node at (2,0) {$\scriptstyle \up$};
\node at (3,0) {$\scriptstyle \down$};
\node at (4,0) {$\scriptstyle \down$};
\node at (5,0) {$\scriptstyle \cdots$};
\draw (0,0) .. controls +(0,1) and +(0,1) .. +(1,0);
\fill (0.5,.75) circle(4pt);
\draw (2,0) -- +(0,1.5);
\fill (2,.75) circle(4pt);
\draw (3,0) -- +(0,1.5);
\draw (4,0) -- +(0,1.5);

\draw (0,0) -- +(0,-1.5);
\draw (1,0) -- +(0,-1.5);
\draw (2,0) .. controls +(0,-1) and +(0,-1) .. +(1,0);
\draw (4,0) -- +(0,-1.5);

\node at (2.5,-2.25) {$g_0\circ g_2$};
\end{tikzpicture}}
\end{array}
\end{eqnarray*}
\abovecaptionskip0cm
\belowcaptionskip0cm
\caption{The homogeneous basis vectors of $\mathbbm{1}_\cB\mathbb{D}_\Lambda \mathbbm{1}_\cB$ (with their degrees).}
\label{basis23}
\end{figure}
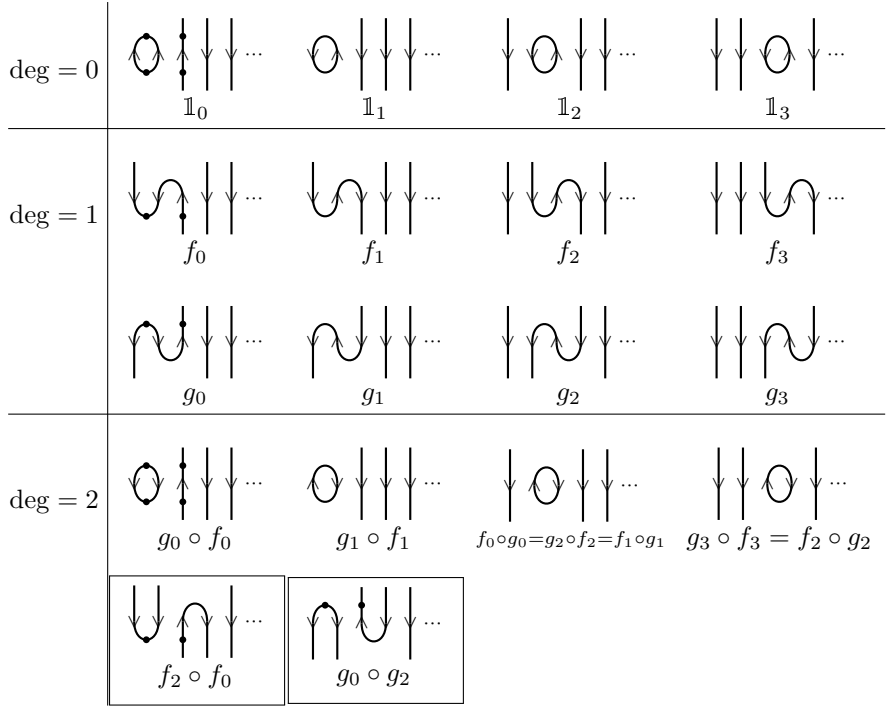}

These diagrams form the basis of the algebra  $\mathbbm{1}_\cB \mathbb{D}_\Lambda \mathbbm{1}_\cB$. The multiplication is given by the rules from \cite[Section 6.2]{ESperv}. The last two (framed) oriented circle diagrams in Figure~\ref{basis23} are exactly those which contain at least one non-propagating line. They span the nuclear ideal $\mathbb{I}$, in $\mathbbm{1}_\cB \mathbb{D}_\Lambda \mathbbm{1}_\cB$, see Proposition~\ref{isideal}. 

The algebra  $\mathbbm{1}_\cB\mathbb{D}_\Lambda \mathbbm{1}_\cB$  can be equipped with a positive $\mathbb{Z}$-grading, via Figure~\ref{oriented}, such that the basis vectors are homogeneous of degree as displayed in Figure~\ref{basis23} with homogeneous ideal $\mathbb{I}$. It descends to a grading on  $\mathbbm{1}_\cB\mathbb{D}_\Lambda \mathbbm{1}_\cB/\mathbb{I}$, and hence gives a grading on the category $\cB$ via Theorem~\ref{main}.

\subsection{The block \texorpdfstring{$\cB$}{B} in terms of a quiver with relations}
From the definition of the multiplication, see \cite[Theorem~6.2]{ESperv}, we directly deduce, using Theorem~\ref{main}, an explicit description of the locally finite endomorphism ring $\op{End}^{\fin}_\cF(P)$:

 \begin{theorem}
 \label{exs}
 The algebra $\mathbbm{1}_\cB \mathbb{D}_\Lambda \mathbbm{1}_\cB /\mathbb{I}$ is isomorphic (as a graded algebra) to the path algebra of the following infinite quiver (with grading given by length of paths)
\small
\abovedisplayskip0.25em
\belowdisplayskip0.25em
\begin{eqnarray}
\label{quiver}
\begin{xy}
  \xymatrix{
  0\ar@/^/[dr]^{f_0}\\
  &2\ar@/^/[ul]^{g_0}\ar@/^/[dl]^{g_1}\ar@/^/[r]^{f_2}
  &3\ar@/^/[l]^{g_2}\ar@/^/[r]^{f_3}
  &4\ar@/^/[l]^{g_3}\ar@/^/[r]^{f_4}
  &5\ar@/^/[l]^{g_4}\cdots\\
 1\ar@/^/[ur]^{f_1}}
\end{xy}
\end{eqnarray}
\normalsize
modulo the (homogeneous) ideal generated by (the homogeneous relations) 
\abovedisplayskip0.25em
\belowdisplayskip0.25em
\[
f_{i+1}\circ f_i=0=g_{i}\circ g_{i+1}\, ,\quad g_{i+1}\circ f_{i+1}=f_{i}\circ g_{i} \quad \text{for }i\geq 0
\]
\abovedisplayskip0.25em
\belowdisplayskip0.25em
\[
g_0\circ f_1=0=g_1\circ f_0\, ,\quad f_0\circ g_0=g_1\circ f_1=g_2\circ f_2 \, , \quad f_2\circ f_0=0=g_0\circ g_2.
\]
Here, the last two relations are the relations from $\mathbb{I}$. In particular, the category of finite-dimensional modules of this algebra is equivalent to the principal block $\cB$ of  $\cF(\op{SOSp}(3|2))$. 
\end{theorem}

The structure of the indecomposable projective modules for $\mathbbm{1}_\cB \mathbb{D}_\Lambda \mathbbm{1}_\cB$ is displayed in the third line of Figure~\ref{PIMS}, where each number stands for the corresponding simple module. The height where the number of a simple module occurs, indicates the degree it is concentrated in, when we consider it as a module for the graded algebra. We displayed the grading filtration which in this case however agrees with the radical and the socle filtration. In comparison, the fourth line shows the structure of the indecomposable projective modules for $\mathbbm{1}_\cB\mathbb{D}_\Lambda \mathbbm{1}_\cB/\mathbb{I}$ . 

The description of $\cF(\op{SOSp}(3|2))$  in Theorem~\ref{exs} reproduces Germoni's result, \cite{Germoni}. The algebra $\mathbbm{1}_\cB \mathbb{D}_\Lambda \mathbbm{1}_\cB/\mathbb{I}$  in this example also occurs under the name {\it zigzag algebra} (of type ${\rm D}_\infty$) in the literature,  see e.g. \cite[2.3]{CautisLicata}. In contrast to the general case of  $\cF(\OSPrn)$, it is tame representation type as shown in \cite{Germoni}. 

\subsection{The failure of quasi-hereditarity and cellularity}
The algebra $\mathbb{D}_\Lambda$ is quasi-hereditary by \cite[Section~6]{ESperv}, and so $\mathbbm{1}_\cB \mathbb{D}_\Lambda \mathbbm{1}_\cB$  is a cellular algebra, \cite[Proposition 4.3]{KoenigXi}. Hence we have cell modules  $\Delta(\la)=\Delta_{\mathbbm{1}_\cB\mathbb{D}_\Lambda \mathbbm{1}_\cB}(\la)$, indexed by some labelling set  (in fact certain weights $\la\in\Lambda$, but we ignore this here). We indicate  in Figure~\ref{PIMS} (by grouping the composition factors) these cell modules. Note that there are two cell modules with simple head labelled by $1$, since the truncation of our quasi-hereditary algebra $\mathbb{D}_\Lambda$ is not compatible with the quasi-hereditary ordering. Hence although $\mathbb{D}_\Lambda $ is quasi-hereditary, the truncation $\mathbbm{1}_\cB \mathbb{D}_\Lambda \mathbbm{1}_\cB$  is only cellular. Factoring out the ideal $\mathbb{I}$ of nuclear morphisms means we kill some of the simple composition factors. The result is displayed in the last line of Figure~\ref{PIMS}. One can also see there for instance that the cell module $\Delta(2)$ gives rise to a different subquotient in $P(0)$ than in $P(3)$, namely in the notation from Section~\ref{shadow} we have
\begin{eqnarray}
\Delta(2)={{\scriptstyle \begin{array}{cccc}&2&\\0&&3\end{array}}}
&\rightsquigarrow&
\Delta(0,2)\;=\;\begin{array}{c}2\\0\end{array}\quad\text{and}\quad 
\Delta(3,2)\;=\;\begin{array}{c}2\\3\end{array}
\end{eqnarray}
We leave it  to the reader to show that this algebra $\mathbbm{1}_\cB\mathbb{D}_\Lambda \mathbbm{1}_\cB/\mathbb{I}$, is not cellular.

\subsection{The \texorpdfstring{$0$}{0}-Calabi-Yau property}
We observe that the resulting projective modules for $\mathbbm{1}_\cB \mathbb{D}_\Lambda \mathbbm{1}_\cB/\mathbb{I}$ are self-dual and they are in fact the maximal self-dual quotients of the indecomposable projective modules for $\mathbbm{1}_\cB \mathbb{D}_\Lambda \mathbbm{1}_\cB$. Hence the projective modules become injective, a property which is well-known to hold in $\cB$, see \cite{BKN}. More conceptually let $\tilde{\tau}:\mathbbm{1}_\cB \mathbb{D}_\Lambda \mathbbm{1}_\cB\rightarrow \mathbb{C}$ be the linear (trace) map defined on basis vectors $b$ from Figure~\ref{basis23} by 
\begin{eqnarray*}
\tilde{\tau}(b)&=&
\begin{cases}
1&\text{\begin{minipage}[t]{9cm}if $b$ is of the form $\underline{\la}\nu\overline{\la}$ (i.e. it has reflection symmetry in the horizontal reflection line), and $\op{deg}(b)=2,$\end{minipage}}\\
0&\text{otherwise},
\end{cases}
\end{eqnarray*}
and consider the corresponding bilinear map $\tau$ defined on basis vectors as 
\begin{eqnarray}
\tau:\quad\mathbbm{1}_\cB \mathbb{D}_\Lambda \mathbbm{1}_\cB\times \mathbbm{1}_\cB \mathbb{D}_\Lambda \mathbbm{1}_\cB\rightarrow\mathbb{C},&&
\tau(b_1,b_2)=\tilde{\tau}(b_1b_2),
\end{eqnarray}
This is by definition a symmetric form, which is however degenerate with radical $\op{rad}(\tau)$ spanned by the nuclear morphisms $f_2\circ f_0$, $g_0\circ g_2$. In particular $\mathbbm{1}_\cB \mathbb{D}_\Lambda \mathbbm{1}_\cB/ \mathbb{I}= \mathbbm{1}_\cB \mathbb{D}_\Lambda \mathbbm{1}_\cB/\op{rad}(\tau)$ is a noncommutative symmetric Frobenius algebra. 
 \begin{center}
{\it The block $\cB$ is the maximal $0$-Calabi-Yau quotient (with respect to $\tau$) of the category of finite-dimensional $\mathbbm{1}_\cB \mathbb{D}_\Lambda \mathbbm{1}_\cB$-modules. }
\end{center}
A corresponding characterisation holds for arbitrary blocks and arbitrary $m,n$ and will be studied in detail in a subsequent paper. 
\subsection{Koszulity}\label{Koszulity}
By constructing an explicit (infinite) linear projective resolutions for each simple module one can check in this special example, that the algebra $\mathbbm{1}_\cB \mathbb{D}_\Lambda \mathbbm{1}_\cB/\mathbb{I}$  here is a locally finite Koszul algebra in the sense of \cite{MOS}. Hence Conjecture~\ref{conj:koszul} holds in this case.
\renc{\thetheorem}{\arabic{section}.\arabic{theorem}}
\renewcommand\thesection{\arabic{section}}
\setcounter{section}{0}
\section{The orthosymplectic supergroup and its Lie algebra}
\label{super}
For the general theory of Lie superalgebras we refer to \cite{Musson}. 
\subsection{Lie superalgebras}
By a {\it (vector) superspace} we always mean a finite-dimensional $\mZ/2\mZ$ -graded vector space $V = V_0 \oplus V_1$. For  any homogeneous element $v\in V$ we denote by $| v | \in \{0,1 \}$ its parity.  The integer $\op{dim}V_0-\op{dim}V_1$ is called the {\it supertrace} of $V$, and the tuple ${\rm sdim}V = (\op{dim}V_0\,|\,\op{dim}V_1)$ the {\it superdimension} of $V$. Given a superspace $V$ let $\mathfrak{gl}(V)$ be the corresponding {\it general Lie superalgebra}, i.e. the superspace ${\rm End}_\mathbb{C}(V)$ of all endomorphism  with the superbracket defined on homogeneous elements by
\begin{eqnarray}
\label{bracket}
[X,Y] = X \circ Y - (-1)^{|X| \cdot |Y|}Y \circ X .
\end{eqnarray}
If $V$ has superdimension $a\,|\,b$ then $\mathfrak{gl}(V)$ is also denoted by $\mathfrak{gl}(a\,|\,b)$. It can be realized as the space of $(a+b)\times (a+b)$-matrices viewed as superspace with the matrix units on the block diagonals being even, and the other matrix units being odd elements, and the bracket given by the supercommutator \eqref{bracket}.

We fix now  $r,n\in\mathbb{Z}_{\geq0}$ and a superspace $V = V_0 \oplus V_1$ of superdimension $r\,|\,2n$ equipped with a non-degenerate supersymmetric bilinear form $\left\langle - , - \right\rangle$, i.e. a bilinear form $V\times V\rightarrow \mC$ which is symmetric when restricted to $V_0 \times V_0$, skew-symmetric on $V_1 \times V_1$ and zero on mixed products. From now on we fix also $m\in\mathbb{Z}_{\geq0}$ such that $r=2m$ or $r=2m+1$. We denote by $\de=r-2n$, the supertrace of the natural representation.

\begin{definition}
The \emph{orthosymplectic Lie superalgebra} $\mg=\mathfrak{osp}(V)$ is the Lie supersubalgebra of $\mathfrak{gl}(V)$ consisting of all endomorphisms which respect a fixed supersymmetric bilinear form. Explicitly, a homogeneous element $X \in \mathfrak{osp}(V)$ has to satisfy for any homogeneous $v\in V$
\begin{equation}
\label{oh}
\left\langle Xv,w \right\rangle + (-1)^{|X| \cdot |v|} \left\langle v,Xw\right\rangle = 0.
\end{equation}
\end{definition}
In case one prefers a concrete realization in terms of endomorphism of a superspace one could choose a homogeneous basis $v_i$, $1 \leq i \leq r+2n$, of $V$ and consider the supersymmetric bilinear form given by the (skew)symmetric matrices 
\begin{eqnarray*}
J^{\mathrm{sym}}=\begin{pmatrix}
1&0&0\\
0&0&\mathbf{1}_m\\
0&\mathbf{1}_m&0
\end{pmatrix}
&\quad\text{and}\quad&
J^{\mathrm{skew}}=\begin{pmatrix}
0&\mathbf{1}_n\\
-\mathbf{1}_n&0
\end{pmatrix}
\end{eqnarray*}
where $\mathbf{1}_k$ denotes the respective identity matrix and $r$ is equal to $2m+1$ or equal to $2m$, in the latter case the first column and row of $J^{\mathrm{sym}}$ are removed. Then $\mg$ can be realized as the Lie super subalgebra of matrices \tiny $\begin{pmatrix}
A&B\\C&D
\end{pmatrix}$
\normalsize
in $\mathfrak{gl}(r|2n)$ 
where 
$$A^tJ^{\mathrm{sym}}+J^{\mathrm{sym}}A=B^tJ^{\mathrm{sym}}-J^{\mathrm{skew}}C=D^tJ^{\mathrm{skew}}+J^{\mathrm{skew}}D=0.$$
The even part $\mg_0$ (resp. $\mg_1$) is the subset of all such matrices with $B=C=0$ (resp. $A=D=0$). In particular, $\mg_0\cong\mathfrak{so}(r)\oplus\mathfrak{sp}(2n)$ with its standard Cartan $\mathfrak{h}=\mathfrak{h}_0$ of all diagonal matrices. We denote therefore $\mg$ also by $\osprn$. Let 
\begin{eqnarray}
\label{intweights}
X=X(\mg)=\bigoplus_{i=1}^m\mathbb{Z}\eps_i\oplus\bigoplus_{j=1}^n\mathbb{Z}\delta_j.
\end{eqnarray}
be the integral weight lattice. Here the  
 $\eps_i$'s and $\delta_j$'s are the standard basis vectors of $\mh^*$ picking out the $i$-th respectively $(r+j)$-th diagonal matrix entry. We fix on $\mh^*$ the standard symmetric bilinear form $(\eps_i,\eps_j)=\delta_{i,j}$, $(\eps_i,\delta_j)=0$, $(\delta_i,\delta_j)=-\delta_{i,j}$ for $1\leq i\leq m$ and $1\leq j\leq n$. We define the parity (an element in $\mathbb{Z}/2\mathbb{Z}$) of the $\varepsilon$'s to be $0$ and the parity of the $\delta$'s to be $1$ and extend this to the whole weight lattice as the unique map of abelian groups to $\mathbb{Z}/2\mathbb{Z}$. In the following by a {\it weight} we always mean an {\it integral weight}. We will often denote weights as $(m+n)$-tuples $(a_1,a_2,\ldots, a_m \mid b_1,b_2,\dots, b_n)$, with the $\varepsilon$-coefficients to the left and the $\delta$-coefficients to the right of the vertical line.

Now $\mg$ decomposes into {\it root spaces} that is into weight spaces with respect to the adjoint action of $\mh$, $$\mg=\mh\oplus\bigoplus_{\alpha\in\Delta}\mg_\alpha.$$

One can check that $\mg_\alpha$ is either even or odd. Hence we can talk about {\it even roots} and {\it odd roots}.
Explicitly, the roots for $\osp(2m|2n)$ respectively $\osp(2m+1|2n)$ are the following (with $1\leq i\leq r$, $1\leq j\leq n$ such that the expressions exist),
\begin{eqnarray}
\label{roots}
\Delta(2m|2n)&=&\{\pm\eps_i\pm\eps_{i'}, \pm\delta_j\pm\delta_{j'}\mid i\not=i'\}\cup\{\pm\varepsilon_i\pm\delta_j\},\\
\Delta(2m+1|2n)&=&\{\pm\varepsilon_i,\pm\eps_i\pm\eps_{i'}, \pm\delta_j\pm\delta_{j'}\mid i\not=i'\}\cup\{\pm\delta_j,\pm\varepsilon_i\pm\delta_j\},\nonumber
\end{eqnarray}
where  all signs can be chosen independently. In each case the first set contains the even and the second the odd roots. 

\subsection{Supergroups and super Harish-Chandra pairs}

Let $\op{G}(r|2n)$ be the affine algebraic supergroup $\OSP(r|2n)$ over $\mC$. Using scheme-theoretic language, $\op{G}(r|2n)$ can be regarded as a functor $G$ from the category of commutative superalgebras over $\mC$
to the category of groups, mapping a commutative superalgebra $A = A_0\oplus A_{\overline{1}}$
to the group $\op{G}(A)$ of all  invertible $(r+2n) \times (r+2n)$ orthosymplectic matrices over $A$, see \cite[Section 3]{Vera}. This functor is representable by {\it an affine super Hopf algebra}  (i.e. a finitely generated supercommutative super Hopf algebra) $R=\mC[G]$, and there is a contravariant equivalence of categories between the categories of algebraic supergroups and of affine super Hopf algebras extending the situation of algebraic groups in the obvious way, see e.g. \cite{Fioresi}, \cite{Hopf}. By restricting the functor $\op{G}$ to commutative algebras defines an (ordinary) algebraic group $G_0$ represented by $R/I=\mC[G]/I$, where $I$ is the ideal generated by the odd part of $R$. In case of $\op{G}(r|2n)$  this algebraic group is just $\op{O}(r)\times \op{Sp}(2n)$. Similarly, we also have the affine algebraic supergroup $G'=\SOSP(r|2n)$ over $\mC$ with algebraic group $\op{SO}(r)\times \op{Sp}(2n)$. They both have $\osprn$ as the associated Lie superalgebra. We refer to \cite[Section 3]{Vera} for more details on these constructions.

We are interested in the category $\cC(r|2n)$ of finite-dimensional $\op{G}$-modules or equivalently the category of {\it integrable} $\mg$-modules, that is {\it Harish-Chandra modules} for the super Harish-Chandra pair $(\mg,G,\op{a})$, where $\op{a}$ is the adjoint action, see \cite{Vish}.  
To make this more precise we recall some facts.
\begin{definition}
A {\it super Harish-Chandra pair} is a triple $(\mathfrak{g},G_0,\op{a})$ where $\mathfrak{g}=\mathfrak{g}_0\oplus\mathfrak{g}_1$ is a Lie superalgebra, $G_0$ is an algebraic group with Lie algebra $\mathfrak{g}_0$, and $a$ is a $G_0$-module structure on $\mathfrak{g}$ whose differential is the adjoint action of $\mathfrak{g}_0$.  A {\it Harish-Chandra module} for such a triple or shorter a {\it $(\mathfrak{g},G_0,\op{a})$-module} is then a $\mg$-module $M$ with a compatible $G_0$-module structure (that means the derivative of the $G_0$-action agrees with the action of $\mg_0$). We denote by  $(\mathfrak{g},G_0,\op{a})-\op{mod}$ the category of finite-dimensional $(\mathfrak{g},G_0,\op{a})$-modules.
\end{definition}
Given any super Harish-Chandra pair $(\mathfrak{g},G_0,\op{a})$ one can construct a Hopf superalgebra $R=\mC[G]$ such that $\mg$ is the Lie algebra of the supergroup $G$ and $R/I=\mC[G_0]$.  Namely $R=\op{Hom}_{U(\mg_0)}(U(\mg),\mC[G_0])$, where $U(\mathfrak{h})$ denotes the universal enveloping algebra of a Lie superalgebra $\mathfrak{h}$, and where ${U(\mg_0)}$ acts by left invariant derivations on $\mC[G_0]$, see \cite[(3.1)]{Vera} for precise formulas and the description of the Hopf algebra structure - with the dependence on the action $\op{a}$. This assignment  $\Phi: (\mathfrak{g},G_0,\op{a})\mapsto G$
for any super Harish-Chandra pair can be extended in fact to the following equivalence of categories, see \cite{Vish}, \cite{Balduzzi} for the super case, but the arguments are very much parallel to the classical case from  \cite{Kostant}.

\begin{prop} The assignment  $\Phi: (\mathfrak{g},G_0,\op{a})\mapsto G$ induces the following:
\begin{enumerate}
\item 
The category of super Harish-Chandra pairs is equivalent to the category of algebraic supergroups. 
\item 
Moreover the category of finite-dimensional $(\mathfrak{g},G_0,\op{a})$-modules, denoted by $(\mathfrak{g},G_0,\op{a})\!-\!\op{mod}$, is equivalent to the category  $G\!-\!\op{mod}$ of finite-dimensional $G$-modules. 
\end{enumerate}
\end{prop}
 The category $\cC(G)$ of finite-dimensional $\op{G}$-modules has enough projectives and enough injectives, \cite[Lemma 9.1]{Vera}, in fact projective and injective modules agree, \cite[Proposition 2.2.2]{BKN}. 

\section{Finite-dimensional representations}
We are interested in the category of Harish-Chandra modules for the particular super Harish-Chandra pairs arising from the (special) orthosymplectic supergroups. Since the action $\op{a}$ in this cases is always the adjoint action, we will usually omit it in the notation. 
From now on we fix $r,n\in\mathbb{Z}_{\geq0}$ and use the following abbreviations:
\abovedisplayskip0.5em
\belowdisplayskip0.5em
\begin{eqnarray*}
&\mg=\osp(r|2n) \qquad G=\OSPrn,\qquad G'=\SOSPrn,&\\
& \cC=\cC(\OSPrn),\qquad \cC'=\cC(\SOSPrn).&
\end{eqnarray*}

The simple objects in $\cC'$ are (viewed as Harish-Chandra modules) highest weight modules, and every simple object is up to isomorphism and parity shift uniquely determined by its highest weight, see e.g. \cite[Theorem 9.9]{Vera}. The category $\cC'$ decomposes into a direct sum of two equivalent subcategories
\[ \cC'=\cF(\SOSPrn)\oplus \parswitch \cF(\SOSPrn),\]
namely $\cF(\SOSPrn)$ and its parity shift $\parswitch\cF(\SOSPrn)$, where the category $\cF(\SOSPrn)$ contains all objects such that the parity of any weight space agrees with the parity of the corresponding weight. Similarly, the categories $\cC$ decomposes into $\cF=\cF(\OSPrn)$ and its parity shift, where $\cF$ consists of those modules that lie in $\cF(\SOSPrn)$ when restricted to $\SOSPrn$.  Therefore it suffices to restrict ourselves to study the summands
\abovedisplayskip0.25em
\belowdisplayskip0.25em
\[ \cF'=\cF(\SOSPrn)\qquad\text {respectively}\qquad \cF=\cF(\OSPrn),\]
which we will consider now in more detail.

\subsection{Finite-dimensional representations of \texorpdfstring{$\SOSPrn$}{SOSp(r|2n)}}
We first consider the case of the {\it special} orthosymplectic group. With a fixed Borel subalgebra in $G'$, every irreducible module in $\cF'$ (viewed as integrable module for $\mg$) is a quotient of a Verma module, in particular a highest weight module $L(\la)$ for some highest weight $\la$, see \cite[Theorem 9.9]{Vera}. The occurring highest weights are precisely the {\it dominant} weights. The explicit dominance condition on the coefficients of $\la$ in our chosen basis \eqref{intweights} depends on the choice of Borel we made, since in the orthosymplectic case Borels are not always pairwise conjugate. We follow now closely \cite{GS2} and fix the slightly unusual choice of Borel with maximal possible number of odd simple roots, see \cite{GS1}, with the simple roots given as follows:

\noindent For $\osp(2m+1|2n)$:
\begin{eqnarray*}
\;\;\text{if $m\geq n:$}&&
\!\!\!\!\!\!\!\!\!\!\left\{
\begin{array}[c]{c}
\eps_1-\eps_2,\eps_2-\eps_3,\ldots, \eps_{m-n}-\eps_{m-n+1},\\
\eps_{m-n+1}-\delta_1,\delta_{1}-\eps_{m-n+2}, \eps_{m-n+2}-\delta_2,\ldots,\eps_m-\delta_n,\delta_n.
\end{array}
\right.\quad\\
\;\;\text{if $m< n:$}&&
\!\!\!\!\!\!\!\!\!\!\left\{
\begin{array}[c]{c}
\delta_1-\delta_2,\delta_2-\delta_3,\ldots, \delta_{n-m-1}-\delta_{n-m},\\
\delta_{n-m}-\eps_1,\eps_{1}-\delta_{n-m+1}, \delta_{n-m+1}-\eps_2,\ldots,\eps_m-\delta_n,\delta_n.
\end{array}
\right.\quad
\end{eqnarray*}
For $\osp(2m|2n)$:
\begin{eqnarray*}
\;\;\text{if $m> n:$}&&
\!\!\!\!\!\!\!\!\!\!\left\{
\begin{array}[c]{c}
\eps_1-\eps_2,\eps_2-\eps_3,\ldots, \eps_{m-n-1}-\eps_{m-n},\\
\eps_{m-n}-\delta_1,\delta_{1}-\eps_{m-n+1}, \eps_{m-n+1}-\delta_2,\ldots,\delta_n-\varepsilon_m,\\
\delta_n+\varepsilon_m.
\end{array}
\right.\quad
\\
\;\;\text{if $m\leq n:$}&&
\!\!\!\!\!\!\!\!\!\!\left\{
\begin{array}[c]{c}
\delta_1-\delta_2,\delta_2-\delta_3,\ldots, \delta_{n-m}-\delta_{n-m+1},\\
\delta_{n-m+1}-\eps_1,\eps_{1}-\delta_{n-m+2}, \delta_{n-m+2}-\eps_2,\ldots,\delta_n-\varepsilon_m,\\
\delta_n+\varepsilon_m.
\end{array}
\right.
\end{eqnarray*}

This choice of Borel has the advantage that the dominance conditions look similar to the ordinary ones for semisimple Lie algebras and moreover is best adapted to our diagrammatics. To formulate it, let $\rho$ be half of the sum of positive even roots minus the sum of positive odd roots, explicitly given as follows.

\noindent 
For $\mg=\osp(2m+1|2n)$: In this case $\tfrac{\de}{2}=m-n+\tfrac{1}{2}$ and 
\abovedisplayskip0.25em
\belowdisplayskip0.25em
\begin{eqnarray*}
\rho&=&
\begin{cases}
\left(\tfrac{\de}{2}-1, \tfrac{\de}{2}-2,\ldots,\tfrac{1}{2},-\frac{1}{2},\ldots,-\tfrac{1}{2}\, \middle|\, \frac{1}{2},\dots,\tfrac{1}{2}\right)&\text{if $m\geq n$},\\
\left(-\tfrac{1}{2},\ldots, -\tfrac{1}{2} \, \middle|\, -\tfrac{\de}{2}, -\tfrac{\de}{2}-1,\ldots, \tfrac{1}{2},\ldots, \tfrac{1}{2}\right)&\text{if $m<n$.}\\
\end{cases}
\end{eqnarray*}
For $\mg=\osp(2m|2n)$: In this case $\tfrac{\de}{2}=m-n$ and 
\abovedisplayskip0.25em
\belowdisplayskip0.25em
\begin{eqnarray*}
\rho&=&
\begin{cases}
\left(\tfrac{\de}{2}-1,\tfrac{\de}{2}-2,\ldots,1,0, \dots,0 \, \middle| \, 0,\ldots, 0 \right)&\text{ if $m>n$,}\\
\left(0,\ldots, 0\, \middle| \, -\tfrac{\de}{2},-\tfrac{\de}{2}-1,\ldots,1,0,\dots,0\right)&\text{ if $m\leq n$}.
\end{cases}
\end{eqnarray*}

\begin{remark}
Note that $n=0$ gives $\rho=(m-1,m-2,\ldots,0)$ for $m$ even and $\rho=(m-\tfrac{1}{2},m-\tfrac{3}{2},\ldots,\tfrac{1}{2})$ for $m$ odd; and $\rho=(n,n-1,\ldots,1)$ in case $m=0$. These are the values for $\rho$ for the semisimple Lie algebras of type $\mathbf{D}_m$, $\mathbf{B}_m$, $\mathbf{C}_n$.
\end{remark}

\begin{definition}
\label{dominance}
For our choice of Borel, a weight $\la\in X(\mg)$ is {\it dominant} if
\begin{eqnarray}
\label{laab}
\la+\rho&=&\sum_{i=1}^m a_i\eps_i+\sum_{j=1}^n b_j\delta_j
\end{eqnarray}
satisfies the following dominance condition, see \cite{GS1}.\\[-0.3cm]

\noindent
For $\mg=\osp(2m+1|2n)$:
\begin{enumerate}[(i)]
\item either $a_1>a_2>\cdots> a_m\geq \frac{1}{2}$ and  $b_1>b_2>\cdots> b_n\geq \frac{1}{2}$,
\item
or $a_1>a_2>\cdots> a_{m-l-1}> a_{m-l}=\cdots= a_m=-\frac{1}{2}$ and \hfill\\ \phantom{or }$b_1>b_2>\cdots> b_{n-l-1}\geq b_{n-l}=\cdots= b_n=\frac{1}{2}$,
\end{enumerate}

\noindent
For $\mg=\osp(2m|2n)$:
\begin{enumerate}[(i)]
\item either $a_1>a_2>\cdots> a_{m-1}>|a_m|$ and  $b_1>b_2>\cdots> b_n>0$,
    \item
or $a_1>a_2>\cdots> a_{m-l-1}\geq a_{m-l}=\cdots= a_m=0$ and \hfill\\ \phantom{or }$b_1>b_2>\cdots> b_{n-l-1}> b_{n-l}=\cdots= b_n=0$.
\end{enumerate}
The set of dominant weights is denoted $X^+(\mg)$. Note that
\abovedisplayskip0.25em
\belowdisplayskip0.5em
\[
X^+(\osp(2m+1|2n))\subset (\mZ+\tfrac{1}{2})^{m+n} \text{ and } X^+(\osp(2m|2n))\subset \mZ^{m+n}.\]
\end{definition}

\begin{definition}
Assume $r=2m$. If $\la\in X^+(\mg)$, written in the form \eqref{laab}, satisfies $a_m\not=0$, then we write $\la=\la_+$ if $a_m>0$, and  we write $\la=\la_-$ if $a_m<0$.
\end{definition}

\begin{definition}
\label{tailg}
Weights satisfying $(i)$ are called {\it tailless} and the number $l+1$ from Definition~\ref{dominance} is the \it{tail length}, $\op{tail}(\la)$, of $\la$.
\end{definition}

\begin{ex}{\rm
The zero weight is always dominant with maximal possible tail length, namely  $\op{tail}(0)=\op{min}\{m,n\}$.} 
\end{ex}

For $\la\in X^+(\mg)$ let $P^\mg(\la)$ be the projective cover of $L^\mg(\la)$, see \cite{BKN} for a construction, and $I^\mg(\la)$ its injective hull. Then the $P^\mg(\la)$ (respectively $I^\mg(\la)$), with $\la\in X^+(\mg)$,  form a complete non-redundant set of representatives for the isomorphism classes of indecomposable projective (resp. injective) objects in $\cF'$.

\subsection{Finite-dimensional representations of \texorpdfstring{$\OSP(r|2n)$}{OSp(r,2n)}}
We recall the classification of simple finite-dimensional representations of $G$ using the one for $G'$. 

For this let $\sigma\in\mZ/2\mZ$ be the non-unit element. Via \eqref{semidirect} it corresponds to an element in $\mathrm{O}(2m)$, also called $\sigma$, which acts by conjugation on  $\mathrm{SO}(2m)$, as well as $\mathfrak{g}$ and preserving the Cartan $\mathfrak{h}$. On weights it acts as $\sigma(\varepsilon_m)=-\varepsilon_m$ and $\sigma(\varepsilon_i)=\varepsilon_i$, $\sigma(\delta_i)=\delta_i$ for $1\leq i\leq m-1$, $1\leq j\leq n$. We have $\mathrm{O}(2m)=\mathrm{SO}(2m)\cup \sigma \mathrm{SO}(2m)$.

To construct the irreducible representations we use  a very special case of Harish-Chandra induction which we recall now. Let $(\mg, H, a)$ be a super Harish-Chandra pair and $H'$ a subgroup of $H$ such that $(\mg, H', a'=a_{|{H'}})$ is also a super-Harish Chandra pair. Then there is a {\it (Harish-Chandra) induction functor} 
\begin{eqnarray}
\label{ind}
\op{Ind}^{\mg, H}_{\mg, H'}: \quad(\mg, H', a')-\op{mod}&\longrightarrow &(\mg, H,\op{a})-\op{mod},
\end{eqnarray}
where $\op{Ind}^{\mg, H}_{\mg, H'}N=\{f: H \rightarrow N\mid f(xh)=xf(h), h\in H, x\in H'\}$ is the usual induction for algebraic groups, \cite[3.3]{Jantzen}. The $H$-action is given by the right regular action and the $\mg$-action is just the $\mg$-action on $N$. This functor  $\op{Ind}^{\mg, H}_{\mg, H'}$ is left exact. It sends injective objects to injective objects, \cite[Proposition 3.9]{Jantzen}, and it is right adjoint to the restriction functor $\op{Res}^{\mg, H'}_{\mg, H}$, \cite[Proposition 3.4]{Jantzen}. \\

We apply this to the two super Harish-Chandra pairs $(\mg,G')$ and $(\mg,G)$. 

\subsubsection{The odd case: \texorpdfstring{$\SOSP(2m+1|2n)$}{SOSp(2m+1,2n)}} \label{sec:oddcase}
In this section we assume $r=2m+1$ is odd. The element $\sigma$ is central and thanks to \eqref{easyodd} we can describe the simple objects in $\cF$:

\begin{prop}\label{def:labelssimples1}
For $G=\OSP(2m+1|2n)$ the set
\begin{eqnarray*}
X^+(G) &=& X^+(\mathfrak{g})\times\mathbb{Z}/2\mathbb{Z}=\{ (\lambda,\epsilon) \mid \lambda \in X^+(\mathfrak{g}), \epsilon\in\{\pm\}\}
\end{eqnarray*}
is a labelling set for the isomorphism classes of irreducible $G$-modules in $\cF$. The simple module $L(\la,\pm)$ is hereby just the simple $G'$-module $L^\mg(\la)$ extended to a module for $G$ by letting $\sigma$ act by $\pm 1$. 
\end{prop}

Observe that $\op{Ind}^{\mg, G}_{\mg, G'}L^\mg(\la)\cong L(\la,+)\oplus L(\la,-)$. By construction, the category $\cF$ decomposes as $\cF^+\oplus \cF^-$, where $\cF^{\pm}$ is the full subcategory of $\cF$ containing all representations with composition factors only of the form $L(\la,\pm)$, and moreover $\cF^{\pm}\cong \cF'$. 

\begin{remark}
\label{theds}
Note that the natural vector representation $V=\mC^{2m+1|2n}$ can be identified with $L(\epsilon_1,-1)$ in case $m > n$ and with $L(\delta_1,-1)$ in case $m \leq n$. In particular, $-\mathrm{id} \in G$ acts on a $d$-fold tensor product $V^{\otimes d}$ by $(-1)^d$. This implies that there is no $G$-equivariant morphism from $V^{\otimes d}$ to $V^{\otimes d'}$ in case $d\not\equiv d'\op{mod} 2$.
\end{remark}

\begin{remark} \label{rem:plusminusodd}
In particular we have for $\la,\mu\in X^+(\mg)$
\begin{eqnarray}
\label{plusminusIodd}
\op{Hom}_\cF(I(\la,+),I(\mu,-))\;=\!&\!\{0\}\!&\!=\;\op{Hom}_\cF(I(\la,-),I(\mu, +)),\\
\label{plusminusPodd}
\op{Hom}_\cF(P(\la,+),P(\mu,-))\;=\!&\!\{0\}\!&\!=\;\op{Hom}_\cF(P(\la,-),P(\mu, +)),
\end{eqnarray}
and the nonzero morphism spaces are controlled by those for $\mg$, more precisely
\begin{eqnarray}
\label{olddim}
\op{Hom}_\cF(P(\la,\pm),P(\mu,\pm))&\!=&\!\op{Hom}_{\cF'}(P^\mg(\la),P^\mg(\mu)).
\end{eqnarray}
\end{remark}

\begin{corollary}
\label{corblockodd}
Let $(\la,\epsilon)$ and $(\mu,\epsilon')$ in $X^+(G)$. Then $L(\la,\epsilon)$ and $L(\mu,\epsilon')$ are in the same block of $\cF$ if and only if $\epsilon=\epsilon'$ and $L^\mg(\la)$ and $L^\mg(\mu)$ are in the same block of $\cF'$.
\end{corollary}

\subsubsection{The even case: \texorpdfstring{$\SOSP(2m|2n)$}{SOSp(2m,2n)}} \label{sec:evencase}
Let now $r=2m$ be even. In this case the situation is slightly more involved, since $\sigma$ is not central. We first construct the irreducible representations using Harish-Chandra induction. 

\begin{definition} \label{def:labelssimples2}
For $G=\OSP(2m|2n)$ we introduce the following set:
\begin{eqnarray*}
X^+(G) &=&\{ (\lambda,\epsilon) \mid \lambda \in X^+(\mathfrak{g})/\sigma \text { and } \epsilon \in {\rm Stab}_{\sigma}(\lambda) \},
\end{eqnarray*}
where ${\rm Stab}_{\sigma}$ denotes the stabilizer of $\lambda$ under the group generated by $\sigma$.
\end{definition}

\begin{notation}
\label{notationweights}
To avoid overloading of notation we usually just write $\la$ instead of $(\lambda,\epsilon)$ if the representatives of $\la$ have trivial stabilizer. Otherwise the orbit has a unique element. In this case the stabilizer has two elements and we often write $(\lambda,+)$ for $(\lambda, e)$ and $(\lambda,-)$ for $(\lambda,\sigma)$. In addition we write $\lambda^G$ for the $\sigma$-orbit of $\lambda \in X^+(\mg)$. We will omit this superscript if the orbit consists of a single element.
\end{notation}
\begin{prop}
\label{erste}
Consider $\mg=\osp(2m|2n)$, $G=\OSP(2m|2n)$, and $G'=\SOSP(2m|2n)$. Assume 
\abovedisplayskip0.25em
\belowdisplayskip0.25em
\begin{eqnarray}
\label{oje}
\la=\sum_{i=1}^m a_i\varepsilon_i+\sum_{j=1}^n b_j\delta_j-\rho
\in X^+(\mg)
\end{eqnarray}
and let $L^\mg(\la)\in \cF'$ be the corresponding irreducible highest weight representation of $G'$ with injective cover $I^\mg(\la)$.  Then with $\op{Ind}^{\mg, G}_{\mg, G'}$ from \eqref{ind} the following holds: 
\begin{enumerate}
\item for induced irreducible representations:
\begin{enumerate}
\item If  $a_m\not=0$ then the $(\osp(2m|2n), \OSP(2m|2n))$-module 
\begin{eqnarray}
\label{Lla}
L(\la^G) = L(\la^G,e) &:=&\op{Ind}^{\mg, G}_{\mg, G'}L^\mg(\la)
\end{eqnarray}
is irreducible. Moreover, 
\begin{equation}
\label{zumGlueck}
\op{Ind}^{\mg, G}_{\mg, G'}L^\mg(\la)\;\;\cong\;\;\op{Ind}^{\mg, G}_{\mg, G'}L^\mg(\sigma(\la)).
\end{equation}
\item  If  $a_m=0$ then  
\begin{eqnarray}
\label{Lplus}
\op{Ind}^{\mg, G}_{\mg, G'}L^\mg(\la)&=:& L(\la,+)\oplus L(\la,-)
\end{eqnarray}
is a direct sum of $L(\la,+)$, and $L(\la,-)$, two non-isomorphic irreducible $(\osp(2m|2n), \OSP(2m|2n))$-modules. As $G'$-modules they  are isomorphic to $L^\mg(\la)$. 
\end{enumerate}
\item  for induced injective representations:
\begin{enumerate}
\item If  $a_m\not=0$ then  $I(\la^G):=\op{Ind}^{\mg, G}_{\mg, G'}I^\mg(\la)$ is the indecomposable injective hull of $L(\la^G)$.
\item If  $a_m=0$ then  $\op{Ind}^{\mg, G}_{\mg, G'}I^\mg(\la)\cong I(\la,+)\oplus I(\la,-),$ where  $I(\la,\pm)$ denotes the injective hull of $L(\la,\pm)$.
\end{enumerate}
\end{enumerate}
The same formulas hold for the indecomposable projective objects. 
\end{prop}

As a consequence we obtain the following:

\begin{prop}
\label{zweite}
The $\{ L(\lambda,\epsilon) \mid (\lambda,\epsilon) \in X^+(G)\}$ are a complete non-redundant set of representatives for the isomorphism classes of irreducible $G$-modules in $\cF$.
\end{prop}

\begin{proof}[Proof of Propositions~\ref{erste} and~\ref{zweite}]
The arguments for (1) of Proposition~\ref{erste} and the classification of irreducible representations in Proposition~\ref{zweite} are precisely as in the classical case, see  e.g.  \cite[5.5.5]{GW}. By construction and the proof there, 
\begin{equation}
\label{eqres}
\op{Res}^{\mg, G'}_{\mg, G}L(\la,\pm)\cong L^\mg(\la) \text{ and } \op{Res}^{\mg, G'}_{\mg, G}L(\la^G)\cong L^\mg(\la)\oplus L^\mg(\sigma(\la))
\end{equation}
where $(\la,\pm)$ is as in (1)(b) respectively $\la$ as in (1)(a). 
More precisely, it is proved that the modules $L(\la,\pm)$  are isomorphic to $L^\mg(\la)$ as $G'$-modules; with the action extended to $G$ such that  $\sigma$ acts on the highest weight vector by multiplication with the scalar $1$ or $-1$ (but see also Lemma~\ref{lemplusminus} and Remark~\ref{rem:plusminusodd}).

Since the functor $\op{Ind}^{\mg, G}_{\mg, G'}$ sends injective objects to injective objects, and is right adjoint to the restriction functor, \cite[Propositions 3.4~and~3.9]{Jantzen}, the statements (2) of Proposition~\ref{erste} can be deduced as follows. Let $\mu\in X^+(\mg)$ with 
\begin{eqnarray}
\label{muinab}
\mu=\sum_{i=1}^m a'_i\varepsilon_i+\sum_{j=1}^n b'_j\delta_j-\rho.
\end{eqnarray}
If $a_m\not=0$ in \eqref{oje} we obtain, for any simple $G$-module $L\in\cF$, by adjunction and the first paragraph of the proof
\begin{align*}
&\op{Hom}_\cF(L, \op{Ind}^{\mg, G}_{\mg, G'}I^\mg(\la))\\
\cong&
\begin{cases}
\op{Hom}_{\cF'}(L^\mg(\mu)\oplus L^\mg(\sigma(\mu)), I^\mg(\la)) & \text{if $L=L(\mu^G)$, i.e. $a'_m\not=0$,}\\
\op{Hom}_{\cF'}(L^\mg(\mu), I^\mg(\la))=\{0\} & \text{if $L=L(\mu, \pm)$, i.e. $a'_m=0$,}
\end{cases}
\end{align*}
If $a_m=0$, we have
\begin{align*}
&\op{Hom}_\cF(L, \op{Ind}^{\mg, G}_{\mg, G'}I^\mg(\la))\\
\cong&
\begin{cases}
\op{Hom}_{\cF'}(L^\mg(\mu)\oplus L^\mg(\sigma(\mu)), I^\mg(\la))=\{0\},&\text{if $L=L(\mu^G)$, i.e. $a'_m\not=0$,}\\
\op{Hom}_{\cF'}(L^\mg(\mu), I^\mg(\la)),&\text{if $L=L(\mu, \pm)$, i.e. $a'_m=0$,}
\end{cases}
\end{align*}
noting that in the first case the homomorphism space vanishes since $ I^\mg(\la)$ is the injective hull of $ L^\mg(\la)$. This proves part (2) in Proposition~\ref{erste}.

To prove the analogous statements (3) for indecomposable projective modules, recall that any indecomposable projective is also injective, \cite[Proposition 2.2.2]{BKN}. Hence $P(\la^G)\cong I(\Phi(\la^G))$ and $P(\la,\pm)\cong I(\Phi(\la,\pm))$ for some function $\Phi:X^+(\mg)\rightarrow X^+(\mg)$. By \cite[Proposition 2.2.1]{BKN}, the function $\Phi$ can be computed as follows: let $N=\dim \mg_{1}=8mn$ and consider the $1$-dimensional $\mg_0$-module $\bigwedge^N\mg_1$ of weight $\nu$.  Then there is an isomorphism of $G'$-modules $P^\mg(\la)\cong I^\mg(\la+\nu)$.  Set $\mu=\la+\nu$. Then, using the explicit description \eqref{roots} of the odd  roots in $\osp(2m|2n)$, one can easily check that in $\nu$ the coefficient for $\varepsilon_m$ vanishes, and therefore $a_m\not=0$ if and only if $a_m'\not=0$ in the notation of \eqref{oje} and \eqref{muinab}. Hence, $\Phi$ preserves the condition $a_m\not=0$. Therefore, the formulas for induced projective modules agree with the formulas for the induced injective modules.  
\end{proof}

We have restriction formulas for the projective-injective modules as follows.
\begin{lemma}
\label{res}
Let $\la\in X^+(\mg)$. There are isomorphisms of  $G'$-modules 
\begin{eqnarray*}
\op{Res}^{\mg, G'}_{\mg, G}I(\la^G) &\cong& I^\mg(\la)\oplus I^\mg(\sigma(\la)) \text{ if } \lambda \neq \sigma(\lambda) \text{ and} \\
\op{Res}^{\mg, G'}_{\mg, G}I(\la,\pm)&\cong& I^\mg(\la) \text{ otherwise}.
\end{eqnarray*}
Similarly for the indecomposable projective objects. 
\end{lemma}

\begin{proof}
Let $P\in\cF$ be indecomposable projective. Then $\op{Hom}_\cF(P,\underline{\phantom{x}})$ is exact. The induction functor $\op{Ind}^{\mg, G}_{\mg, G'}$ is exact as well, due to  \eqref{semidirect}, see \cite[3.8.(3),~or~4.9]{Jantzen}. Moreover it is right adjoint to the restriction functor, thus we obtain that $\op{Res}^{\mg, G'}_{\mg, G}P$ is projective. The restriction formulas for projective modules follow then using adjunction from \eqref{Lla} and \eqref{Lplus}. Via the identification with indecomposable injective objects (as in the last part of the proof of Proposition~\ref{zweite}), the claims follow also for these.
\end{proof}

\begin{lemma}
Assume $\la\in X^+(\mg)$ with $a_m=0$ in the notation from \eqref{oje}. Then as $G'$-modules $I(\la,+)\cong I(\la,-)$,  similarly for $P(\la,\pm)$.
\end{lemma}

\begin{proof}
We first claim that our Harish-Chandra induction commutes with Lie algebra induction in the following sense. Let $M$ be a finite-dimensional Harish-Chandra module for $(\mg_0,G')$. Then there is a natural isomorphism of Harish-Chandra modules for $(\mg,G)$ as follows
\begin{eqnarray}
\label{IndInd}
U(\mg)\otimes_{U(\mg_0)}(\op{Ind}^{\mg_0, G}_{\mg_0, G'}M)&\cong&\op{Ind}^{\mg, G}_{\mg, G'}(U(\mg)\otimes_{U(\mg_0)}M).\nonumber\\
u\otimes f&\mapsto &f_u,
\end{eqnarray}
where $f_u(g)=u\otimes f(g)$ for any $g\in G$. The map is obviously well-defined and injective, and therefore also an isomorphism by a dimension count using that $U(\mg)$ is free over $U(\mg_0)$ of finite rank. 

By Proposition~\ref{erste} (2)(b) we obtain $P(\la,+)\oplus P(\la,-)\cong \op{Ind}^{\mg, G}_{\mg, G'} P^\mg(\la)$. On the other hand, by \cite[Proof of Proposition 2.2.2]{BKN}, the indecomposable $(\mg,G')$-module $P^\mg(\la)$ is a summand of $U(\mg)\otimes_{U(\mg_0)}L_0(\la)$, where $L_0(\la)$ is the  irreducible $(\mg_0, G')$-Harish-Chandra module of highest weight $\la$. 

Together with \eqref{IndInd}, Proposition~\ref{erste} implies that $P(\la,\pm)$ is a summand of
\begin{eqnarray*}
U(\mg)\otimes_{U(\mg_0)}(\op{Ind}^{\mg_0, G}_{\mg_0, G'}L_0(\la))\cong 
 U(\mg)\otimes_{U(\mg_0)}(L_0(\la,+)\oplus L_0(\la,-))
\end{eqnarray*}
By carefully following the highest weight vectors through the isomorphism we obtain that $P(\la,\pm)$ is in fact a summand of  $U(\mg)\otimes_{U(\mg_0)}L_0(\la,\pm)$. By Proposition~\ref{erste}~(1)(b) the action of $\sigma$ on the highest weight vector of $L_0(\la,+)$ is given by a scalar, hence it acts by the same scalar on the highest weight vector of $U(\mg)\otimes_{U(\mg_0)}L_0(\la,+)$, and thus also on  the highest weight vector of $P(\la,\pm)$. The analogous statements hold then for $I(\la,\pm)$ as well (again via the identification $\Phi$ from the proof of Proposition~\ref{zweite}).
\end{proof}


We deduce now a few dimension formulas for homomorphism spaces.
\begin{prop}  
\label{blocksO}
With the notations from Proposition~\ref{erste}, in particular \eqref{oje} and \eqref{muinab}, we have the following.
\begin{enumerate} 
\item Let $(\la,\epsilon),\mu^G \in X^+(G)$ with $\mathrm{Stab}_\sigma(\mu)$ being trivial, then 
\begin{eqnarray}
\label{first}
\op{dim}\op{Hom}_{\cF}(I(\la,\epsilon),I(\mu^G))&\!\!=&\!\!\op{dim}\op{Hom}_{\cF'}(I^\mg(\la),I^\mg(\mu))\\
\label{firstb}
&\!\!=&\!\!\op{dim}\op{Hom}_{\cF'}(I^\mg(\la),I^\mg(\sigma(\mu)),\\
\label{second}
\op{dim}\op{Hom}_{\cF}(I(\mu^G),I(\la,\epsilon))
&\!\!=&\!\!\op{dim}\op{Hom}_{\cF'}(I^\mg(\mu), I^\mg(\la))\\
\label{secondb}
&\!\!=&\!\!\op{dim}\op{Hom}_{\cF'}(I^\mg(\sigma(\mu)), I^\mg(\la)).
\end{eqnarray}
\item Let $\la^G,\mu^G \in X^+(G)$ with $\mathrm{Stab}_\sigma(\lambda)$ and $\mathrm{Stab}_\sigma(\mu)$ trivial, then
\begin{eqnarray}
\label{eins}
\op{dim}\op{Hom}_{\cF}(I(\la^G),I(\mu^G))&\!\!=&\!\!\op{dim}\op{Hom}_{\cF'}(I^\mg(\la),I^\mg(\mu))\\
\label{einsb}
&\!\!=&\!\!\op{dim}\op{Hom}_{\cF'}(I^\mg(\sigma(\la)),I^\mg(\sigma(\mu)),
\end{eqnarray}
where $\lambda$ and $\mu$ are chosen such that either $a_m > 0 < a_m'$ or $a_m < 0 > a_m'$.
\item Let $(\la,\epsilon),(\mu,\epsilon^\prime) \in X^+(G)$, then
\begin{eqnarray}
&&\op{dim}\op{Hom}_{\cF}(I(\mu,+) \oplus I(\mu,-),I(\mu,+) \oplus I(\mu,-)) \quad\quad\nonumber\\
&=&2\op{dim}\op{Hom}_{\cF'}(I^\mg(\la),I^\mg(\mu)),
\end{eqnarray}
\end{enumerate}
The analogous formulas hold for indecomposable projective objects.
\end{prop}
\begin{proof}
For the first statement \eqref{first} we calculate using Proposition~\ref{erste}, adjunction of restriction and induction, and Lemma~\ref{res}
\begin{eqnarray*}
\op{dim}\op{Hom}_{\cF}(I(\la,\pm),I(\mu^G))&=&\op{dim}\op{Hom}_{\cF}(I(\la,\pm),\op{Ind}^{\mg, G}_{\mg, G'} I^\mg(\mu))\\
&=&\op{dim}\op{Hom}_{\cF'}(\op{Res}^{\mg, G'}_{\mg, G}I(\la,\pm),I^\mg(\mu))\\
&=&\op{dim}\op{Hom}_{\cF'}(I^\mg(\la),I^\mg(\mu)).
\end{eqnarray*}
Similarly, \eqref{firstb} holds.
Again, the same formulas hold for projective objects. On the categories $\cF$ and $\cF'$ there is the usual duality $\boldsymbol{ \op{d}}$, \cite[13.7.1]{Musson}, given by taking the sum of the vector space dual of the weight spaces with the action of $\mg$, $G$, $G'$ twisted by the Chevalley automorphism. This duality sends simple objects to simple objects and their injective hulls to the projective covers. Applying $\boldsymbol{\op{d}}$ to \eqref{first} resp.\eqref{firstb} gives \eqref{second} and  \eqref{secondb}.

For the statement  \eqref{eins} we calculate
\begin{eqnarray*}
\op{dim}\op{Hom}_{\cF}(I(\la^G),I(\mu^G))&=&\op{dim}\op{Hom}_{\cF}(I(\la^G),\op{Ind}^{\mg, G}_{\mg, G'} I^\mg(\mu))\\
&=&\op{dim}\op{Hom}_{\cF'}(\op{Res}^{\mg, G'}_{\mg, G}I(\la^G),I^\mg(\mu))\\
&=&\op{dim}\op{Hom}_{\cF'}(I^\mg(\la)\oplus I^\mg(\sigma(\la)),I^\mg(\mu))\\
&=&\op{dim}\op{Hom}_{\cF'}(I^\mg(\la),I^\mg(\mu)),
\end{eqnarray*}
again using Proposition~\ref{erste}, adjunction and Lemma~\ref{res} for the first to  third equalities. The last one follows from the Gruson-Serganova combinatorics, \cite{GS2}, see Proposition~\ref{prop:alternating} (1). Hence, \eqref{eins} and similarly \eqref{einsb} follow. 


The equality in the third statement follows from
\begin{eqnarray*}
&&\op{dim}\op{Hom}_{\cF}(I(\la,+) \oplus I(\la,-),I(\mu,+) \oplus I(\mu,-))\\
&=&\op{dim}\op{Hom}_{\cF}( \op{Ind}_{\mg, G'}^{\mg, G}I(\la),I(\mu,+) \oplus I(\mu,-))\\
&=&\op{dim}\op{Hom}_{\cF}I(\la),  \op{Res}^{\mg, G'}_{\mg, G}(I(\mu,+) \oplus I(\mu,-)))\\
&=&2\op{dim}\op{Hom}_{\cF'}(I^\mg(\la),I^\mg(\mu)).
\end{eqnarray*}
where we used again Lemma~\ref{res}, adjunction, and Proposition~\ref{erste}. The analogous formulas for the projectives hold as well.
\end{proof}

The following refines the last part of Proposition~\ref{blocksO}.
\begin{lemma}
\label{lemplusminus}
Let $G=\op{OSp}(2m|2n)$. In the notation from \eqref{oje} consider the set $X^+(\mg)_{\operatorname{sign}}=\{\la\in X^+(\mg)\mid a_m = 0\}$. 
The sign in the labelling  of the irreducible modules from \eqref{Lplus} can be chosen in such a way such that for any $\la,\mu\in X^+(\mg)_{\op{sign}}$ one of the following holds.
\begin{enumerate}
\item \label{plusminusmixed} Either $\op{dim}\op{Hom}_\cF(P(\la,\epsilon),P(\mu,\epsilon'))$ is independent of  $\epsilon, \epsilon'\in\{+,-\}$ and equal to $\tfrac{1}{2}\op{dim}\op{Hom}_{\cF'}(P^\mg(\la),P^\mg(\mu))$,
\item or there exists $\epsilon, \epsilon'\in\{+,-\}$ such that 
\begin{eqnarray}
\label{above}
 \op{dim}\op{Hom}_\cF(P(\la,\epsilon),P(\mu,\epsilon'))&=&\op{dim}\op{Hom}_{\cF'}(P^\mg(\la),P^\mg(\mu))\not=\{0\},\quad\quad
 \end{eqnarray}
in which case the same holds if we change both signs $\epsilon$ and $\epsilon'$ in \eqref{above}, whereas the left hand side vanishes if only one of the two is changed. 
\end{enumerate}
The analogous statement holds for the indecomposable injectives as well. In both situations the dimensions of the morphism spaces are invariant under interchanging the two objects.
\end{lemma}

\begin{proof}
The proof of this Lemma will be given in Part~II of this series. It is a consequence of the action of the Jucys-Murphy elements of the Brauer algebra and the  classification theorem of indecomposable summands in $V^{\otimes d}$ from \cite{CH}. The proof is an inductive argument.
\end{proof}

\subsection{The Cartan matrix}
We apply the results so far to deduce the symmetry of the Cartan matrix.

\begin{prop}
\label{Cartanmatrix}
Consider $G=\OSPrn$  for fixed $m,n$. The Cartan matrix of $\cF$ is symmetric, i.e. for any $\la,\mu\in X^+(G)$ we have an equality of multiplicities of irreducible modules in a Jordan-H\"older series
\begin{eqnarray}
\label{PL}
[P(\la):L(\mu)]=[P(\mu):L(\la)],
\end{eqnarray}
and therefore
    $\op{dim}\op{Hom}_\cF(P(\la),P(\mu))
    =\op{dim}\op{Hom}_\cF(P(\mu),P(\la))$.
\end{prop}
\begin{proof}
We first claim the analogous formulas for $\cF'$. So given $\la,\mu\in X^+(\mg)$, the multiplicity $[P^\mg(\la):L^\mg(\mu)]$ is the coefficient of the class of $L^\mg(\mu)$ when we express the class of $[P^\mg(\la)]$ in terms of the classes of the irreducible modules of $\cF'$ in the Grothendieck group of $\cF'$. Now by \cite{GS2} we have another class of linearly independent elements in the Grothendieck group, namely the Euler-characteristics $\cE^\mg(\nu)$, where $\nu$ runs through all tailless elements in $X^+(\mg)$ and the classes $[P^\mg(\la)]$ are all in the $\mZ$-lattice spanned by these, with coefficients denoted by $(P^\mg(\la):\cE^\mg(\nu))$, see \cite[Theorem 1]{GS2}. Hence 
\begin{eqnarray*}
[P^\mg(\la): L^\mg(\mu)]&=&\sum_{\nu}(P^\mg(\la):\cE^\mg(\nu))[\cE^\mg(\nu):L^\mg(\mu)]\\
&=&\sum_{\nu}[\cE^\mg(\nu):L^\mg(\la)][\cE^\mg(\nu):L^\mg(\mu)]\\
&=&[P^\mg(\mu):L^\mg(\la)],
\end{eqnarray*}
where the second equality is the BGG-reciprocity, \cite[Theorem 1]{GS2}, and the third equality holds then by symmetry. Hence the analogue of \eqref{PL} for $\cF'$ holds. Now $\op{dim}\op{Hom}_{\cF'}(P^\mg(\la),L^\mg(\la))=1$, since $L^\mg(\la)$ is a highest weight module, and therefore 
$\op{dim}\op{Hom}_{\cF'}(P^\mg(\la),P^\mg(\mu))=[P^\mg(\mu):L^\mg(\la)]$. Hence the proposition holds for $\cF'$. (Alternatively one could use that $P^\mg(\la)\cong I^\mg(\la)$ and apply the usual simple preserving duality on $\cF'$).  Proposition~\ref{erste} implies that $\op{dim}\op{End}_{\cF}(L)=1$ for any irreducible object in $\cF$. Then the statement from the proposition follows directly from the statement for $\cF'$ and the formulas for the dimensions of homomorphism spaces (Lemma~\ref{lemplusminus} and Proposition~\ref{blocksO}). 
\end{proof}

\subsection{Hook partitions}
Let still $G=\OSP(r|2n)$ for $r=2m+1$ or $r=2m$ and recall (from Propositions~\ref{def:labelssimples1} and~\ref{def:labelssimples2} and~\eqref{dominance}) the labelling sets $X^+(G)$ respectively $X^+(\mg)$ for the isomorphism classes of irreducible objects in $\cF$ and $\cF'$.

A different commonly used labelling of the simple modules in $\cF'$ is given by hook partitions,  see e.g. \cite{CW}. To make the connection,  recall that a {\it partition}, denoted\footnote{We chose this notation to distinguish partitions from integral weights.} by $\pga$,  is a weakly decreasing sequence of non-negative integers, $\pga=(\pga_1\geq\pga_2\geq\cdots)$. We denote by $\pga^t$ its transpose partition, i.e. $\pga^t_i=|\{k\mid \la_k\geq i\}|$.  A partition $\pga$ is called {\it $(n,m)$-hook} if $\pga_{n+1}\leq m$. The partition $\pga=(8,7,6,3,3,1)$ is for instance $(5,7)$-hook and $(5,5)$-hook, but not $(2,5)$-hook, see Figure~\ref{shifteddiag}. Note that the empty partition $\varnothing$ is $(n,m)$-hook for any $n,m\geq 0$.  and corresponds to the zero weight via the following dictionary. 

\begin{definition}
\label{nmhook}
Given an $(n,m)$-hook partition $\pga$ we associate {\it weights} 
\begin{eqnarray*}
\op{wt}(\pga) \in X^+(\osp(2m+1|2n)), &\text{respectively}& \op{wt}(\pga)\in X^+(\osp(2m|2n))
\end{eqnarray*}
defined, via \eqref{oje}, as follows, (with $1\leq i\leq m$, $1\leq j\leq n$):
\begin{itemize}
\item in the odd case $\op{wt}(\pga)=(a_1,a_2,\ldots, a_m\mid b_1,b_2,\ldots, b_n)-\rho$, where 
\begin{eqnarray*}
b_j=\op{max}\left\{\pga_j-j-\frac{\de}{2}+1,\frac{1}{2}\right\}&\text{and}&
a_i=\op{max}\left\{\pga^t_i-i+\frac{\de}{2},-\frac{1}{2}\right\}, 
\end{eqnarray*}
\item in the even case $\op{wt}(\pga)=(a_1,a_2,\ldots, a_m\mid b_1,b_2,\ldots, b_n)-\rho$, where 
\begin{eqnarray*}
b_j=\op{max}\left\{\pga_j-j-\frac{\de}{2}+1,0\right\}&\text{and}&
a_i=\op{max}\left\{\pga^t_i-i+\frac{\de}{2},0\right\}.
\end{eqnarray*}
\end{itemize}
\end{definition}

The $a_i$ and $b_j$ give a different way to describe $(n,m)$-hook partitions by encoding the number of boxes below and to the right of the $\lfloor\frac{\de}{2}\rfloor$-shifted diagonal (which we just call {\it diagonal}). For example let $\la=(8,7,6,3,3,1)$. Consider it as a hook partition, for instance as $(5,7)$-hook respectively $(5,5)$-hook, and mark the diagonal (it intersects the inflexion point of the hook and the boxes on the diagonal have content $\frac{\de}{2}+\frac{1}{2}$ respectively $\frac{\de}{2}+1$, where the content is the row minus the column number of the box).
\intextsep0.25em
\begin{figure}[h]
\centering
\begin{tikzpicture}[thick,scale=0.3]
\node at (5,1) {$\frac{\de}{2}=m-n=2$};
\draw[very thick] (0,-10) -- (0,0) -- (12,0);
\draw[very thick] (7,-9) -- (7,-5) -- (11,-5);
\draw[<->] (0,-8) -- node[above]{m=7} (7,-8);
\draw[<->] (10,0) -- node[left]{n=5} (10,-5);

\draw[step=1] (0,-3) grid +(6,3);
\draw (6,0) rectangle +(1,-1);
\draw (7,0) rectangle +(1,-1);
\draw (6,-1) rectangle +(1,-1);
\draw[step=1] (0,-5) grid +(3,5);
\draw (0,-5) rectangle +(1,-1);

\begin{scope}[xshift=2cm]
\draw[blue,very thick,dotted] (0,0) -- (5,-5);
\end{scope}
\end{tikzpicture}
\;\;
\begin{tikzpicture}[thick,scale=0.3]

\node at (4,1) {$\frac{\de}{2}=0$};
\draw[very thick] (0,-10) -- (0,0) -- (12,0);
\draw[very thick] (5,-9) -- (5,-5) -- (11,-5);
\draw[<->] (0,-8) -- node[above]{m=5} (5,-8);
\draw[<->] (10,0) -- node[left]{n=5} (10,-5);

\draw[step=1] (0,-3) grid +(6,3);
\draw (6,0) rectangle +(1,-1);
\draw (7,0) rectangle +(1,-1);
\draw (6,-1) rectangle +(1,-1);
\draw[step=1] (0,-5) grid +(3,5);
\draw (0,-5) rectangle +(1,-1);

\draw[blue,very thick, dotted] (0,0) -- (5,-5);
\end{tikzpicture}
\;\;
\begin{tikzpicture}[thick,scale=0.3]
\node at (6,1) {$\frac{\de}{2}=m-n+\frac{1}{2}=\frac{5}{2}$};
\draw[very thick] (0,-10) -- (0,0) -- (12,0);
\draw[very thick] (7,-9) -- (7,-5) -- (11,-5);
\draw[<->] (0,-8) -- node[above]{m=7} (7,-8);
\draw[<->] (10,0) -- node[left]{n=5} (10,-5);

\draw[step=1] (0,-3) grid +(6,3);
\draw (6,0) rectangle +(1,-1);
\draw (7,0) rectangle +(1,-1);
\draw (6,-1) rectangle +(1,-1);
\draw[step=1] (0,-5) grid +(3,5);
\draw (0,-5) rectangle +(1,-1);

\begin{scope}[xshift=2cm]
\draw[blue,very thick,dotted] (0,0) -- (5,-5);
\end{scope}
\end{tikzpicture}
\caption{The translation between weights and hook partitions.}
\label{shifteddiag}
\end{figure}
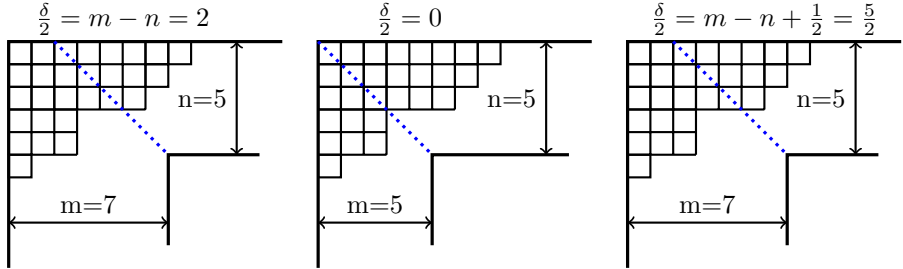

\begin{itemize}
\item In the even case this implies that $a_i$ counts the number of boxes in column $i$ strictly below the diagonal, while $b_j$ counts the number of boxes in row $j$ on and to the right of this diagonal. In the first two cases of Figure~\ref{shifteddiag} we get $\mathbf{a}=(7,5,4,1,0,0,0)$, respectively $\mathbf{a}=(5,3,2,0,0)$, and on the other hand $\mathbf{b}=(6,4,2,0,0)$, respectively $\mathbf{b}=(8,6,4,0,0)$.
\item In the odd case this implies that $a_i$ counts the number of boxes in column $i$ on and below the diagonal minus $\tfrac{1}{2}$. On the other hand $b_j$ counts the number of boxes in row $j$ strictly to the right of this diagonal minus $\tfrac{1}{2}$ and takes the absolute value of this expression. In the third case above in Figure~\ref{shifteddiag} this gives $\mathbf{a}=(\tfrac{15}{2},\tfrac{11}{2},\tfrac{9}{2},\tfrac{3}{2},\tfrac{1}{2},-\tfrac{1}{2},-\tfrac{1}{2})$ and $\mathbf{b}=(\tfrac{11}{2},\tfrac{7}{2},\tfrac{3}{2},\tfrac{1}{2},\tfrac{1}{2})$.
\end{itemize}

(Note that we also count the numbers of boxes which can be put in the region between the marked diagonal and the partition, i.e. above or to the left of the diagram depending on the given diagonal.)

\begin{definition}
A {\it signed $(n,m)$-hook partition} is an $(n,m)$-hook partition $\pga$ with $\pga_{n+1}\geq m$ or a pair $(\pga,\epsilon)$ of an $(n,m)$-hook partition with $\pga_{n+1}<m$ and a sign $\epsilon\in\{\pm\}.$ 
\end{definition}
The following easy identification allows us to work with hook partitions plus signs (in the odd case) respectively with signed hook partitions (in the even case) instead of dominant weights. 

\begin{lemma}
\label{AB}
The assignments $\pga\mapsto \op{wt}(\pga) $ defines a bijection 
\begin{eqnarray}
\label{AB1}
\Psi=\Psi_{2m+1,2n}:\;\left\{(n,m)-\text{hook partitions}\right\}\times\mZ/2\mZ
&\stackrel{1:1}{\leftrightarrow}&
X^+(G)\\
(\pga,\pm)&\mapsto&( \op{wt}(\pga),\pm)\quad\nonumber
\end{eqnarray}
in case $G=\OSP(2m+1|2n)$, and, in case $G=\OSP(2m|2n)$, a bijection
\begin{eqnarray*}
\Psi=\Psi_{2m,2n}:\; \left\{(n,m)-\text{signed hook partitions} \right\}
&\stackrel{1:1}{\leftrightarrow}&
X^+(G)\\
\pga&\mapsto& \op{wt}(\pga)\\
(\pga,\pm)&\mapsto&(\op{wt}(\pga),\pm)).\quad
\end{eqnarray*}
\end{lemma}

\begin{notation} \label{not:part}
In either case: given $\la\in X^+(G)$, we denote by $\pla$ the unique hook partition such that $\pla$ respectively  $(\pla,\pm)$ is the preimage of $\la$ under $\Psi$ and call it the {\it underlying hook partition}.
\end{notation}

\begin{proof}[Proof of Lemma~\ref{AB}]
Take an $(n,m)$-hook partition $\pga$. To see that the maps are well-defined it suffices to show that $\op{wt}(\pga) $ us a dominant weight for $\mg$ (since we can clearly ignore the signs). 

Let us first consider the case $\Psi_{2m,2n}$. Since $\pga$ is a partition we have $a_{i+1}<a_i$ and $b_{j+1}<b_j$ whenever they are defined and non-zero.
For the map to be well-defined it remains to show that the number of zero $a$'s is equal or one larger than the number of zero $b$'s.

{\it Claim:} For $s\leq \op{min}\{m,n\}$ we have $a_{m-s}>0$ implies $b_{n-s}>0$.  
If $b_{n-s}= 0$ then $\pga_{n-s}-m+n-n+s+1\leq 0$, hence $\pga_{n-s}\leq m-s-1$ and so $\pga$ has at most $n-s-1$ rows of length $m-s$. This means $\pga^t_{m-s}\leq n-s-1$ and thus $a_{m-s}=\pga^t_{m-s}-n+m-m+s\leq n-s-1-n+s=-1$ which is a contradiction and the claim follows.
This shows that there are at least as many zero $a$'s as $b$'s.  It suffices now to show that  $a_{m-r}=0$ forces $b_{n-r+1}=0$. So assume $a_{m-r}=0$. Since $a_{m-n-1}=\pga^t_{m-n-1}-n+m-m+n-1=\pga^t_{m-n-1}+1>0$ we see that $a_{m-r}=0$ implies $r\leq n$ and so $b_{n-r+1}$ must exist. If $b_{n-r+1}=\pga_{n-r+1}-m+n-n+r-1+1>0$,then $\pga_{n-r+1}>m-r$ which implies $\pga^t_{m-r}\leq n-r+1$ and therefore $a_{m-r}=\pga^t_{m-r}-n+m-m+r\geq n-r+1-n+r\geq 1$ which is a contradiction. Hence the map is well-defined and obviously injective.  

 Clearly the weights in the image satisfy $a_m\geq 0$. For the description of the image it suffices to show that if $({\bf a}\,|\,{\bf b})\in X^+(\mg)$ satisfies the dominance condition from Definition~\ref{dominance} with $a_m>0$ in case (i) then  it comes from a hook partition. It is enough to see it defines a partition, since $b_j$ is only defined for $1\leq j\leq n$ and $a_i$ for $1\leq i\leq m$, hence if it is a partition it must be $(n,m)$-hook. For that it suffices to see that $a_i\not=0$ with $i=\tfrac{\de}{2}+k$ for some $k$ implies $b_{k}\geq 1$. Write $i=m-s$ then this is equivalent to ($a_{m-s}\not=0$ implies $b_{n-s}\geq 1$), since $k=i-\tfrac{d}{2}=i-m+n=n-s$. But this was exactly the claim above. The arguments for $\Psi_{2m+1,2n}$ are analogous, but the last step is even easier here.
\end{proof}

\begin{definition}
\label{def:tail}
The tail length $\op{tail}(\la)$ of $\la\in X^+(G)$ or equivalently of the underlying hook partition, is equal to $\op{min}\{m,n\}-d$, where $d$ is the number of boxes on the diagonal of the hook partition.
\end{definition}

\begin{remark}
It is easy to check that this notion agrees with the notion of tail length from Definition~\ref{tailg}. Note that $\op{tail}(\la)$ counts the number of missing boxes on the diagonal of the hook partition, in particular, it is maximal possible for the empty partition, i.e. the zero weight.
\end{remark}

We will present now a new (and more convenient) way of encoding dominant weights and the labeling set of irreducible finite-dimensional representations of $G$ in terms of {\it diagrammatic weights}. This is  in the spirit of \cite{BS4} built on the combinatorics introduced in \cite{ESperv}.

\section{Diagrammatics: Generalities}
We attach now a certain diagrammatic weight to each simple object in $\cF(G)$. This will allow us to develop a diagrammatic description of the morphism spaces between indecomposable projective objects in the corresponding categories $\cF(G)$. 

\subsection{Diagrammatic weights attached to \texorpdfstring{$X^+(G)$}{X(G)}}
\label{L}
To establish the combinatorics consider the non-negative number line $\cL$ and call its integral points {\it vertices}. 
\begin{definition} \label{def:diagrammaticweight}
An {\it (infinite) diagrammatic weight} or just a {\it diagrammatic weight} $\la$ is a diagram obtained by labelling each of the vertices by exactly one of the symbols $\times$ (cross), $\circ$ (nought), $\down$ (down), $\up$ (up); for the position zero we do not distinguish the labels $\up$ and $\down$ and use instead the label $\Diamond$. The vertices labelled $\circ$ or $\times$ are called {\it core symbols} and the diagram obtained from $\la$ by removing all symbols $\up$, $\down$ and $\diamond$ is called its {\it core diagram}. 
\end{definition}
For a diagrammatic weight $\la$ we denote by $\# \times(\la)$, $\# \circ(\la)$, $\# \up(\la)$, $\# \down(\la)$ the number of crosses, noughts, downs and ups respectively occurring in $\la$. 
\begin{definition}
A diagrammatic weight $\la$ is called 
\begin{itemize}
\item {\it finite} if $\# \down(\la)+\# \up(\la)+\# \times(\la)<\infty$, and 
\item {\it of hook partition type} if  $\# \down(\la)+\# \circ(\la)+\# \times(\la)<\infty$, and
\item {\it of super type}  if   $\# \up(\la)+\# \circ(\la)+\# \times(\la)<\infty$.
\end{itemize}
\end{definition}
Hence a finite weight has only noughts far to the right, a weight of super type has only $\down$'s far to the right, and a weight of hook partition type has only $\up$'s far to the right. For instance, consider the diagrammatic weights 
\begin{equation}
\label{questionmarks}
\down\;\circ\;\times \;\up\;\times\;\up\;\up\;\up \;\down \;\up\;\up\;\up\;\down\;\down\;\up\;\down \;\up\; ?\:\,?\,\:?\,\:?\cdots
\end{equation}
where the ?'s and the dots indicate either only $\circ$'s, only $\up$'s or only $\down$'s respectively. Then the resulting three weights $\la_{\op{fin}}$, $\la_{\op{hook}}$,  and $\la_{\op{super}}$ are finite, hook partition type or super type respectively. 

\begin{definition}
\label{diagramblocks}
Two diagrammatic weights $\la$ and $\mu$  with a finite total number of $\up$'s are {\it linked} or {\it in the same block} if their core diagrams coincide, and in addition, in case there is no $\Diamond$ the parity of $\up$'s agree, in formulas
\begin{eqnarray*}
\# \up(\la)\equiv \# \up(\mu)\;\op{mod} 2.
\end{eqnarray*}
\end{definition}

We now assign to each  $(n,m)$-hook partition $\pga$ a diagrammatic weight. 

\begin{definition}
For any partition $\pga$ and $\delta=r-2n$ set
\begin{eqnarray}
\label{Sla} 
\cS(\pga)&=&\left(\frac{\de}{2}+i-\pga_i-1\right)_{i\geq 1}.
\end{eqnarray}
\end{definition}

This is a strictly increasing sequence of half-integers (i.e. from  $\mZ+\tfrac{1}{2}$) if $r$ is odd, and of integers in case $r$ is even. In case $r$ is odd we identify the vertices of $\cL$ order-preserving with $\mZ_{\geq0}+\tfrac{1}{2}$. That means we have then vertices $\tfrac{1}{2},\tfrac{3}{2},\tfrac{5}{2},\ldots$ etc. In case $r$ is even, we identify the vertices of $\cL$ order-preserving with $\mZ_{\geq0}$.

\begin{definition}
\label{lainfty}
To the sequence $\cS(\pga)$ we then assign an infinite diagrammatic weight ${\ulcorner{\!\gamma}}^\infty$ 
by attaching to the vertex $p$ the label
\begin{eqnarray}\label{dict}
\left\{
\begin{array}{cl}
\circ&\text{if neither $p$ nor $-p$ occurs in $\cS(\pga)$},\\
{\scriptstyle\down} &\text{if $-p$, but not $p$, occurs in $\cS(\pga)$,}\\
{\scriptstyle\up} &\text{if $p$, but not $-p$, occurs in $\cS(\pga)$,}\\
\times&\text{if both, $-p\not=p$ occur in $\cS(\pga)$,}\\
\Diamond&\text{if $p=0$ occurs in $\cS(\pga)$.}
\end{array}\right.
\end{eqnarray}
\end{definition}
Note that there are only finitely many labels different from $\up$, hence these resulting diagrammatic weights are all of hook partition type. Moreover, the zero position can only have labels $\circ$ or $\Diamond$.

\begin{example}
The empty partition gives in case of odd $r$ the following diagrammatic weights
\begin{eqnarray}
\label{empty1}
\phantom{xxx}
\begin{cases}
\begin{tikzpicture}[anchorbase,thick,scale=1.6]
\begin{scope}
\node at (0,0) {$\circ$};
\node at (0,.3) {$\scriptstyle \frac{1}{2}$};
\node at (.5,0) {$\cdots$};
\node at (1,0) {$\circ$};
\draw[thin] (-.1,-.1) to +(0,-.1) to +(1.2,-.1) to +(1.2,0);
\node at (.5,-.4) {$\scriptstyle m-n$};
\node at (1.5,.3) {$\scriptstyle \frac{\de}{2}$};
\node at (1.5,0) {$\up$};
\node at (2,0) {$\up$};
\node at (2.5,0) {$\cdots$};
\node at (3,0) {$\up$};
\node at (3.5,0) {$\up$};
\draw[thin] (1.4,-.1) to +(0,-.1) to +(2.2,-.1) to +(2.2,0);
\node at (2.5,-.4) {$\scriptstyle 2n$};
\node at (4,0) {$\owedge$};
\node at (4.5,0) {$\owedge$};
\node at (5,0) {$\cdots$};
\end{scope}
\end{tikzpicture}
&{\text{if $\de>0$,}}\\
\begin{tikzpicture}[anchorbase,thick,scale=1.6]
\begin{scope}
\node at (0,0) {$\times$};
\node at (0,.3) {$\scriptstyle \frac{1}{2}$};
\node at (.5,0) {$\cdots$};
\node at (1,0) {$\times$};
\draw[thin] (-.1,-.1) to +(0,-.1) to +(1.2,-.1) to +(1.2,0);
\node at (.5,-.4) {$\scriptstyle n-m$};
\node at (0.95,.3) {$\scriptstyle -\frac{\de}{2}$};
\node at (1.5,0) {$\up$};
\node at (2,0) {$\up$};
\node at (2.5,0) {$\cdots$};
\node at (3,0) {$\up$};
\node at (3.5,0) {$\up$};
\draw[thin] (1.4,-.1) to +(0,-.1) to +(2.2,-.1) to +(2.2,0);
\node at (2.5,-.4) {$\scriptstyle 2m$};
\node at (4,0) {$\owedge$};
\node at (4.5,0) {$\owedge$};
\node at (5,0) {$\cdots$};
\end{scope}
\end{tikzpicture}
&{\text{if $\de<0$.}}
\end{cases}
\end{eqnarray}
(The circles around the $\up$ in $\owedge$ should be ignored for the moment. They will play an important role later).

In case $r$ is even, the empty partition gives the following diagrammatic weights
\begin{eqnarray}
\label{empty2}
\phantom{xxx}
\begin{cases}
\begin{tikzpicture}[anchorbase,thick,scale=1.6]
\begin{scope}
\node at (0,0) {$\circ$};
\node at (0,.3) {$\scriptstyle 0$};
\node at (.5,0) {$\cdots$};
\node at (1,0) {$\circ$};
\draw[thin] (-.1,-.1) to +(0,-.1) to +(1.2,-.1) to +(1.2,0);
\node at (.5,-.4) {$\scriptstyle m-n$};
\node at (1.5,.3) {$\scriptstyle \frac{\de}{2}$};
\node at (1.5,0) {$\up$};
\node at (2,0) {$\up$};
\node at (2.5,0) {$\cdots$};
\node at (3,0) {$\up$};
\draw[thin] (1.4,-.1) to +(0,-.1) to +(1.7,-.1) to +(1.7,0);
\node at (2.25,-.4) {$\scriptstyle 2n$};
\node at (3.5,0) {$\owedge$};
\node at (4,0) {$\cdots$};
\end{scope}
\end{tikzpicture}
&{\text{if $\de \geq 0$,}}\quad\\
\begin{tikzpicture}[anchorbase,thick,scale=1.6]
\begin{scope}
\node at (0,0) {$\Diamond$};
\node at (0,.3) {$\scriptstyle 0$};
\node at (.5,0) {$\times$};
\node at (1,0) {$\cdots$};
\node at (1.5,0) {$\times$};
\draw[thin] (.4,-.1) to +(0,-.1) to +(1.2,-.1) to +(1.2,0);
\node at (1,-.4) {$\scriptstyle n-m$};
\node at (1.4,.3) {$\scriptstyle -\frac{\de}{2}$};
\node at (2,0) {$\up$};
\node at (2.5,0) {$\cdots$};
\node at (3,0) {$\up$};
\draw[thin] (1.9,-.1) to +(0,-.1) to +(1.2,-.1) to +(1.2,0);
\node at (2.25,-.4) {$\scriptstyle 2m-1$};
\node at (3.5,0) {$\owedge$};
\node at (4,0) {$\cdots$};
\end{scope}
\end{tikzpicture}
&{\text{if $\de<0$,}}\quad
\end{cases}
\end{eqnarray}
Again, the circle around the $\up$ in $\owedge$ should be ignored for the moment.  
\end{example}

We refer to Section~\ref{sec:osp32} for more examples.

\begin{lemma}
\label{posneg}
Let $\la\in X^+(G)$. We have $\cS(\pla)_i<0$ (respectively $\cS(\pla)_i\geq0$) in \eqref{Sla} iff the $i$-th row  in the underlying hook partition, in the sense of Notation \ref{not:part}, ends above or on (respectively strictly below) the $\tfrac{\de}{2}$-shifted diagonal. 
\end{lemma}
\begin{proof}
Note that $\cS(\pla)_i< 0$ iff $\frac{\de}{2}+i-\pla_i-1< 0$ or equivalently $\pla_i > i+\frac{\de}{2}-1$. 
\end{proof}

The tail length of $\la\in X^+(G)$ can be expressed again combinatorially.
\begin{corollary}
\label{tail}
Let $\lambda \in X^+(G)$. The tail length of $\la$ equals $\op{tail}(\la)=n-s$ where $s=\# \down(\la^\infty)+\#\times(\la^\infty)$. 
\end{corollary}
\begin{proof}
If $m\geq n$ then there is a box on the diagonal in row $i$ iff $\cS(\pla)_i<0$. This implies that there are exactly $s$ boxes on the shifted diagonal, hence $\op{tail}(\la)=n-s$. If on the other hand $m<n$ then again $s$ is the number of rows that end above or on the shifter diagonal, but we have to subtract the first $n-m$ rows, thus there are $s-(n-m)$ boxes on the diagonal, hence $\op{tail}(\la)=m-s+(n-m)=n-s$. 
\end{proof}

The following characterizes the weights with non-zero tail in the even case.
\begin{corollary}
\label{cornotail}
Assume $r=2m$ and let $\la\in X^+(\mg)$ in the notation from~\eqref{oje}.  Consider the underlying $(n,m)$-hook partition ${\pla}$ and the diagrammatic weight $\pla^\infty $ given by $\cS(\pla)$. Then the following are equivalent:
\begin{eqnarray*}
a_m>0\; \Leftrightarrow \; {\pla}_{n+1}=m \; \Leftrightarrow \; \op{Ind}^{\mg, G}_{\mg, G'}L^\mg(\la) \text{ is irreducible} \; \Leftrightarrow \; {\cS(\pla)}_{n+1}=0. 
\end{eqnarray*}
Moreover, in this case the associated diagrammatic weight $\pla^\infty$ has label $\Diamond$ at position zero, and $\op{tail}(\la)=0$.
\end{corollary}
\begin{proof}
Obviously $a_m>0$ is equivalent to $\pla_{n+1}=m$  by Definition~\ref{nmhook}, and hence to $\op{tail}(\la)=0$ by definition. It is moreover equivalent to $\op{Ind}^{\mg, G}_{\mg, G'}L^\mg(\la) $ being irreducible by Proposition~\ref{erste}.  On the other hand ${\pla}_{n+1}=m$ if and only if $\cS(\pla)_{n+1}=m-n+n+1-\pla_{n+1}-1=m-\pla_{n+1}=0$ (which then obviously causes a $\Diamond$ at position zero).
\end{proof}

\subsection{Cup diagrams}

Given a diagrammatic weight $\la$ which is finite, of hook partition type or of super type, we like to assign a unique cup diagram. For this we say that two vertices in a diagrammatic weight are {\it neighboured} if they are only separated by vertices with labels $\circ$'s and $\times$'s. 
\begin{definition}
\label{decoratedcups}
The \emph{infinite decorated cup diagram} or just {\it cup diagram} $\underline{\la}$ associated with a diagrammatic weight $\la$ (finite, hook partition type or super type) is obtained by applying the following steps in order.
\begin{enumerate}[(Cup-1)]
\item \label{cup1} If the diagrammatic weight $\lambda$ contains a $\Diamond$ we change it into an $\up$ or $\down$ in such a way that the resulting number of $\up$'s is odd or infinite.
\item \label{cup2} First connect neighboured vertices labelled $\down\up$ successively by a cup, i.e. an arc forming a cup blow the labels, (ignoring already joint vertices) as long as possible. (The result is independent of the order in which the connections are made).
\item \label{cup3} Attach to each remaining $\down$ a vertical ray.
\item \label{cup4} Connect from left to right pairs of two neighboured $\up$'s by cups (viewing $\Diamond$ as $\up$).
\item \label{cup5} If a single $\up$ remains, attach a vertical ray.
\item \label{cup6} Put a decoration $\bullet$ on each cup created in \ref{cup4} and each ray created in \ref{cup5}.
\item \label{cup7} Finally delete all labels at vertices.
\end{enumerate}
The arcs for the connections should always be drawn without intersections. Moreover two cup diagrams are considered the same if there is a bijection between the set of arcs respecting the connected vertices.
\end{definition}

\begin{remark}
Observe that the conditions finite, of hook partition type and of super type make sure that the algorithm producing the cup diagram is well-defined. In case of hook partition types the steps (Cup-\ref{cup3}) and (Cup-\ref{cup5}) can be removed and the diagram will never have dotted or undotted rays, but infinitely many dotted cups. In case the diagram is of super type it will have only rays far to the right.  In case it is of finite type it has only finitely many cups and rays.
\end{remark}

\begin{exs}
The  three diagrammatic weights $\la_{\op{fin}}$, $\la_{\op{hook}}$, and $\la_{\op{super}}$ from \eqref{questionmarks} provide the following three cup diagrams. 
\begin{equation}
\label{thecup}
\scriptstyle
\begin{tikzpicture}[thick,scale=0.8]
\node at (0,0) {$\down$};
\node at (.5,0) {$\circ$};
\node at (1,0) {$\times$};
\node at (1.5,0) {$\up$};
\node at (2,0) {$\times$};
\node at (2.5,0) {$\up$};
\node at (3,0) {$\up$};
\node at (3.5,0) {$\up$};
\node at (4,0) {$\down$};
\node at (4.5,0) {$\up$};
\node at (5,0) {$\up$};
\node at (5.5,0) {$\up$};
\node at (6,0) {$\down$};
\node at (6.5,0) {$\down$};
\node at (7,0) {$\up$};
\node at (7.5,0) {$\down$};
\node at (8,0) {$\up$};
\node at (8.5,0) {$?$};
\node at (9,0) {$?$};
\node at (9.5,0) {$?$};
\node at (10,0) {$?$};
\node at (10.75,-0.25) {$\cdots$};

\begin{scope}[yshift=-1cm]
\node at (-1,0) {$\underline\la_{\op{fin}}:$};
\draw (0,0) .. controls +(0,-.75) and +(0,-.75) .. +(1.5,0);
\node at (.5,0) {$\circ$};
\node at (1,0) {$\times$};
\node at (2,0) {$\times$};
\draw (2.5,0) .. controls +(0,-.5) and +(0,-.5) .. +(.5,0);
\fill (2.75,-.365) circle(2pt);
\draw (3.5,0) .. controls +(0,-.75) and +(0,-.75) .. +(1.5,0);
\fill (4.25,-.55) circle(2pt);
\draw (4,0) .. controls +(0,-.5) and +(0,-.5) .. +(.5,0);
\fill (5.5,-.5) circle(2pt);
\draw (5.5,0) -- +(0,-1);
\draw (6,0) -- +(0,-1);
\draw (6.5,0) .. controls +(0,-.5) and +(0,-.5) .. +(.5,0);
\draw (7.5,0) .. controls +(0,-.5) and +(0,-.5) .. +(.5,0);
\node at (8.5,0) {$\circ$};
\node at (9,0) {$\circ$};
\node at (9.5,0) {$\circ$};
\node at (10,0) {$\circ$};
\node at (10.75,-0.25) {$\cdots$};
\end{scope}

\begin{scope}[yshift=-4cm]
\node at (-1,0) {$\underline\la_{\op{super}}:$};
\draw (0,0) .. controls +(0,-.75) and +(0,-.75) .. +(1.5,0);
\node at (.5,0) {$\circ$};
\node at (1,0) {$\times$};
\node at (2,0) {$\times$};
\draw (2.5,0) .. controls +(0,-.5) and +(0,-.5) .. +(.5,0);
\fill (2.75,-.365) circle(2pt);
\draw (3.5,0) .. controls +(0,-.75) and +(0,-.75) .. +(1.5,0);
\fill (4.25,-.55) circle(2pt);
\draw (4,0) .. controls +(0,-.5) and +(0,-.5) .. +(.5,0);
\fill (5.5,-.5) circle(2pt);
\draw (5.5,0) -- +(0,-1);
\draw (6,0) -- +(0,-1);
\draw (6.5,0) .. controls +(0,-.5) and +(0,-.5) .. +(.5,0);
\draw (7.5,0) .. controls +(0,-.5) and +(0,-.5) .. +(.5,0);
\draw (8.5,0) -- +(0,-1);
\draw (9,0) -- +(0,-1);
\draw (9.5,0) -- +(0,-1);
\draw (10,0) -- +(0,-1);
\node at (10.75,-0.25) {$\cdots$};
\end{scope}

\begin{scope}[yshift=-2.5cm]
\node at (-1,0) {$\underline\la_{\op{hook}}:$};
\draw (0,0) .. controls +(0,-.75) and +(0,-.75) .. +(1.5,0);
\node at (.5,0) {$\circ$};
\node at (1,0) {$\times$};
\node at (2,0) {$\times$};
\draw (2.5,0) .. controls +(0,-.5) and +(0,-.5) .. +(.5,0);
\fill (2.75,-.365) circle(2pt);
\draw (3.5,0) .. controls +(0,-.75) and +(0,-.75) .. +(1.5,0);
\fill (4.25,-.55) circle(2pt);
\draw (4,0) .. controls +(0,-.5) and +(0,-.5) .. +(.5,0);
\draw (5.5,0) .. controls +(0,-1.5) and +(0,-1.5) .. +(3.5,0);
\fill (7.25,-1.15) circle(2pt);
\draw (6,0) .. controls +(0,-1) and +(0,-1) .. +(2.5,0);
\draw (6.5,0) .. controls +(0,-.5) and +(0,-.5) .. +(.5,0);
\draw (7.5,0) .. controls +(0,-.5) and +(0,-.5) .. +(.5,0);
\draw (9.5,0) .. controls +(0,-.5) and +(0,-.5) .. +(.5,0);
\fill (9.75,-.365) circle(2pt);
\node at (10.75,-0.25) {$\cdots$};
\end{scope}
\end{tikzpicture}
\end{equation}
\end{exs}

The empty partition gives always an infinite cup diagram (for the diagrammatic weight see \eqref{empty1} and \eqref{empty2}): In the case of $G=\OSP(2m+1|2n)$ we have
\begin{eqnarray}
\label{emptycupoddcase}
&&\begin{cases}
\begin{array}[t]{lll}
\begin{tikzpicture}[thick]
\begin{scope}
\node at (0,0) {$\circ$};
\node at (0,.4) {$\scriptstyle \frac{1}{2}$};
\node at (.5,0) {$\cdots$};
\node at (1,0) {$\circ$};
\node at (1.5,0) {$\circ$};
\node at (2,.4) {$\scriptstyle \frac{\de}{2}$};
\draw (2,0) .. controls +(0,-.5) and +(0,-.5) .. +(.5,0);
\fill (2.25,-.365) circle(2pt);
\draw (3,0) .. controls +(0,-.5) and +(0,-.5) .. +(.5,0);
\fill (3.25,-.365) circle(2pt);
\draw (4,0) .. controls +(0,-.5) and +(0,-.5) .. +(.5,0);
\fill (4.25,-.365) circle(2pt);
\draw (5,0) .. controls +(0,-.5) and +(0,-.5) .. +(.5,0);
\fill (5.25,-.365) circle(2pt);
\node at (6.25,0) {$\cdots$};
\end{scope}
\end{tikzpicture}
&&{\text{if $\de>0$,}}\\
\begin{tikzpicture}[thick]
\begin{scope}
\node at (0,0) {$\times$};
\node at (0,.4) {$\scriptstyle \frac{1}{2}$};
\node at (.5,0) {$\cdots$};
\node at (1,0) {$\times$};
\node at (1.5,0) {$\times$};
\node at (2,.4) {$\scriptstyle -\frac{\de}{2}+1$};
\draw (2,0) .. controls +(0,-.5) and +(0,-.5) .. +(.5,0);
\fill (2.25,-.365) circle(2pt);
\draw (3,0) .. controls +(0,-.5) and +(0,-.5) .. +(.5,0);
\fill (3.25,-.365) circle(2pt);
\draw (4,0) .. controls +(0,-.5) and +(0,-.5) .. +(.5,0);
\fill (4.25,-.365) circle(2pt);
\draw (5,0) .. controls +(0,-.5) and +(0,-.5) .. +(.5,0);
\fill (5.25,-.365) circle(2pt);
\node at (6.25,0) {$\cdots$};
\end{scope}
\end{tikzpicture}
&&{\text{if $\de<0$.}}
\end{array}
\end{cases}
\end{eqnarray}

whereas in the case of $G=\OSP(2m|2n)$ we have
\begin{eqnarray}
\label{emptycupevencase}
&&\begin{cases}
\begin{array}[t]{lll}
\begin{tikzpicture}[thick]
\begin{scope}
\node at (0,0) {$\circ$};
\node at (0,.4) {$\scriptstyle 0$};
\node at (.5,0) {$\cdots$};
\node at (1,0) {$\circ$};
\node at (1.5,0) {$\circ$};
\node at (2,.4) {$\scriptstyle \frac{\de}{2}$};
\draw (2,0) .. controls +(0,-.5) and +(0,-.5) .. +(.5,0);
\fill (2.25,-.365) circle(2pt);
\draw (3,0) .. controls +(0,-.5) and +(0,-.5) .. +(.5,0);
\fill (3.25,-.365) circle(2pt);
\draw (4,0) .. controls +(0,-.5) and +(0,-.5) .. +(.5,0);
\fill (4.25,-.365) circle(2pt);
\draw (5,0) .. controls +(0,-.5) and +(0,-.5) .. +(.5,0);
\fill (5.25,-.365) circle(2pt);
\node at (6.25,0) {$\cdots$};
\end{scope}
\end{tikzpicture}
&&{\text{if $\de \geq 0$}}\\
\begin{tikzpicture}[thick]
\begin{scope}
\node at (0,.4) {$\scriptstyle 0$};
\node at (.5,0) {$\times$};
\node at (1,0) {$\cdots$};
\node at (1.5,0) {$\times$};
\node at (2,0) {$\times$};
\node at (2.5,.4) {$\scriptstyle -\frac{\de}{2}+1$};
\draw (0,0) .. controls +(0,-.5) and +(0,-.5) .. +(2.5,0);
\fill (1.25,-.365) circle(2pt);
\draw (3,0) .. controls +(0,-.5) and +(0,-.5) .. +(.5,0);
\fill (3.25,-.365) circle(2pt);
\draw (4,0) .. controls +(0,-.5) and +(0,-.5) .. +(.5,0);
\fill (4.25,-.365) circle(2pt);
\draw (5,0) .. controls +(0,-.5) and +(0,-.5) .. +(.5,0);
\fill (5.25,-.365) circle(2pt);
\node at (6.25,0) {$\cdots$};
\end{scope}
\end{tikzpicture}
&&{\text{if $\de\leq 0$,}}
\end{array}
\end{cases}
\end{eqnarray}

\begin{remark}\label{weightforc}
Note that, by construction, there might be cups nested inside each other, but such cups cannot be dotted. By construction, there is also never a $\bullet$ to the right of a ray. 
Given any such cup diagram $c$ there is a unique diagrammatic weight $\la$ such that $\underline{\la}=c$. Namely $\la$ is the unique diagrammatic weight such that, when put on top of $c$, the core symbols match and all cups and rays are oriented in the unique degree zero way as displayed in Figure~\ref{oriented}.
\end{remark}

\begin{definition}
\label{defdef}
We call cups or rays with a decoration $\bullet$ {\it dotted} and those without decorations ${\it undotted}$. The total number (possibly infinite) of undotted plus dotted cups in a cup diagram $c$ is called its {\it defect} or {\it atypicality} and denoted $\op{def}(c)$ and define the defect $\op{def}(\la)$ of a diagrammatic weight to be  the defect $\op{def}(\underline{\la})$ of its associated cup diagram. In particular $\op{def}(\pla^\infty) = \infty$ for all $\la \in X^+(G)$.
\end{definition}

\subsection{(Nuclear) circle diagrams}
A pair of compatible cup diagrams can be combined to a circle diagram:

\begin{definition}
Given $\la,\mu,\nu$ diagrammatic weights. We call the ordered pair $(\la,\mu)$ a {\it circle diagram} if $\la$ and $\mu$ have the same core diagrams. We usually denote this circle diagrams by $\underline{\la}\overline{\mu}$ and think of it as a diagram obtained from putting the cup diagram $\underline{\mu}$ upside down on top of the cup diagram $\underline{\la}$. The upside down cups in $\overline{\mu}$ are called \textit{caps} in the following.
\end{definition}

For examples we refer to Figure~\ref{basis23}, where the diagrammatic weight in the middle of each circle diagram should be ignored. The connected components in a circle diagram are (ignoring dots) either {\it lines} or {\it circles}. 


We now introduce the following important set of nuclear circle diagrams
\begin{definition}
\label{nuclear}
Given two diagrammatic weights $\la,\mu$ we call the circle diagram $\un\la\ov\mu$ {\it nuclear} if it contains at least one line which is not propagating. 
\end{definition}

In Figure~\ref{basis23}, the last two circle diagrams (again ignoring the diagrammatic weight in the middle) are nuclear, the others not.

\subsection{Orientations and degree}
Assume $\la$ is a diagrammatic weight and $\underline{\la}$  its associated decorated cup diagram. An {\it  orientation} of $\underline{\la}$ is a diagrammatic weight $\nu$ such that  $\la$ and $\nu$ have the same core diagram and 
if we put $\nu$ on top of $\underline{\la}$ (identifying along the corresponding vertices), then all cups and rays in the resulting diagram are `oriented' in one of the ways displayed in Figure~\ref{oriented}. An oriented infinite decorated cup diagram is such a pair $(\underline{\la},\nu)$, often denoted $\underline{\la}\nu$. 

We usually just draw the cup diagram with the orientation on top and think of it in a topological way.
\smallskip
\begin{figure}[h]
\begin{tikzpicture}[thick,>=angle 60,scale=0.6]
\draw [-] (-4,0) .. controls +(0,-1) and +(0,-1) .. +(1,0) node at
+(0.5,-1.5) {0};
\node at (-4,.2) {$\scriptstyle \boldsymbol{\down}$};
\node at (-3,.2) {$\scriptstyle \boldsymbol{\up}$};
\draw [-] (-2,0) .. controls +(0,-1) and +(0,-1) .. +(1,0) node at
+(0.5,-1.5) {1};
\node at (-2,.2) {$\scriptstyle \boldsymbol{\up}$};
\node at (-1,.2) {$\scriptstyle \boldsymbol{\down}$};
\draw [-] (0,0) .. controls +(0,-1) and +(0,-1) .. +(1,0) node at
+(0.5,-1.5) {0};
\node at (0,.2) {$\scriptstyle \boldsymbol{\Diamond}$};
\node at (1,.2) {$\scriptstyle \boldsymbol{\up}$};
\draw [-] (2,0) .. controls +(0,-1) and +(0,-1) .. +(1,0) node at
+(0.5,-1.5) {1};
\node at (2,.2) {$\scriptstyle \boldsymbol{\Diamond}$};
\node at (3,.2) {$\scriptstyle \boldsymbol{\down}$};
\draw [-] (4,-0.7) .. controls +(0,1) and +(0,1) .. +(1,0) node at
+(0.5,-0.8) {0};
\node at (4,-.9) {$\scriptstyle \boldsymbol{\down}$};
\node at (5,-.9) {$\scriptstyle \boldsymbol{\up}$};
\draw [-] (6,-0.7) .. controls +(0,1) and +(0,1) .. +(1,0)  node at
+(0.5,-0.8) {1};
\node at (6,-.9) {$\scriptstyle \boldsymbol{\up}$};
\node at (7,-.9) {$\scriptstyle \boldsymbol{\down}$};
\draw [-] (8,-0.7) .. controls +(0,1) and +(0,1) .. +(1,0)  node at
+(0.5,-0.8) {0};
\node at (8,-.9) {$\scriptstyle \boldsymbol{\Diamond}$};
\node at (9,-.9) {$\scriptstyle \boldsymbol{\up}$};
\draw [-] (10,-0.7) .. controls +(0,1) and +(0,1) .. +(1,0)  node at
+(0.5,-0.8) {1};
\node at (10,-.9) {$\scriptstyle \boldsymbol{\Diamond}$};
\node at (11,-.9) {$\scriptstyle \boldsymbol{\down}$};
\draw [-] (12,0) -- +(0,-0.7) node at +(0,-1.5) {0};
\node at (12,.2) {$\scriptstyle \boldsymbol{\down}$};
\draw [-] (13,0) -- +(0,-0.7) node at +(0,-1.5) {0};
\node at (13,.2) {$\scriptstyle \boldsymbol{\Diamond}$};
\draw [-] (14,0) -- +(0,-0.7) node at +(0,-1.5) {0};
\node at (14,-.9) {$\scriptstyle \boldsymbol{\down}$};
\draw [-] (15,0) -- +(0,-0.7) node at +(0,-1.5) {0};
\node at (15,-.9) {$\scriptstyle \boldsymbol{\Diamond}$};

\begin{scope}[yshift=-3cm]
\draw [-] (-4,0) .. controls +(0,-1) and +(0,-1) .. +(1,0) node at
+(0.5,-1.5) {0};
\fill (-3.5,-.75) circle(2.5pt);
\node at (-4,.2) {$\scriptstyle \boldsymbol{\up}$};
\node at (-3,.2) {$\scriptstyle \boldsymbol{\up}$};
\draw [-] (-2,0) .. controls +(0,-1) and +(0,-1) .. +(1,0) node at
+(0.5,-1.5) {1};
\fill (-1.5,-.75) circle(2.5pt);
\node at (-2,.2) {$\scriptstyle \boldsymbol{\down}$};
\node at (-1,.2) {$\scriptstyle \boldsymbol{\down}$};
\draw [-] (0,0) .. controls +(0,-1) and +(0,-1) .. +(1,0) node at
+(0.5,-1.5) {0};
\fill (0.5,-.75) circle(2.5pt);
\node at (0,.2) {$\scriptstyle \boldsymbol{\Diamond}$};
\node at (1,.2) {$\scriptstyle \boldsymbol{\up}$};
\draw [-] (2,0) .. controls +(0,-1) and +(0,-1) .. +(1,0) node at
+(0.5,-1.5) {1};
\fill (2.5,-.75) circle(2.5pt);
\node at (2,.2) {$\scriptstyle \boldsymbol{\Diamond}$};
\node at (3,.2) {$\scriptstyle \boldsymbol{\down}$};
\draw [-] (4,-0.7) .. controls +(0,1) and +(0,1) .. +(1,0) node at
+(0.5,-0.8) {0};
\fill (4.5,.05) circle(2.5pt);
\node at (4,-.9) {$\scriptstyle \boldsymbol{\up}$};
\node at (5,-.9) {$\scriptstyle \boldsymbol{\up}$};
\draw [-] (6,-0.7) .. controls +(0,1) and +(0,1) .. +(1,0)  node at
+(0.5,-0.8) {1};
\fill (6.5,.05) circle(2.5pt);
\node at (6,-.9) {$\scriptstyle \boldsymbol{\down}$};
\node at (7,-.9) {$\scriptstyle \boldsymbol{\down}$};
\draw [-] (8,-0.7) .. controls +(0,1) and +(0,1) .. +(1,0)  node at
+(0.5,-0.8) {0};
\fill (8.5,.05) circle(2.5pt);
\node at (8,-.9) {$\scriptstyle \boldsymbol{\Diamond}$};
\node at (9,-.9) {$\scriptstyle \boldsymbol{\up}$};
\draw [-] (10,-0.7) .. controls +(0,1) and +(0,1) .. +(1,0)  node at
+(0.5,-0.8) {1};
\fill (10.5,.05) circle(2.5pt);
\node at (10,-.9) {$\scriptstyle \boldsymbol{\Diamond}$};
\node at (11,-.9) {$\scriptstyle \boldsymbol{\down}$};
\draw [-] (12,0) -- +(0,-0.7) node at +(0,-1.5) {0};
\fill (12,-.35) circle(2.5pt);
\node at (12,.2) {$\scriptstyle \boldsymbol{\up}$};
\draw [-] (13,0) -- +(0,-0.7) node at +(0,-1.5) {0};
\fill (13,-.35) circle(2.5pt);
\node at (13,.2) {$\scriptstyle \boldsymbol{\Diamond}$};
\draw [-] (14,0) -- +(0,-0.7) node at +(0,-1.5) {0};
\fill (14,-.35) circle(2.5pt);
\node at (14,-.9) {$\scriptstyle \boldsymbol{\up}$};
\draw [-] (15,0) -- +(0,-0.7) node at +(0,-1.5) {0};
\fill (15,-.35) circle(2.5pt);
\node at (15,-.9) {$\scriptstyle \boldsymbol{\Diamond}$};
\end{scope}
\end{tikzpicture}
\vspace{-0.2cm}
\caption{Orientations (local picture) and their degrees}
\label{oriented}
\end{figure}
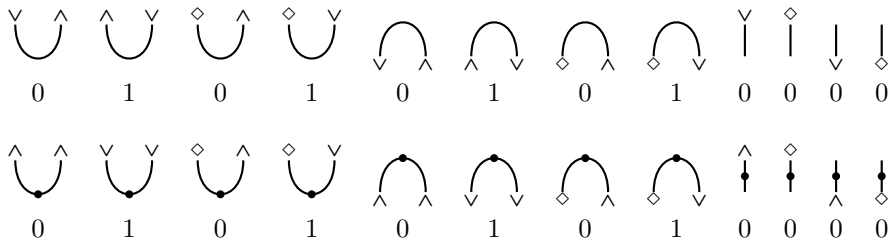
\smallskip
\smallskip

For instance, the cup diagram in \eqref{thecup} together with the weight from \eqref{thecup} is an oriented cup diagram. In fact $\underline{\la}\la$ is always an oriented cup diagram for any diagrammatic weight $\la$. Note that $\underline\la_{\op{super}}$ in \eqref{thecup} has $2^6$ possible orientations, namely precisely given by those weights $\nu$ which we obtain by choosing any subset of the cups in $\underline\la_{\op{super}}$ and changing the corresponding labels in $\la$ from $\down$ to $\up$ respectively $\up$ to $\down$ at each cup. In general a cup diagram $c$ has precisely $2^{\op{def}(c)}$ number of orientations.

\begin{definition}
\label{orientedcircle}
A triple $(\la,\nu,\mu)$ of diagrammatic weights is an {\it oriented circle diagram} if $\underline{\la}\overline{\mu}$  is a circle diagram and $\nu$ is an orientation of both $\underline{\la}$ and $\underline{\mu}$. We usually write such a triple as $\underline{\la}\nu\ov{\mu}$ and display it as the diagram $\underline{\la}\overline{\mu}$ with some labelling in the middle turning it into an oriented diagram in the sense that locally every arc looks like one of the form displayed in Figure~\ref{oriented}.

The dots should be thought of as orientation reversing points justifying the displayed local situations in Figure~\ref{oriented}.
\end{definition}

We refer to Figure~\ref{basis23} for all possible orientations on circle diagrams obtained from the cup diagrams in Figure~\ref{PIMS}. Obviously the following holds.

\begin{lemma}
If a circle diagram can be oriented, then there are precisely $2^x$ possible orientations, where $x$ is the number of circles in the diagram.  
\end{lemma}

\begin{remark}
\label{ornuclear}
Not every circle diagram can be oriented. As shown in \cite[Lemma 4.8]{ES2} to be orientable one needs at least that each circle in $\underline{\la}\overline{\mu}$ has an even number of $\bullet$'s. 
By \cite[Lemma 4.8]{ES2} a circle diagram which is not nuclear can be oriented if and only if  each component (circle or line) has an even number of dots.
\end{remark}

\begin{corollary}
\label{signsign}
Assume $\la,\mu \in X^+(G)$ and $\underline{\la}\overline{\mu}$ is not nuclear. Then $\underline{\la}\overline{\mu}$ is orientable if and only if each component contains an even number of dots. Moreover the number of possible orientations equals $2^c$, where $c$ is the number of closed components in $\underline{\la}\overline{\mu}$.
\end{corollary}

\begin{proof}
This follows immediately from \cite[Lemma 4.8]{ES2}, see Remark~\ref{ornuclear}.
\end{proof}

\begin{definition}
The {\it degree} of an oriented cup diagram $\underline{\la}\nu$ or an oriented circle diagram  $\underline{\la}\nu\ov{\mu}$ is the sum of the degrees of its components of the form as in Figure~\ref{oriented}, where the degree of each component is listed below each picture.
\end{definition}

It follows from the definitions that $\underline{\la}\la$ is the unique orientation of $\underline{\la}$ of degree zero; all other orientations have positive degrees. In \cite{ES2} we called cups or caps of degree $0$ {\it anticlockwise} and those of degree $1$ {\it clockwise}. Then the degree is just the number of clockwise cups plus clockwise caps. For examples see Figure~\ref{basis23}.

\section{Diagrammatics: \texorpdfstring{$\OSPrn$}{OSP(r|2n))} }
The goal of this section is the assignment of a certain cup diagram to each irreducible finite-dimensional $\OSPrn$-module in $\cF$. This allows us to make the connection with the  Khovanov algebra $\D$ and to formulate and prove the main theorem (Theorem~\ref{mainprop}). 

\subsection{Fake cups and frozen vertices}
Our infinite diagrammatic weights $\pla^\infty$ attached to $\la\in X^+(\mg)$ via \eqref{lainfty} and Notation \ref{not:part} and their cup diagrams $\underline{\pla^\infty}$ are slightly more general than those allowed in \cite{ES2} in the sense that they might have infinite defect. Diagrammatic weights with infinite defect were carefully avoided however in \cite{BS1} and in \cite{ES2}, since the associated Khovanov algebra would not be well-defined. Note moreover that $\underline{\pla}^\infty$ only depends on $\la$ and $\tfrac{\de}{2}$, but not on $m,n$ itself. We will next introduce a dependence on $m,n$ which also has the effect of giving a certain finiteness condition which allows us to avoid working with infinite defects. This will finally put us into the framework from \cite{ES2} and enable us to talk about the Khovanov algebra associated to a block of $\OSPrn$. The defect will correspond to the usual notion of atypicality of weights in the context of Lie superalgebras. We start by incorporating the dependence on $m$ and $n$.
\begin{definition}
\label{fakecups}
{\rm
Given $\la\in X^+(\mg)$ with associated infinite cup diagram $\underline{\pla^\infty}$, a cup $C$ is a \emph{fake cup} if $C$ is dotted and there are at least $\op{tail}(\pla)$ dotted cups to the left of $C$.
The vertices attached to fake cups are called \emph{frozen vertices}. We indicate the frozen vertices by $\owedge$.}
\end{definition}

\begin{remark}
By definition,  fake cups are never nested inside another cup, since dotted cups are never nested. Moreover, all dotted cups to the right of a fake cup are obviously also fake cups.
\end{remark}

To determine the fake cups, note that Corollary~\ref{tail} gives a formula to compute the tail length $\op{tail}(\pla)$. For instance, $\op{tail}(\varnothing)$ equals $\op{min}\{m,n\}$. For the empty partition the frozen vertices are indicated in \eqref{empty1} and \eqref{empty2}, where the dependence on $m$ and $n$ is also illustrated; see also Section~\ref{sec:osp32} for more examples.

\begin{definition}
Given $\la\in X^+(\mg)$ we define the super diagrammatic weight $\pla^\owedge$ as the one obtained from $\pla^\infty$ by replacing all the frozen labels by $\down$'s. 
\end{definition}

\begin{example} For instance consider $G=\op{OSp}(6|4)$, that is $m=3$, $n=2$. First consider $\la\in X^+(\mg)$ with corresponding  hook partition $\pla=(4,2,1)$. Then  
\begin{equation}
\label{upatzero}
\begin{tikzpicture}[thick,scale=0.65]
\node at (-1.2,0) {$\pla^\infty: $};
\node at (0,0) {$\Diamond$};
\node at (.5,0) {$\circ$};
\node at (1,0) {$\up$};
\node at (1.5,0) {$\down$};
\node at (2,0) {$\up$};
\draw[thin] (2.25,-.15) to +(0,-.2) to +(2.25,-.2);
\draw[thin,dotted] (4.5,-.35) to +(.75,0);
\node at (2.5,0) {$\owedge$};
\node at (3,0) {$\owedge$};
\node at (3.5,0) {$\owedge$};
\node at (4,0) {$\owedge$};
\node at (4.8,0) {$\cdots$};

\node at (6,0) {$\rightsquigarrow$};

\begin{scope}[xshift=9cm]
\node at (-1.2,0) {$\pla^\owedge: $};
\node at (0,0) {$\Diamond$};
\node at (.5,0) {$\circ$};
\node at (1,0) {$\up$};
\node at (1.5,0) {$\down$};
\node at (2,0) {$\up$};
\draw[thin] (2.25,-.15) to +(0,-.2) to +(2.25,-.2);
\draw[thin,dotted] (4.5,-.35) to +(.75,0);
\node at (2.5,0) {$\down$};
\node at (3,0) {$\down$};
\node at (3.5,0) {$\down$};
\node at (4,0) {$\down$};
\node at (4.8,0) {$\cdots$};
\end{scope}
\end{tikzpicture}
\end{equation}
where we indicated the relevant positions by a horizontal line. For the hook partition $\pla=(4,1,1)$ we obtain 
\begin{equation}
\label{circatzero}
\begin{tikzpicture}[thick,scale=0.65]
\node at (-1.2,0) {$\pla^\infty: $};
\node at (0,0) {$\circ$};
\node at (.5,0) {$\up$};
\node at (1,0) {$\up$};
\node at (1.5,0) {$\down$};
\node at (2,0) {$\up$};
\draw[thin] (2.25,-.15) to +(0,-.2) to +(2.25,-.2);
\draw[thin,dotted] (4.5,-.35) to +(.75,0);
\node at (2.5,0) {$\owedge$};
\node at (3,0) {$\owedge$};
\node at (3.5,0) {$\owedge$};
\node at (4,0) {$\owedge$};
\node at (4.8,0) {$\cdots$};

\node at (6,0) {$\rightsquigarrow$};

\begin{scope}[xshift=9cm]
\node at (-1.2,0) {$\pla^\owedge: $};
\node at (0,0) {$\circ$};
\node at (.5,0) {$\up$};
\node at (1,0) {$\up$};
\node at (1.5,0) {$\down$};
\node at (2,0) {$\up$};
\draw[thin] (2.25,-.15) to +(0,-.2) to +(2.25,-.2);
\draw[thin,dotted] (4.5,-.35) to +(.75,0);
\node at (2.5,0) {$\down$};
\node at (3,0) {$\down$};
\node at (3.5,0) {$\down$};
\node at (4,0) {$\down$};
\node at (4.8,0) {$\cdots$};
\end{scope}
\end{tikzpicture}
\end{equation}
\end{example}
\begin{example}
Note that in case $G=\op{OSp}(7|4)$, the weights $\la\in X^+(\mg)$ with hook partitions $\pla=(5,2,1)$ respectively $(5,1,1)$  give rise to the same four diagrams as in \eqref{upatzero} and \eqref{circatzero} above (except that $\Diamond$ is replaced by $\up$), but placed on the positive half-integer line instead of the positive integer line. 
\end{example}

\subsection{Diagrammatics associated to irreducible modules} 
We now can assign to each dominant weight $\lambda \in X^+(G)$ a diagrammatic weight. 

Assume we are given a weight diagram $\la$ of hook partition type. For each label $x$ of the form $\circ$ or $\times$ appearing in $\la$ we let $d(x)$ be the total number of $\up$'s and $\down$'s to the left of $x$
and let $d(\la)=\sum_{x}d(x)$ be the sum of all these numbers. 
\begin{definition}
\label{superweight}
Consider $G=\SOSPrn$. Given $\la\in X^+(G)$ with underlying hook partition $\pla$, we define the \emph{(super) diagrammatic weight attached to $\la$}, and also denoted by $\la$, as follows 
\begin{subequations}
\begin{numcases}{\la=}
\pla^{\owedge} & if $\la=\Psi(\pla)$, \label{nosign}\\
(\pla^\owedge,+) & if $\la=\Psi((\pla,+))$,\label{plus}\\
(\pla^\owedge,-) & if $\la=\Psi((\pla,-))$,\label{minus}
\end{numcases}
\end{subequations}
where we use the identifications from Lemma~\ref{AB}. Here for $r$ even 
$(\pla^\owedge,-)$ is the same as $\pla^\owedge$, while $(\pla^\owedge,+)$ is obtained from $\pla^\owedge$ by changing the label at the first occurring ray, in the corresponding cup diagram $\underline{\pla}^\owedge$, from $\down$ to $\up$. In case $r$ is odd we first determine the parity of $d(\pla^\owedge)$ plus the number of $\up$'s in $\pla^\owedge$. In case it is even then we apply the same rule as for odd $r$, in case the parity is odd, the role of $(\pla^\owedge,-)$ and $(\pla^\owedge,+)$ are swapped. (Note that the parity of $d(\pla^\owedge)$ is the parity of the number of boxes in $\pla$.)
\end{definition}

\begin{example}
As above let $G=\op{OSp}(6|4)$. In situation \eqref{circatzero} we have
\begin{equation*}
\begin{tikzpicture}[thick,scale=0.65]
\node at (-1,0) {$\pla^\infty: $};
\node at (0,0) {$\circ$};
\node at (.5,0) {$\up$};
\node at (1,0) {$\up$};
\node at (1.5,0) {$\down$};
\node at (2,0) {$\up$};
\draw[thin] (2.25,-.15) to +(0,-.2) to +(2.25,-.2);
\draw[thin,dotted] (4.5,-.35) to +(.75,0);
\node at (2.5,0) {$\owedge$};
\node at (3,0) {$\owedge$};
\node at (3.5,0) {$\owedge$};
\node at (4,0) {$\owedge$};
\node at (4.8,0) {$\cdots$};

\node at (6,0) {$\rightsquigarrow$};

\begin{scope}[xshift=10cm,yshift=.5cm]
\node at (-1.5,0) {$(\pla^\owedge,+): $};
\node at (0,0) {$\circ$};
\node at (.5,0) {$\up$};
\node at (1,0) {$\up$};
\node at (1.5,0) {$\down$};
\node at (2,0) {$\up$};
\draw[thin] (2.25,-.15) to +(0,-.2) to +(2.25,-.2);
\draw[thin,dotted] (4.5,-.35) to +(.75,0);
\node at (2.5,0) {$\up$};
\node at (3,0) {$\down$};
\node at (3.5,0) {$\down$};
\node at (4,0) {$\down$};
\node at (4.8,0) {$\cdots$};
\end{scope}
\begin{scope}[xshift=10cm,yshift=-.5cm]
\node at (-1.5,0) {$(\pla^\owedge,-): $};
\node at (0,0) {$\circ$};
\node at (.5,0) {$\up$};
\node at (1,0) {$\up$};
\node at (1.5,0) {$\down$};
\node at (2,0) {$\up$};
\draw[thin] (2.25,-.15) to +(0,-.2) to +(2.25,-.2);
\draw[thin,dotted] (4.5,-.35) to +(.75,0);
\node at (2.5,0) {$\down$};
\node at (3,0) {$\down$};
\node at (3.5,0) {$\down$};
\node at (4,0) {$\down$};
\node at (4.8,0) {$\cdots$};
\end{scope}
\end{tikzpicture}
\end{equation*}
and in situation \eqref{upatzero}, we obtain
\begin{equation*}
\begin{tikzpicture}[thick,scale=0.65]
\node at (-1,0) {$\pla^\infty: $};
\node at (0,0) {$\Diamond$};
\node at (.5,0) {$\circ$};
\node at (1,0) {$\up$};
\node at (1.5,0) {$\down$};
\node at (2,0) {$\up$};
\draw[thin] (2.25,-.15) to +(0,-.2) to +(2.25,-.2);
\draw[thin,dotted] (4.5,-.35) to +(.75,0);
\node at (2.5,0) {$\owedge$};
\node at (3,0) {$\owedge$};
\node at (3.5,0) {$\owedge$};
\node at (4,0) {$\owedge$};
\node at (4.8,0) {$\cdots$};

\node at (6,0) {$\rightsquigarrow$};

\begin{scope}[xshift=10cm,yshift=.5cm]
\node at (-1.5,0) {$(\pla^\owedge,+): $};
\node at (0,0) {$\Diamond$};
\node at (.5,0) {$\circ$};
\node at (1,0) {$\up$};
\node at (1.5,0) {$\down$};
\node at (2,0) {$\up$};
\draw[thin] (2.25,-.15) to +(0,-.2) to +(2.25,-.2);
\draw[thin,dotted] (4.5,-.35) to +(.75,0);
\node at (2.5,0) {$\up$};
\node at (3,0) {$\down$};
\node at (3.5,0) {$\down$};
\node at (4,0) {$\down$};
\node at (4.8,0) {$\cdots$};
\end{scope}

\begin{scope}[xshift=10cm,yshift=-.5cm]
\node at (-1.5,0) {$(\pla^\owedge,-): $};
\node at (0,0) {$\Diamond$};
\node at (.5,0) {$\circ$};
\node at (1,0) {$\up$};
\node at (1.5,0) {$\down$};
\node at (2,0) {$\up$};
\draw[thin] (2.25,-.15) to +(0,-.2) to +(2.25,-.2);
\draw[thin,dotted] (4.5,-.35) to +(.75,0);
\node at (2.5,0) {$\down$};
\node at (3,0) {$\down$};
\node at (3.5,0) {$\down$};
\node at (4,0) {$\down$};
\node at (4.8,0) {$\cdots$};
\end{scope}
\end{tikzpicture}
\end{equation*}
\end{example}

\begin{definition} \label{def:supercupdiagram}
The cup diagram $\underline{\lambda}$ attached to $\lambda \in  X^+(G)$ is the cup diagram obtained from \eqref{nosign}, \eqref{plus}, and \eqref{minus} via Definition~\ref{decoratedcups}.
\end{definition}
\noindent For explicit examples we refer to Section~\ref{sec:examples} or Figure~\ref{huge}.
\smallskip

\begin{figure}[h]
\begin{center}
\begin{tikzpicture}[thick]
\node at (-1,-1) {$\pla$};
\node at (2.5,-1) {$\underline{\pla^\owedge}$};
\node at (6,-1) {$\op{tail}(\pla)$};
\node at (8,-1) {$\op{def}(\pla^\owedge)$};
\begin{scope}
\node at (-1,0) {\scalebox{.4}{\yng(3,1,1)}};
\node at (6,0) {$2$};
\node at (8,0) {$3$};
\node at (1,.1) {$\scriptstyle \mathbf{\up}$};
\node at (1.5,.1) {$\scriptstyle \mathbf{\up}$};
\node at (2,.1) {$\scriptstyle \mathbf{\down}$};
\node at (2.5,.1) {$\scriptstyle \mathbf{\up}$};
\node at (3,.1) {$\scriptstyle \mathbf{\up}$};
\node at (3.5,.1) {$\scriptstyle \mathbf{\up}$};
\node at (4,.1) {$\scriptstyle \mathbf{\owedge}$};
\draw (1,0) .. controls +(0,-.5) and +(0,-.5) .. +(.5,0);
\fill (1.25,-.365) circle(2pt);
\draw (2,0) .. controls +(0,-.5) and +(0,-.5) .. +(.5,0);
\draw (3,0) .. controls +(0,-.5) and +(0,-.5) .. +(.5,0);
\fill (3.25,-.365) circle(2pt);
\draw (4,0) -- +(0,-.5);
\node at (4.5,0) {$\cdots$};
\end{scope}
\begin{scope}[yshift=1cm]
\node at (-1,0) {\scalebox{.4}{\yng(3,2,1)}};
\node at (6,0) {$1$};
\node at (8,0) {$3$};
\node at (1,.1) {$\scriptstyle \mathbf{\down}$};
\node at (1.5,.1) {$\scriptstyle \mathbf{\up}$};
\node at (2,.1) {$\scriptstyle \mathbf{\down}$};
\node at (2.5,.1) {$\scriptstyle \mathbf{\up}$};
\node at (3,.1) {$\scriptstyle \mathbf{\up}$};
\node at (3.5,.1) {$\scriptstyle \mathbf{\up}$};
\node at (4,.1) {$\scriptstyle \mathbf{\owedge}$};
\draw (1,0) .. controls +(0,-.5) and +(0,-.5) .. +(.5,0);
\draw (2,0) .. controls +(0,-.5) and +(0,-.5) .. +(.5,0);
\draw (3,0) .. controls +(0,-.5) and +(0,-.5) .. +(.5,0);
\fill (3.25,-.365) circle(2pt);
\draw (4,0) -- +(0,-.5);
\node at (4.5,0) {$\cdots$};
\end{scope}
\begin{scope}[yshift=2cm]
\node at (-1,0) {\scalebox{.4}{\yng(3,2,2)}};
\node at (6,0) {$1$};
\node at (8,0) {$2$};
\node at (1,.1) {$\scriptstyle \mathbf{\times}$};
\node at (1.5,.1) {$\scriptstyle \mathbf{\circ}$};
\node at (2,.1) {$\scriptstyle \mathbf{\down}$};
\node at (2.5,.1) {$\scriptstyle \mathbf{\up}$};
\node at (3,.1) {$\scriptstyle \mathbf{\up}$};
\node at (3.5,.1) {$\scriptstyle \mathbf{\up}$};
\node at (4,.1) {$\scriptstyle \mathbf{\owedge}$};
\draw (2,0) .. controls +(0,-.5) and +(0,-.5) .. +(.5,0);
\draw (3,0) .. controls +(0,-.5) and +(0,-.5) .. +(.5,0);
\fill (3.25,-.365) circle(2pt);
\draw (4,0) -- +(0,-.5);
\node at (4.5,0) {$\cdots$};
\end{scope}
\begin{scope}[yshift=3cm]
\node at (-1,0) {\scalebox{.4}{\yng(3,3,2)}};
\node at (6,0) {$1$};
\node at (8,0) {$3$};
\node at (1,.1) {$\scriptstyle \mathbf{\up}$};
\node at (1.5,.1) {$\scriptstyle \mathbf{\down}$};
\node at (2,.1) {$\scriptstyle \mathbf{\down}$};
\node at (2.5,.1) {$\scriptstyle \mathbf{\up}$};
\node at (3,.1) {$\scriptstyle \mathbf{\up}$};
\node at (3.5,.1) {$\scriptstyle \mathbf{\up}$};
\node at (4,.1) {$\scriptstyle \mathbf{\owedge}$};
\draw (1,0) .. controls +(0,-.75) and +(0,-.75) .. +(2.5,0);
\draw (1.5,0) .. controls +(0,-.5) and +(0,-.5) .. +(1.5,0);
\draw (2,0) .. controls +(0,-.25) and +(0,-.25) .. +(.5,0);
\fill (2.25,-.565) circle(2pt);
\draw (4,0) -- +(0,-.5);
\node at (4.5,0) {$\cdots$};
\end{scope}
\begin{scope}[yshift=4cm]
\node at (-1,0) {\scalebox{.4}{\yng(3,3,3)}};
\node at (6,0) {$0$};
\node at (8,0) {$3$};
\node at (1,.1) {$\scriptstyle \mathbf{\down}$};
\node at (1.5,.1) {$\scriptstyle \mathbf{\down}$};
\node at (2,.1) {$\scriptstyle \mathbf{\down}$};
\node at (2.5,.1) {$\scriptstyle \mathbf{\up}$};
\node at (3,.1) {$\scriptstyle \mathbf{\up}$};
\node at (3.5,.1) {$\scriptstyle \mathbf{\up}$};
\node at (4,.1) {$\scriptstyle \mathbf{\owedge}$};
\draw (1,0) .. controls +(0,-.75) and +(0,-.75) .. +(2.5,0);
\draw (1.5,0) .. controls +(0,-.5) and +(0,-.5) .. +(1.5,0);
\draw (2,0) .. controls +(0,-.25) and +(0,-.25) .. +(.5,0);
\draw (4,0) -- +(0,-.5);
\node at (4.5,0) {$\cdots$};
\end{scope}
\end{tikzpicture}
\end{center}
\abovecaptionskip-5pt
\belowcaptionskip0pt
\caption{Cup diagrams $\underline{\pla}^\owedge$ associated to hook partitions in the case $\SOSP(7|6)$. With the corresponding weights ${\pla}^\owedge$ displayed above the cup diagrams.}
\label{huge}
\end{figure}
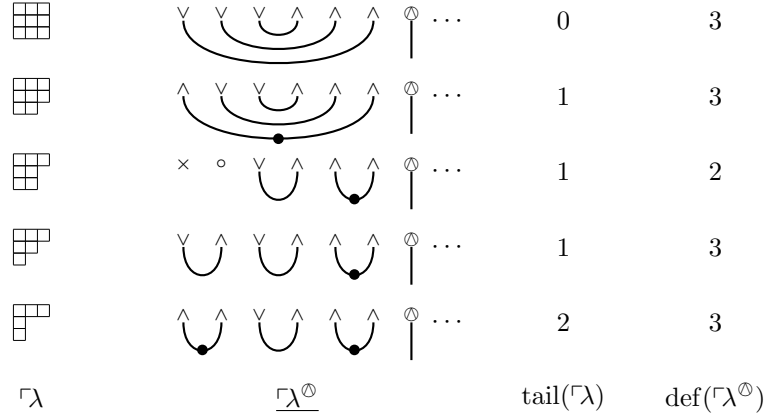

\begin{remark}
As a result we have attached to any $\la\in X^+(G)$ a cup diagram $\underline{\la}$ which has an infinite number of undotted rays,  tail length many dotted cups, and at most one dotted ray. Observe that $\underline {\la}$  coincides with $\underline{\pla^\infty}$, except that each fake cup is replaced by two vertical rays (with the leftmost ray and cup possibly decorated differently). In other words, we keep the undotted cups, but force the diagram to have exactly as many dotted cups as the length of the tail by taking the first $\op{tail}(\la)$ dotted cups. Note also that the core diagram of the diagrammatic weight $\la$ is the same as the core diagram of $\pla^\infty$ from Definition~\ref{lainfty}. 
\end{remark}

\begin{remark}
\label{Weyl}
The weight diagrams attached to the pair $(\la,\pm)$ can be viewed as super analogues of the notion of {\it associated partitions}, see \cite[$\S$ 19.5]{FH}, which was used by Weyl to label pairs of irreducible representations for $\op{O}(r)$ which restrict to isomorphic representations for $\op{SO}(r)$, see Section~\ref{sec:classical} for more details.
\end{remark}


\begin{prop}
\label{Voila}
Assume $r=2m$ and let $\la\in X^+(\mg)$. Then there are the following two cases:
\begin{enumerate}
\item $\op{Ind}^{\mg, G}_{\mg, G'}L^\mg(\la)=L(\la^G)$ is irreducible. 
Then, at  position zero, the attached diagrammatic weight $\la^G$ has a $\Diamond$ and the cup diagram $\underline{\la^G}$ has a ray at position zero. This ray is dotted if $\op{min}\{m,n\}$ is even and undotted if it is odd.
\item $\op{Ind}^{\mg, G}_{\mg, G'}L^\mg(\la)\cong L(\la,+)\oplus L(\la,-)$. Then both of the  diagrammatic weights $(\la,+)$ and $(\la,-)$ have a $\circ$ or a $\Diamond$ at position zero. In case of a $\Diamond$, the cup diagrams $\underline{(\lambda,+)}$ and $\underline{(\lambda,-)}$ have a cup attached to position zero, one of them dotted, the other undotted.
\end{enumerate} 
\end{prop}
\begin{proof}
In the situation (1) Corollary~\ref{cornotail} implies that there is a $\Diamond$ at  position zero and the tail is zero. Hence the dotted cup attached to the zero position in $\la^\infty$ is fake, and thus gives a ray in $\underline{\la}$. Depending on the parity of $\op{min}\{m,n\}$, $\underline{\lambda}$ has the indicated type of leftmost ray (i.e. dotted or undotted). Now situation (2) is equivalent to $\pla_{n+1}<m$ and $\circ$ or $\Diamond$ can occur at position zero.  Assume first $\op{tail}(\pla)=0$, this means $\pla_{n}\geq m$. Then $\cS(\pla)_n=\frac{\delta}{2}+n-1-\pla_n=m-n+n-1-\pla_n<0$, and  $\cS(\pla)_{n+1}=\frac{\delta}{2}+n+1-1-\pla_{n+1}=m-n+n-\pla_{n+1}>0$. Since the sequence $\cS(\pla)$ is strictly increasing, the value zero does not occur and thus we have $\circ$ at position zero. In particular,  if $\Diamond$ occurs then we must have $\op{tail}(\pla)>0$, in which case the dotted cup attached to zero in $\pla^\infty$ is not fake and so there is a dotted or undotted cup at position zero depending on the parity. 
\end{proof}

\subsection{Blocks and diagrammatic linkage}
\begin{definition}
\label{dlinkage}
{\rm We say that two elements $\la,\mu\in X^+(G)$ are {\it diagrammatically linked} if their attached super diagrammatic weights $\la$ and $\mu$, in the sense of Definition~\ref{superweight}, are in the same block, in the sense of Definition~\ref{diagramblocks}}.
\end{definition}
\begin{lemma}      
\label{lem:defect}
Given $\la\in X^+(G)$ then  $\op{def}(\la)=n-\#\times(\la)$ with the notation from Definition~\ref{defdef}.  In particular, if $\la$ and $\mu$ give rise to the same core diagram then  $\op{def}(\la)= \op{def}(\mu)$.
\end{lemma}
\begin{proof}
Note that passing from ${\pla^\owedge}$ to ${\la}$ does not change the total number of cups in the corresponding cup diagram. Now, the number of undotted cups in $\underline{\pla^\owedge}$ equals $\#\down(\la^\infty)$, whereas the number of dotted cups is by construction equal to $\op{tail}(\la)=n-s$, where $s=\# \down(\la^\infty)+\#\times(\la^\infty)$. The claim follows.
 \end{proof}

\begin{corollary} \label{lem:samedefect}
Two diagrammatically linked elements $\la,\mu\in X^+(G)$ have the same defect. 
\end{corollary}
\begin{proof}
Since they have by definition the same core diagram this follows directly from Lemma~\ref{lem:defect}.
\end{proof}

\begin{prop}
\label{Homzero}
Let $G=\OSPrn$. Assume $\la=(\la^\prime,\epsilon),\mu=(\mu^\prime,\epsilon^\prime) \in X^+(G)$ for $\epsilon,\epsilon^\prime \in \{ \pm \}$ such that the circle diagram $\un\la\ov\mu$ is not nuclear. 
\begin{enumerate}[(i)]
\item If the core diagrams of both $\la$ and $\mu$ do not contain the zero vertex, then $\underline{\la}\overline{\mu}$ is not orientable for $\epsilon \neq \epsilon^\prime$. For $\epsilon=\epsilon^\prime$ the number of orientations is independent of the choice of $\epsilon$.
If the core diagrams of both $\la$ and $\mu$ contain the zero vertex, we distinguish the following.
\item If $\underline{\la}\overline{\mu}$ contains a line passing through the zero vertex for some choice (and thus for all) of $\epsilon$ and $\epsilon^\prime$, then the number of orientations of the diagram is independent of the choice.
\item If $\underline{\la}\overline{\mu}$ contains a circle passing through the zero vertex for some choice of $\epsilon$ and $\epsilon^\prime$, then either it is not orientable for any choice of $\epsilon$ and $\epsilon^\prime$ or it is orientable for precisely two choices $(\epsilon,\epsilon^\prime)$ and $(-\epsilon,-\epsilon^\prime)$ and again the number of orientations agrees for both of these choices.
\end{enumerate}
\end{prop}
\begin{proof} 
In case (i), $\la$ and $\mu$ are not diagrammatically linked if $\epsilon \neq \epsilon^\prime$. Indeed the parities of $\# \up(\la)$ and $\# \up(\mu)$ differ by construction. On the other hand,  by Figure~\ref{oriented}, every orientation $\nu$ of $\underline{\la}$ satisfies  $\# \up(\nu)\equiv \# \up(\la)\op{mod} 2$. Hence any orientation $\un\la\nu\ov\mu$ of $\un\la \ov\mu$ implies $\# \up(\la)\equiv \# \up(\mu)\op{mod} 2$. If $\epsilon=\epsilon^\prime$ then for a non nuclear diagram the number of orientations agrees for both possible choices, since the leftmost rays in $\underline{\la}$ and in $\overline{\mu}$ are on the same propagating line.

Assume now that both core diagrams contain zero. In case (ii) the (propagating) line through zero contains a cup and a cap at zero, as well as the leftmost ray in both $\underline{\la}$ and $\overline{\mu}$. Thus if is orientable, it is orientable for any choice of $\epsilon$ and $\epsilon^\prime$ and the number is independent by Remark~\ref{ornuclear}.

Assume now case (iii), that means the component containing the zero vertex is a circle. Here, the leftmost rays are contained in the same propagating line not passing through zero. Hence if $\underline{\la}\overline{\mu}$ is orientable, then negating both $\epsilon$ and $\epsilon^\prime$ also produces an orientable diagram, while all other choices are not orientable due to Remark~\ref{ornuclear}.
\end{proof}

The following is therefore only applicable in case $G=\op{OSp}(2m|2n)$, since otherwise $(\la,+)$, and $(\la,-)$ are in different diagrammatic blocks.
\begin{prop}
\label{nosignsign}
Consider elements $\mu, (\la,+), (\la,-)\in X^+(G)$ and the corresponding diagrammatic weights, which we denote by the same notation, covering exactly the three cases in Definition~\ref{superweight}. Assume that these weights are in the same diagrammatic block. 
\begin{enumerate}
\item
Then $\underline{\mu}\overline{(\la,+)}$ is not nuclear if and only if $\underline{\mu}\overline{(\la,-)}$ is not nuclear. 
\item If $\underline{\mu}\overline{(\la,+)}$ is not nuclear, we have moreover that the number of possible orientations of $\underline{\mu}\overline{(\la,+)}$ equals the number of possible orientations of   $\underline{\mu}\overline{(\la,-)}$.
\item This number of possible orientations is non-zero if and only if every component in $\underline{\mu}\overline{(\la,\pm)}$ contains an even number of dots in which case it equals $2^c$ where $c$ is the number of closed components. 
\end{enumerate}
\end{prop}
\begin{proof}
The first statement is obvious, since $(\la,+)$ differs from $(\la,-)$ only by some dots. To see (2), note that by Proposition~\ref{Voila}, the cup diagrams $\underline{(\la,+)}$ and $\underline{(\la,-)}$ have both a cup at the position zero, in one case dotted and in the other undotted, whereas $\underline{\mu}$ has a ray at position zero (dotted or not depending on the parity of $\op{min}\{n,m\}$). 
This implies that both $\underline{\mu}\overline{(\la,+)}$ and $\underline{\mu}\overline{(\la,-)}$ have a propagating line through zero in case they are not nuclear. By construction, this propagating line contains cup at position zero and the (leftmost) ray in both $\underline{(\la,+)}$ and $\underline{(\la,-)}$. Since by Definitions~\ref{superweight} and~\ref{def:supercupdiagram} these are exactly the two components that differ in $\underline{(\la,+)}$ and $\underline{(\la,-)}$ by a dot, thus part (2) follows. For the third part observe that in the non-nuclear case components are orientable precisely if and only if they have an even number of dots, see Corollary~\ref{signsign}.
\end{proof}

\section{The main theorem,  duality, and the nuclear ideal}
Our main theorem gives now a description of the underlying vector space of   $\Hom_\cF(P(\la),P(\mu))$ for any $\la,\mu\in X^+(G)$, which in particular includes an explicit counting formula for the dimension of the spaces of morphisms between two indecomposable projective objects. In the special case $G=\OSP(2m+1|2n)$ this gives Theorem~\ref{A} from the introduction.

\subsection{The main theorem}
Recall the vector space $\mathbb{I}$ from Definition~\ref{nuclear}.
\begin{theorem}
\label{mainprop}
Consider $G=\OSPrn$  for fixed $m,n$. For any $\la,\mu\in X^+(G)$ there is an isomorphisms of vector spaces
\begin{eqnarray}
\label{BI}
\op{Hom}_\cF(P(\la),P(\mu))&\cong&
\mathbb{B}(\la,\mu)/\mathbb{I}_{\la,\mu}.
\end{eqnarray}
Here, $\mathbb{B}(\la,\mu)$ is the vector space with basis all oriented circle diagrams of the form $\underline{\la}\nu\overline{\mu}$ for some diagrammatic weights $\nu$, and $\mathbb{I}_{\la,\mu}$ is the vector subspace spanned by its set of nuclear diagrams. Hence
\begin{eqnarray}
\mathbb{B}(\la,\mu)/\mathbb{I}_{\la,\mu}=\left\langle\underline{\la}\nu\overline{\mu}\mid
\underline{\la}\nu\overline{\mu}\in\mathbb{B}\text{ and }\underline{\la}\overline{\mu}\not\in \mathbb{I}
 \right\rangle_\mathbb{C}.
\end{eqnarray}
\end{theorem}
\begin{proof}
Theorem~\ref{mainprop} will follow from the Dimension Formula (Theorem~\ref{maintheorem}).
\end{proof}

The following is  a shadow of the duality explained in \cite[5.5]{MW}:
\begin{corollary}[Duality]
Let $G=\OSP(2m+1|2n)$ and $G^t=\OSP(2n+1|2m)$. Let $\la,\mu\in X^+(G)$ and $\la^t,\mu^t\in X^+(G^t)$ be the corresponding element with the same sign, but transposed partition. Then 
\abovedisplayskip0.25em
\belowdisplayskip0.5em
\begin{eqnarray*}
\op{Hom}_{\cF(G)}(P(\la),P(\mu))&\cong&\op{Hom}_{\cF(G^t)}(P(\la^t),P(\mu^t)).
\end{eqnarray*}
\end{corollary}
\begin{proof}
This follows directly from Theorem~\ref{mainprop}, since the associated diagrammatic weight for $\la^t$ is obtained from that of $\la$ by swapping $\circ $ with $\times$ and $\times$ with $\circ$. This swapping is however irrelevant for the dimension counting, since it does (up to core symbols) not change the corresponding cup diagram, and also not the possible orientations.  
\end{proof}

Before we prove Theorem~\ref{maintheorem},  and thus Theorem~\ref{mainprop}, we explain how to put an algebra structure on $\bigoplus_{\la,\mu}\mathbb{B}(\la,\mu)/\mathbb{I}_{\la,\mu}$ as required in Theorem~\ref{main}.

\subsection{The algebra structure and the nuclear ideal}
Let $G=\OSPrn$ and consider a fixed block $\cB$ of $\cF$. Let $P=\oplus_\la P(\la)$ be a minimal projective generator, that is the direct sum runs over all $\la\in X^+(G)$ such that $P(\la)\in\cB$.  By Propositions~\ref{blocksindiagramsodd} and~\ref{blocksindiagramseven} below, the corresponding set $\Lambda(\cB)$ of diagrammatic weights is contained in a block $\Lambda$ in the sense of Definition~\ref{diagramblocks}. 
Let  $\mathbb{D}_\Lambda$ be the Khovanov algebra of type $\rm{D}$ attached to $\Lambda$ as defined in \cite{ESperv}. Let $\mathbbm{1}_\cB=\sum_{\la\in\Lambda(\cB)}\mathbbm{1}_\la$ be the idempotent in $\mathbb{D}_\Lambda$ corresponding to $\Lambda(\cB)$, see \cite[Theorem 6.2]{ESperv}. We consider now the algebra  $\mathbbm{1}_\cB\mathbb{D}_\Lambda \mathbbm{1}_\cB$. By definition it has a basis given by all oriented circle diagrams  $\underline{\la}\nu\overline{\mu}$, where $\la,\mu\in\Lambda(\cB)$.
We first observe the following crucial fact: 

\begin{prop}
\label{isideal}
The subspace $\mathbb{I}_\cB$ of $\mathbbm{1}_\cB\mathbb{D}_\Lambda \mathbbm{1}_\cB$ spanned by all basis vectors whose underlying circle diagram is nuclear is an ideal.
\end{prop}

\begin{proof}
Let $x\in\mathbb{I}_\cB$ be a basis vector. Hence we can find $\la,\mu\in\Lambda(\cB)$ such that $x\in\mathbb{B}(\la,\mu)\cap\mathbb{I}_\cB$ and $x$ contains at least one non-propagating line. It is enough to show that $cx, xc\in \mathbb{I}$ for any basis element $c$ of  $\mathbbm{1}_\cB\mathbb{D}_\Lambda\mathbbm{1}_\cB$. The algebra $\mathbb{D}_\Lambda$ has an anti-automorphism which sends a basis element $a\nu b$ to $b^*\nu a^*$ in the notation from Definition~\ref{orientedcircle}, see \cite[Corollary~6.4]{ESperv}. Obviously this descends to an anti-automorphism of $\mathbbm{1}_\cB\mathbb{D}_\Lambda \mathbbm{1}_\cB$ which preserve $\mathbb{I}_\cB$. Therefore, it is enough to show $bc\in \mathbb{I}_\cB$. 

Consider the non-propagating lines in $b$. Then the number of those ending at the top equals the number of those ending at the bottom since the weights in $\Lambda(\cB)$ are linked and have the same defect by Corollary~\ref{lem:samedefect}. Hence assume there is at least one such line $L$ ending at the bottom. 

From the surgery procedure defining the algebra structure we see directly that any surgery involving such a line and a circle either preserves this property (\cite[first two cases in Remark~5.13]{ESperv} and \cite[Remark~5.15]{ESperv}), or produces zero (\cite[last two cases in Remark~5.13]{ESperv} and \cite[Reconnect in 5.2.3]{ESperv}). Hence the claim follows.
\end{proof}
    
Now thanks to Theorem~\ref{mainprop}, there is a canonical isomorphism of vector spaces 
\abovedisplayskip0.5em
\belowdisplayskip0.5em
\[
\op{End}^{\fin}_\cF(P)\cong\mathbbm{1}_{\cB} \mathbb{D}_\Lambda \mathbbm{1}_{\cB}/\mathbb{I}_\cB,\]
sending a basis vector to the corresponding basis vector of $\mathbbm{1}_{\cB} \mathbb{D}_\Lambda \mathbbm{1}_{\cB}$ denoted in the same way. In particular,  $\op{End}^{\fin}_\cF(P)$ inherits by Proposition~\ref{isideal} an algebra structure from the Khovanov algebra  $\mathbb{D}_\Lambda$ via this identification. In part II of this series we show (a more general version of Theorem~\ref{main}) that the two algebras are isomorphic.

\section{Connection to the Gruson-Serganova combinatorics} \label{dictionary}
To prove Theorem~\ref{mainprop} we have to connect the diagram calculus developed in \cite{GS2} to our calculus. For later reference and to make a precise connection to \cite{GS2} we give an explicit dictionary, although we could prove the result more directly. The GS-diagrammatic weight $\op{GS}(\la)$ associated with $\la\in X^+(\mg)$ is
a certain labelling $\cL$ with the symbols $<$, $>$, $\times$, $\circ$, $\color{myred}{\boldsymbol \otimes}$ with almost all vertices labelled $\circ$. Gruson and Serganova obtain this labelling as the image of a composite of two maps
\begin{eqnarray}
\op{GS}:\;\; X^+(\mg)&\longrightarrow& \left\{\op{GS}-\text{diagrams with tail} \right\}\nonumber\\
&\longrightarrow &\left\{\text{coloured }\op{GS}-\text{diagrams without tail} \right\}\label{GSmap}
\end{eqnarray}
We refer to  \cite{GS2} for details, but will briefly recall the construction in Section~\ref{sec:GS_combinatorics} below. (The additional signs appearing in \cite{GS2} and in the weights for $X^+(\mg)$ do not play any role for us thanks to \eqref{zumGlueck} and therefore we can ignore them.)

For convenience we provide the explicit map $\op{T}$ which translates from  $\op{GS}$-weights $\op{GS}(\la)$ to our diagrammatic weights $\la^\infty=\op{T}(\op{GS}(\la))$ and vice versa.
The dictionary is as follows, where the first line shows the label in $\op{GS}(\la)$ and the second line the corresponding label in the diagrammatic weight $\op{T}(\op{GS}(\la))$:
\begin{eqnarray}
\label{T2}
\begin{array}{c||c|c|c|c|c||cc|c|c}
\op{GS}(\la)&<&>&\times&\circ&\color{myred}{\boldsymbol \otimes}&
\text{at $0$:}&\color{myred}{\boldsymbol \otimes}&>&
\text{}{\circ}\\
\hline
\op{T}(\op{GS}(\la))&\times&\circ&\down&\up&\up&&\Diamond&\circ&\Diamond
\end{array}
\end{eqnarray}
Even though the vertex $\frac{1}{2}$ will play a special role in the proofs to come, only the vertex $0$ in the even case has a special assignment rule.

\subsection{Comparison of the two cup diagram combinatorics}
In \cite{GS2}, Gruson and Serganova assigned to any $\op{GS}(\la)$-weight also some cup diagram (without any decorations). We claim that our combinatorics refines their combinatorics in the following sense (with the felicitous consequence that the assignment from $X^+(G)$ to cup diagrams is injective):

\begin{prop}
\label{prop:GSES}
Let $\la\in X^+(\mg)$ with associated hook diagram $\pla$. 
\begin{enumerate}
\item The assignment $\op{T}$, from \eqref{T2} satisfies
\begin{eqnarray}
\label{match}
\op{T}(\op{GS}(\la))&=&\pla^\infty.
\end{eqnarray}
\item Moreover, the cup diagram attached to $\op{GS}(\la)$ in the sense of \cite{GS2} agrees with our cup diagram $\underline{\pla^\infty}$ when forgetting the decorations and fake cups, and with $\underline{\pla^\owedge}$ when forgetting the decoration and all rays.
\item Under this correspondence the cups attached to $\color{myred}{\boldsymbol \otimes}$'s correspond precisely to the dotted, non-fake cups in $\underline{\pla^\infty}$,  and to the dotted cups in $\underline{\pla^\owedge}$.
\end{enumerate}
\end{prop}

\begin{proof}
It suffices to prove the statements involving $\pla^\infty$, since the others follow then directly from the definition of $\underline{\pla^\owedge}$. The proof is given in Section~\ref{sec:proofcomb}.
\end{proof}

\subsection{Blocks in terms of diagrammatic blocks}
Before we prove Proposition~\ref{prop:GSES}, we deduce some important consequence:

\begin{prop}
\label{blocksindiagramsodd}
Assume that $G=\OSPrn$ with $r$ odd. Let $\la,\mu\in X^+(G)$. Then $P(\la)$ and $P(\mu)$ (and hence then also $L(\la)$ and $L(\mu)$) are in the  same  block if and only if $\pla^\infty$ and $\pmu^\infty$ have the same core diagrams in the sense of  Definition~\ref{def:diagrammaticweight} and additionally $\#\up({\pla^\owedge})\equiv\#\up({\pmu^\owedge})\op{mod}2$. 
\end{prop}
\begin{proof}
Observe that the assignment $\op{T}$ identifies core symbols in the sense of \cite{GS2} with core symbols in the sense of  Definition~\ref{def:diagrammaticweight}.

By Definition~\ref{def:labelssimples1} we have $\la=(\la^\prime,\epsilon)$ and $\mu=(\mu^\prime,\epsilon^\prime)$ for some $\la^\prime,\mu^\prime\in X^+(\mg)$ and $\epsilon,\epsilon^\prime \in \{\pm\}$. Now by Corollary~\ref{corblockodd}, $P(\la)$ and $P(\mu)$ are in the same block if and only if  $\epsilon=\epsilon^\prime$ (that means $\sigma$ acts by the same scalar) and additionally $P^\mathfrak{g}(\la^\prime)$ and $P^\mathfrak{g}(\mu^\prime)$ are in the same block for $\cF'$. By \cite{GS2}, the latter holds precisely if the associated weight diagrams $\op{GS}(\la^\prime)$ and $\op{GS}(\mu^\prime)$ have the same core diagram in the sense of \cite{GS2}, and hence by Proposition~\ref{prop:GSES}  $\pla^\infty$ and $\pmu^\infty$  have the same core diagrams in the sense of Definition~\ref{def:diagrammaticweight}. Therefore $P(\la)$ and $P(\mu)$ are in the  same  block if and only if $\pla^\infty$ and $\pla^\infty$ have the same core diagrams and  $\#\up({\pla^\owedge})\equiv\#\up({\pmu^\owedge})\op{mod}2$, since this parity is given by $\epsilon$, by Definition~\ref{superweight}.
\end{proof}

\begin{prop}
\label{blocksindiagramseven}
Assume that $G=\OSPrn$ with $r$ even. Let $\la,\mu\in X^+(G)$. Then $P(\la)$ and $P(\mu)$ (and hence also $L(\la)$ and $L(\mu)$) are in the  same  block if and only if $\pla^\infty$ and $\pmu^\infty$ have the same core diagrams, see Definition~\ref{def:diagrammaticweight}, as well as $\#\up({\pla^\owedge})\equiv\#\up({\pmu^\owedge})\op{mod}2$ in case no $\Diamond$ occurs.
\end{prop}
\begin{remark}
In particular $P(\la)$ and $P(\mu)$ are in the same block if and only if $\pla^\owedge$ and $\pmu^\owedge$ are diagrammatically linked in the sense of Definition~\ref{dlinkage}.
\end{remark}

\begin{proof}
Assume $\op{Hom}_\cF(P(\la),P(\mu))\not=\{0\}$ then
\abovedisplayskip0.5em
\belowdisplayskip0.5em
\[
\op{Hom}_{\cF'}(\op{Res}^{\mg, G'}_{\mg, G}P(\la),\op{Res}^{\mg, G'}_{\mg, G}P(\mu))\not=\{0\}.
\] 
By Lemma~\ref{res} $\op{Res}^{\mg, G'}_{\mg, G}P(\la)$ and $\op{Res}^{\mg, G'}_{\mg, G}P(\mu)$ give rise to weight diagrams in the sense of \cite{GS2} which have the some core diagrams, \cite[Lemma 7]{GS1}, hence $\pla^\infty$ and $\pmu^\infty$ have the same core diagrams thanks to Proposition~\ref{prop:GSES}. Note that in case the restricted module contains more than one summand, the two summands give rise to the same core diagram.

In case that both $\la$ and $\mu$ do not contain a $\Diamond$ in their respective weight diagrams, it follows from Proposition~\ref{Voila} that the weights are of the form $(\lambda^\prime,\epsilon_\la)$ and $(\mu^\prime,\epsilon_\mu)$ for some $\la^\prime, \mu^\prime \in X^+(\mathfrak{g})$ and signs $\epsilon_\la$ and $\epsilon_\mu$. Furthermore one is in case (1) of Lemma~\ref{lemplusminus} and thus $\epsilon_\la = \epsilon_\mu$, which by Definition~\ref{superweight} implies that $\#\up({\pla^\owedge})\equiv\#\up({\pmu^\owedge})\op{mod}2$. This shows the "only if" direction of the claim.

For the "if" direction of the claim assume that both equality of core diagrams as well as the parity condition are fulfilled.

In case the core diagram contains a symbol at position zero, necessarily a $\circ$, it holds by case (1) of Lemma~\ref{lemplusminus} that
\abovedisplayskip0.5em
\belowdisplayskip0.5em
\[
{\rm Hom}_\mathcal{F}(P(\nu,\epsilon)),P(\eta,\epsilon^\prime)) \neq \{ 0\} \Longleftrightarrow 
{\rm Hom}_\mathcal{F^\prime}(P^\mathfrak{g}(\nu)),P^\mathfrak{g}(\eta)) \neq \{ 0\} \text{ and } \epsilon = \epsilon^\prime,
\]
for $\nu,\eta \in X^+(\mathfrak{g})$ with the same core diagram as $\lambda$ and $\mu$. Therefore the claim follows from \cite{GS2}, since the sets of weights with a fixed sign give rise to a block.

Assume now that the core diagram does not contain a symbol at zero, i.e. $\la$ and $\mu$ both have $\Diamond$ at position zero in their respective diagrammatic weights. Write $\lambda=(\lambda^\prime,\epsilon)$ and $\mu=(\mu^\prime,\epsilon^\prime)$ for $\la^\prime,\mu^\prime \in X^+(\mathfrak{g})$ and $\epsilon,\epsilon^\prime$ elements in the respective stabilisers. Since the core diagrams are the same, $P^\mg(\la^\prime)$ and $P^\mg(\mu^\prime)$ are in the same block by \cite[Lemma 7]{GS1} and hence are connected by a sequence of homomorphisms between projective modules. By Proposition~\ref{blocksO} this can be lifted to a sequence of morphisms between the modules $P(\la^\prime,\iota)$ and $P(\mu^\prime,\iota^\prime)$. Here $\iota$ and $\iota^\prime$ can differ from $\epsilon$ and $\epsilon^\prime$. In case any of them do have a sign, one notes that there is a module $P(\nu)$ with the same core diagram such that $\nu$ does not have a sign (this can be $\la$ or $\mu$). By Proposition~\ref{blocksO} and Lemma~\ref{lemplusminus} all the occurring modules $P(\la^\prime,\epsilon)$, $P(\la^\prime,\iota)$, $P(\mu^\prime,\epsilon^\prime)$, and $P(\mu^\prime,\iota^\prime)$ can be connected to $P(\nu)$ by a sequence of morphisms and hence all of them are in the same block. The claim follows.
\end{proof}

\begin{corollary}
Let $\la,\mu\in X^+(G)$. If $P(\la)$ and $P(\mu)$  are in the  same  block $\cB$ of $\cF$ then $\op{def}(\la)=\op{def}(\mu)$. In particular, one can talk about the {\it defect} $\op{def}(\cB)$ of a block $\cB$ of $\cF$.  
\end{corollary}
\begin{proof}
This follows directly from Propositions~\ref{blocksindiagramsodd} and~\ref{blocksindiagramseven}, and Corollary~\ref{lem:samedefect}.
\end{proof}

\begin{remark}
\label{remarkatyp}
Using the dictionary to \cite{GS2} which we have developed in \eqref{T2}, one can show that the defect is precisely the {\it atypicality} of the block in the sense of Lie superalgebras. We expect that, in contrast to the $\SOSP$-case treated in \cite[Theorem ~2]{GS1}, the blocks depend up to equivalence of categories only on the atypicality, see Section~\ref{sec:examples} for examples.
\end{remark}

To prove Proposition~\ref{prop:GSES} we recall some of the constructions from \cite{GS2}.

\subsection{The Gruson-Serganova combinatorics} \label{sec:GS_combinatorics}
We start by recalling the construction of the map $\op{GS}$ from \cite{GS2}. Recall the notion of vertices on $\cL$ as in Section~\ref{L}. The first map in \eqref{GSmap} takes a weight $\eta \in X^+(\mg)$ writes $\eta+\rho$ in the form \eqref{laab} and puts at the vertex $p$ of $\cL$ then $\alpha_p$ symbols $>$ and $\beta_p$ symbols $<$, where 
\abovedisplayskip0.25em
\belowdisplayskip0.75em
\[ \alpha_p=|\{1\leq j\leq m\mid a_j=\pm p\}| \text{ and } \beta_p=|\{1\leq i\leq n\mid b_i=\pm p\}|\]
and a symbol $\circ$ if $\alpha_p=\beta_p=0$. We use the abbreviation $\times$ for a pair $>$ and $<$ at a common vertex. We call the resulting diagram a $\op{GS}$-{\it diagram with tail}.

\textit{Case: $\osp(2m+1|2n)$:} In this case the dominance condition is equivalent to the statement that there is at most one symbol, $>$, $<$, $\times$ or $\circ$ at each vertex $p > \frac{1}{2}$ and at $\frac{1}{2}$ at most one $<$ or $>$, but possibly many $\times$. If there are only $\times$'s at $\frac{1}{2}$ we have to put an indicator which is $\mathbf{(+)}$ if $a_j=\frac{1}{2}$ for some $j$ and $\mathbf{(-)}$ otherwise. For instance, the diagram for the trivial weight are the following for $n>m$, $m=n$, $m>n$ respectively.
\belowdisplayskip0.5em
\begin{eqnarray}
\label{withtail}
\begin{tikzpicture}[anchorbase,thick,scale=0.8]
\begin{scope}
\node at (0,2) {$\scriptstyle \times$};
\node at (0,1.6) {$\scriptstyle \vdots$};
\node at (0,1) {$\scriptstyle \times$};
\node at (0,.5) {$\scriptstyle \times$};
\node at (0,0) {$\scriptstyle >$};
\node at (.5,0) {$\scriptstyle >$};
\node at (1,0) {$\scriptstyle \cdots$};
\node at (1.5,0) {$\scriptstyle >$};
\node at (2,0) {$\scriptstyle >$};
\node at (2.5,0) {$\scriptstyle \circ$};
\node at (3,0) {$\scriptstyle \circ$};
\node at (3.5,0) {$\scriptstyle \cdots$};
\draw[thin] (-.1,-.15) to +(0,-.1) to +(2.2,-.1) to +(2.2,0);
\node at (1,-.6) {$\scriptstyle m-n$};
\draw[thin] (-.15,2.1) to +(-.1,0) to +(-.1,-1.7) to +(0,-1.7);
\node at (-.6,1.3) {$\scriptstyle n$};
\end{scope}
\end{tikzpicture}
\quad
\begin{tikzpicture}[anchorbase,thick,scale=0.8]
\begin{scope}
\node at (0,2) {$\scriptstyle \times$};
\node at (0,1.5) {$\scriptstyle \times$};
\node at (0,1.1) {$\scriptstyle \vdots$};
\node at (0,.5) {$\scriptstyle \times$};
\node at (0,0) {$\scriptstyle \times$};
\node at (-.6,0) {$\scriptscriptstyle (-)$};
\node at (.5,0) {$\scriptstyle \circ$};
\node at (1,0) {$\scriptstyle \circ$};
\node at (1.5,0) {$\scriptstyle \circ$};
\node at (2,0) {$\scriptstyle \cdots$};
\draw[thin] (-.15,2.1) to +(-.1,0) to +(-.1,-2.2) to +(0,-2.2);
\node at (-.6,.8) {$\scriptstyle n$};
\node at (1,-.6) {$\scriptstyle \text{\phantom{m-n}}$};
\end{scope}
\end{tikzpicture}
\quad
\begin{tikzpicture}[anchorbase,thick,scale=0.8]
\begin{scope}
\node at (0,2) {$\scriptstyle \times$};
\node at (0,1.6) {$\scriptstyle \vdots$};
\node at (0,1) {$\scriptstyle \times$};
\node at (0,.5) {$\scriptstyle \times$};
\node at (0,0) {$\scriptstyle <$};
\node at (.5,0) {$\scriptstyle <$};
\node at (1,0) {$\scriptstyle \cdots$};
\node at (1.5,0) {$\scriptstyle <$};
\node at (2,0) {$\scriptstyle <$};
\node at (2.5,0) {$\scriptstyle \circ$};
\node at (3,0) {$\scriptstyle \circ$};
\node at (3.5,0) {$\scriptstyle \cdots$};
\draw[thin] (-.1,-.15) to +(0,-.1) to +(2.2,-.1) to +(2.2,0);
\node at (1,-.6) {$\scriptstyle n-m$};
\draw[thin] (-.15,2.1) to +(-.1,0) to +(-.1,-1.7) to +(0,-1.7);
\node at (-.6,1.3) {$\scriptstyle m$};
\end{scope}
\end{tikzpicture}
\end{eqnarray} 
The tail length is the number of $\times$ at the leftmost vertex, subtracting one if the indicator is $(+)$, and similarly the tail are all symbols $\times$ at position $\frac{1}{2}$ except for one if the indicator is $(+)$. 

\textit{Case: $\osp(2m|2n)$:} The dominance condition in this case is equivalent to the statement that there is at most one symbol, $>$, $<$, $\times$ or $\circ$ at each vertex $p>0$ and at $0$ either $\circ$ or at most one $>$ but possibly many $\times$. If there is a $\circ$ at $0$ one has to remember a sign to distinguish $a_m>0$ and $a_m<0$ (denoted by $[\pm]$ in \cite{GS2}). 
The trivial weight corresponds to the following (for $m > n$, $n\geq m$ respectively).
\belowdisplayskip0.5em
\begin{eqnarray}
\label{withtail2}
\begin{tikzpicture}[anchorbase,thick,scale=0.8]
\begin{scope}
\node at (0,2) {$\scriptstyle \times$};
\node at (0,1.6) {$\scriptstyle \vdots$};
\node at (0,1) {$\scriptstyle \times$};
\node at (0,.5) {$\scriptstyle \times$};
\node at (0,0) {$\scriptstyle >$};
\node at (.5,0) {$\scriptstyle >$};
\node at (1,0) {$\scriptstyle \cdots$};
\node at (1.5,0) {$\scriptstyle >$};
\node at (2,0) {$\scriptstyle >$};
\node at (2.5,0) {$\scriptstyle \circ$};
\node at (3,0) {$\scriptstyle \circ$};
\node at (3.5,0) {$\scriptstyle \cdots$};
\draw[thin] (-.1,-.15) to +(0,-.1) to +(2.2,-.1) to +(2.2,0);
\node at (1,-.6) {$\scriptstyle m-n$};
\draw[thin] (-.15,2.1) to +(-.1,0) to +(-.1,-1.7) to +(0,-1.7);
\node at (-.6,1.3) {$\scriptstyle n$};
\end{scope}
\end{tikzpicture}
\quad
\begin{tikzpicture}[anchorbase,thick,scale=0.8]
\begin{scope}
\node at (-.5,2) {$\scriptstyle \times$};
\node at (-.5,1.5) {$\scriptstyle \times$};
\node at (-.5,1.1) {$\scriptstyle \vdots$};
\node at (-.5,.5) {$\scriptstyle \times$};
\node at (-.5,0) {$\scriptstyle \times$};
\node at (0,0) {$\scriptstyle <$};
\node at (.5,0) {$\scriptstyle <$};
\node at (1,0) {$\scriptstyle \cdots$};
\node at (1.5,0) {$\scriptstyle <$};
\node at (2,0) {$\scriptstyle <$};
\node at (2.5,0) {$\scriptstyle \circ$};
\node at (3,0) {$\scriptstyle \circ$};
\node at (3.5,0) {$\scriptstyle \cdots$};
\draw[thin] (-.1,-.15) to +(0,-.1) to +(2.2,-.1) to +(2.2,0);
\node at (1,-.6) {$\scriptstyle n-m$};
\draw[thin] (-.65,2.1) to +(-.1,0) to +(-.1,-2.2) to +(0,-2.2);
\node at (-1.1,.8) {$\scriptstyle m$};
\end{scope}
\end{tikzpicture}
\end{eqnarray}
The tail length is the number of $\times$ at the leftmost vertex. 

For the second map \eqref{GSmap} we have to turn the diagram with tail into a coloured weight diagram.  
In case of $\osp(2m+1|2n)$ proceed as follows: First remove the tail of the diagram, but remember the number $l=\op{tail}(\eta)$, of symbols removed (note that in case of an indicator this can mean that one symbol $\times$ at position $\frac{1}{2}$ is kept). Ignoring the core symbols $<$ and $>$, connect neighboured pairs $\times$ $\circ$ (in this order) successively by a cup. Then number the vertices not connected to a cup and not containing $<$ or $>$ from the left by $1,2,3,\ldots$. Then relabel those positions with number $1,3,5, \ldots, 2l-1$ etc. by $\color{myred}{\boldsymbol \otimes}$. (The symbol $\color{myred}{\boldsymbol \otimes}$ indicates that at least apart from the special case of the leftmost vertex a $\times$ was actually moved and placed on top of a $\circ$). Finally connect neighboured pairs $\color{myred}{\boldsymbol \otimes}$ and $\circ$ successively by a cup. 

The resulting diagram with all labels at cups removed is the $\op{GS}$-{\it cup diagram} attached to $\eta$. In \cite{GS2} these new labels $\color{myred}{\boldsymbol \otimes}$ are called {\it coloured} and we call the attached cups {\it coloured}; note they are by construction never nested inside other cups.  
The resulting labelling of $\cL$ (after all cups are removed) is the {\it coloured $\op{GS}$-diagram without tail} attached to $\eta$.

In case of $\osp(2m+1|2n)$ proceed in the same way  but viewing the vertex $0$ as the vertex $\frac{1}{2}$ and always using the rule that if there are only $\times$ at position zero the indicator is $(+)$. Note that whether $a_m$ is strictly larger or smaller than $0$ does not play a role in the construction of the diagram.
\begin{lemma}
\label{lem:special_empty}
With the assignment $\op{T}$, from \eqref{T2}, we have $\op{T}(\op{GS}(0))=\ulcorner{\!0}^\infty$, and Proposition~\ref{prop:GSES} holds for $\eta=0$.
\end{lemma}
\begin{proof}
\noindent \textit{$\triangleright$ \textbf{Case} $\osp(2m+1|2n)$:} The weights from the diagrams  \eqref{withtail} with tail are transferred into the cup diagram with $m$, respectively $n$ in the last case, coloured cups placed next to each other starting at position $-\frac{\de}{2}+1$, $0$, and $\frac{\de}{2}$ respectively. On the other hand, our diagrammatics assigns to the empty partition the diagrammatic weights \eqref{empty1} and hence produce a cup diagrams with $n$, respectively $m$ in the last case, dotted cups placed next to each other starting at position $-\frac{\de}{2}+1$, $0$ and $\frac{\de}{2}$ respectively, see \eqref{emptycupoddcase}. The corresponding coloured weight diagram contains the $>$'s and $<$'s at the correct positions and only $\circ$ and $\color{myred}{\boldsymbol \otimes}$ at the positions of the cups.  Applying $\op{T}$ this translates into the diagrammatic weights \eqref{empty1}. Hence the claim is true in case $\osp(2m+1|2n)$.\\[0.1cm]
\noindent \textit{$\triangleright$ \textbf{Case} $\osp(2m |2n)$:} In this case the diagrams \eqref{withtail2} with tail are transformed into a cup diagram with $n$, respectively $m-1$ in the second case, coloured cups placed next to each other starting at positions $\frac{\de}{2}$, respectively $-\frac{\de}{2}+2$. In the latter case it also contains one uncoloured cup connecting position zero and $-\frac{\de}{2}+1$. Using our diagrammatics will produce the weight diagrams in \eqref{empty2}, which in turn produce cup diagrams with $n$, respectively $m-1$ dotted cups placed next to each other starting at positions $\frac{\de}{2}$, respectively $-\frac{\de}{2}+2$, see \eqref{emptycupevencase}. 
\end{proof}

\subsection{The proof of Proposition~\ref{prop:GSES}} \label{sec:proofcomb}
The proof proceeds by induction on the number of boxes in the corresponding hook partition (where we are allowed to ignore the sign).
\begin{proof}[Proof of Proposition~\ref{prop:GSES}]
In case of the empty partition the claim follows from Lemma~\ref{lem:special_empty} above.\\[0.1cm]
\noindent \textit{Adding a box:} We assume that the hook partition $\pla$ for $\la$ is obtained from a hook partition $\pmu$, for some $\mu$, by adding a box. Then for $\pmu$ the following holds
\begin{itemize}
\item If $b_i > 0$ and we can add a box in row $i$, then $b_j > b_i + 1$ for all $j<i$. This implies that in the $\op{GS}$ weight there is no symbol $<$ or $\times$ at positions $b_i + 1$, i.e. immediately to the right of $b_i$.
\item If $a_i > 0$ and we can add a box in column $i$, then $a_j > a_i + 1$ for all $j<i$. This implies that in the $\op{GS}$ weight there is no symbol $>$ or $\times$ at positions $a_i + 1$.
\item There can be a symbol $\color{myred}{\boldsymbol \otimes}$ at the position $a_i+1$ respectively $b_i+1$.
\end{itemize}

Assume first that the box is added {\it far away}, which means not on or directly next to the diagonal. In these cases we need not distinguish between odd and even.

\subsubsection*{The additional box is added far above the diagonal} We add the box in position $(j_0,i_0)$ above the diagonal, but not directly above the diagonal. In both the even and odd case $b_{j_0} > \frac{1}{2}$, and it will be increased by $1$ whereas all other $a$'s and $b$'s are preserved. This means a symbol $<$ is moved to the right from position $b_{j_0}$ to $b_{j_0}+1$. Note that if the symbol $<$ is part of a $\times$ there cannot be a symbol $\color{myred}{\boldsymbol \otimes}$ at position $b_{j_0}+1$ by construction of the coloured $\op{GS}$-diagram.

Hence we are in exactly one of the situations listed in the row \eqref{GST1a} in the table below where the symbols not in brackets are placed at positions $b_{j_0}$ and $b_{j_0}+1$.  Applying $\op{T}$ to  \eqref{GST1a} gives  \eqref{GST1b}.
\begin{eqnarray}
\label{GST1a}
&{\arraycolsep=2pt\begin{array}{r|c|c|c|c|c|c}
\begin{tikzpicture}[anchorbase,scale=.5]
\node at (-1.5,-.15) {$\mu$};
\node at (-1.5,1.6) {$\lambda$};
\end{tikzpicture} & 
\begin{tikzpicture}[anchorbase,scale=.5]
\node at (0,0) {$<$};
\node at (1,0){$\circ$};
\draw (1,.3) to [out=90,in=-90] +(-1,.9);
\node at (0,1.5) {$\circ$};
\node at (1,1.5){$<$};
\end{tikzpicture} & 
\begin{tikzpicture}[anchorbase,scale=.5]
\node at (0,0) {$<$};
\node at (1,0){$>$};
\node at (2,0){$ (\circ)$};
\draw (2,.3) to [out=90,in=-90] +(-2,.9);
\draw (1,1.25) .. controls +(0,-.5) and +(0,-.5) .. +(1,0);
\node at (0,1.5) {$\circ$};
\node at (1,1.5){$\times$};
\node at (2,1.5){$ (\circ)$};
\end{tikzpicture}&
\begin{tikzpicture}[anchorbase,scale=.5]
\node at (0,0) {$<$};
\node at (1,0){$>$};
\node at (2,0){$(\color{myred}{\boldsymbol \otimes}$\phantom{)}};
\node at (3,0){\phantom{(}$\circ)$};
\draw (2,.3) to [out=90,in=-90] +(-2,.9);
\draw (1,1.25) .. controls +(0,-.5) and +(0,-.5) .. +(1,0);
\draw (2,-.3) .. controls +(0,-.5) and +(0,-.5) .. +(1,0);
\draw (3,.3) -- +(0,.9);
\node at (0,1.5) {$\color{myred}{\boldsymbol \otimes}$};
\node at (1,1.5){$\times$};
\node at (2,1.5){$(\circ$\phantom{)}};
\node at (3,1.5){\phantom{(}$\circ)$};
\end{tikzpicture}&
\begin{tikzpicture}[anchorbase,scale=.5]
\node at (0,0) {$<$};
\node at (1,0){$\color{myred}{\boldsymbol \otimes}$};
\node at (2,0){$ (\circ)$};
\draw (1,.3) to [out=90,in=-90] +(-1,.9);
\draw (1,-.3) .. controls +(0,-.5) and +(0,-.5) .. +(1,0);
\draw (2,.3) -- +(0,.9);
\node at (0,1.5) {$\color{myred}{\boldsymbol \otimes}$};
\node at (1,1.5){$<$};
\node at (2,1.5){$ (\circ)$};
\end{tikzpicture}&
\begin{tikzpicture}[anchorbase,scale=.5]
\node at (0,0) {$\times$};
\node at (1,0){$\circ$};
\draw (0,.3) .. controls +(0,.5) and +(0,.5) .. +(1,0);
\node at (0,1.5) {$>$};
\node at (1,1.5){$<$};
\end{tikzpicture}&
\begin{tikzpicture}[anchorbase,scale=.5]
\node at (0,0) {$\times$};
\node at (1,0){$>$};
\node at (2,0){$(\circ)$};
\draw (0,.3) to [out=90,in=-90] +(1,.9);
\draw (0,-.3) .. controls +(0,-.5) and +(0,-.5) .. +(2,0);
\draw (2,.3) -- +(0,.9);
\node at (0,1.5) {$>$};
\node at (1,1.5){$\times$};
\node at (2,1.5){$(\circ)$};
\end{tikzpicture}\\
\end{array}}
\\
\label{GST1b}
&{\arraycolsep=2pt
\begin{array}{r|c|c|c|c|c|c}
\begin{tikzpicture}[anchorbase,scale=.5]
\node at (-1.5,-.15) {$\mu$};
\node at (-1.5,1.6) {$\lambda$};
\end{tikzpicture} & 
\begin{tikzpicture}[anchorbase,scale=.5]
\node at (0,0) {$\times$};
\node at (1,0){$\up$};
\draw (1,.3) to [out=90,in=-90] +(-1,.9);
\node at (0,1.5) {$\up$};
\node at (1,1.5){$\times$};
\end{tikzpicture} & 
\begin{tikzpicture}[anchorbase,scale=.5]
\node at (0,0) {$\times$};
\node at (1,0){$\circ$};
\node at (2,0){$ (\up)$};
\draw (2,.3) to [out=90,in=-90] +(-2,.9);
\draw (1,1.25) .. controls +(0,-.5) and +(0,-.5) .. +(1,0);
\node at (0,1.5) {$\up$};
\node at (1,1.5){$\down$};
\node at (2,1.5){$ (\up)$};
\end{tikzpicture}&
\begin{tikzpicture}[anchorbase,scale=.5]
\node at (0,0) {$\times$};
\node at (1,0){$\circ$};
\node at (2,0){$(\up$\phantom{)}};
\node at (3,0){\phantom{(}$\up)$};
\draw (2,.3) to [out=90,in=-90] +(-2,.9);
\draw (1,1.25) .. controls +(0,-.5) and +(0,-.5) .. +(1,0);
\draw (2,-.3) .. controls +(0,-.5) and +(0,-.5) .. +(1,0);
\draw (3,.3) -- +(0,.9);
\fill (2.5,-.67) circle(3.5pt);
\node at (0,1.5) {$\up$};
\node at (1,1.5){$\down$};
\node at (2,1.5){$(\up$\phantom{)}};
\node at (3,1.5){\phantom{(}$\up)$};
\end{tikzpicture}&
\begin{tikzpicture}[anchorbase,scale=.5]
\node at (0,0) {$\times$};
\node at (1,0){$\up$};
\node at (2,0){$ (\up)$};
\fill (1.5,-.67) circle(3.5pt);
\draw (1,.3) to [out=90,in=-90] +(-1,.9);
\draw (1,-.3) .. controls +(0,-.5) and +(0,-.5) .. +(1,0);
\draw (2,.3) -- +(0,.9);
\node at (0,1.5) {${\up}$};
\node at (1,1.5){$\times$};
\node at (2,1.5){$ (\up)$};
\end{tikzpicture}&
\begin{tikzpicture}[anchorbase,scale=.5]
\node at (0,0) {$\down$};
\node at (1,0){$\up$};
\draw (0,.3) .. controls +(0,.5) and +(0,.5) .. +(1,0);
\node at (0,1.5) {$\circ$};
\node at (1,1.5){$\times$};
\end{tikzpicture}&
\begin{tikzpicture}[anchorbase,scale=.5]
\node at (0,0) {$\down$};
\node at (1,0){$\circ$};
\node at (2,0){$(\up)$};
\draw (0,.3) to [out=90,in=-90] +(1,.9);
\draw (0,-.3) .. controls +(0,-.5) and +(0,-.5) .. +(2,0);
\draw (2,.3) -- +(0,.9);
\node at (0,1.5) {$\circ$};
\node at (1,1.5){$\down$};
\node at (2,1.5){$(\up)$};
\end{tikzpicture}\\
\end{array}}
\end{eqnarray}

On the other hand,  since $b_{j_0}>0$ with $b_{j_0}=\mu_{j_0}-j_0-\frac{\de}{2}+1$ we have   $\cS(\pmu)_{j_0}=-b_{j_0}$. By assumption this will be decreased by $1$ when passing to $\pla$. Hence we have either the symbol $\down$ or $\times$ at position $b_{j_0}$ and the symbol $\down$ gets moved to the right. Furthermore $\cS(\pmu)_{j} < \cS(\pmu)_{j_0} - 1$, since $\mu_j > \mu_{j_0}$ for $j<j_0$, which in turn implies that at position $b_{j_0}+1$ there is the symbol $\up$ or $\circ$.

In all of the listed cases neither the tail length nor the number of dotted cups is changed, thus all fake cups are unchanged, and if an $\up$ in $\pmu^\infty$ is frozen and moved, it is still frozen in $\pla^\infty$. Hence the claim follows in this case.

\subsubsection*{The box is added far below the diagonal} We add the box in position $(j_0,i_0)$ below the diagonal and not adjacent to the diagonal. 
In this case $a_{i_0}>\frac{1}{2}$. It will be increased by one whereas all other $a_i$'s and $b_i$'s are left unchanged. Thus we move a symbol $>$ to the right. As before, there cannot be the symbol $\color{myred}{\boldsymbol \otimes}$ at position $a_{i_0}+1$ if there is the symbol $\times$ at position $a_{i_0}$. In total this gives us the configurations in the first row \eqref{GST2a} below (showing the positions $a_{i_0}$ and $a_{i_0}+1$ without brackets). The second row \eqref{GST2b} shows then the image under $\op{T}$. 
\abovedisplayskip0.5em
\belowdisplayskip0.5em
\begin{eqnarray}
\label{GST2a}
&{\arraycolsep=2pt
\begin{array}{r|c|c|c|c|c|c}
\begin{tikzpicture}[anchorbase,scale=.5]
\node at (-1.5,-.15) {$\mu$};
\node at (-1.5,1.6) {$\lambda$};
\end{tikzpicture} & 
\begin{tikzpicture}[anchorbase,scale=.5]
\node at (0,0) {$>$};
\node at (1,0){$\circ$};
\draw (1,.3) to [out=90,in=-90] +(-1,.9);
\node at (0,1.5) {$\circ$};
\node at (1,1.5){$>$};
\end{tikzpicture} & 
\begin{tikzpicture}[anchorbase,scale=.5]
\node at (0,0) {$>$};
\node at (1,0){$<$};
\node at (2,0){$ (\circ)$};
\draw (2,.3) to [out=90,in=-90] +(-2,.9);
\draw (1,1.25) .. controls +(0,-.5) and +(0,-.5) .. +(1,0);
\node at (0,1.5) {$\circ$};
\node at (1,1.5){$\times$};
\node at (2,1.5){$ (\circ)$};
\end{tikzpicture}&
\begin{tikzpicture}[anchorbase,scale=.5]
\node at (0,0) {$>$};
\node at (1,0){$<$};
\node at (2,0){$(\color{myred}{\boldsymbol \otimes}$\phantom{)}};
\node at (3,0){\phantom{(}$\circ)$};
\draw (2,.3) to [out=90,in=-90] +(-2,.9);
\draw (1,1.25) .. controls +(0,-.5) and +(0,-.5) .. +(1,0);
\draw (2,-.3) .. controls +(0,-.5) and +(0,-.5) .. +(1,0);
\draw (3,.3) -- +(0,.9);
\node at (0,1.5) {$\color{myred}{\boldsymbol \otimes}$};
\node at (1,1.5){$\times$};
\node at (2,1.5){$(\circ$\phantom{)}};
\node at (3,1.5){\phantom{(}$\circ)$};
\end{tikzpicture}&
\begin{tikzpicture}[anchorbase,scale=.5]
\node at (0,0) {$>$};
\node at (1,0){$\color{myred}{\boldsymbol \otimes}$};
\node at (2,0){$ (\circ)$};
\draw (1,.3) to [out=90,in=-90] +(-1,.9);
\draw (1,-.3) .. controls +(0,-.5) and +(0,-.5) .. +(1,0);
\draw (2,.3) -- +(0,.9);
\node at (0,1.5) {$\color{myred}{\boldsymbol \otimes}$};
\node at (1,1.5){$>$};
\node at (2,1.5){$ (\circ)$};
\end{tikzpicture}&
\begin{tikzpicture}[anchorbase,scale=.5]
\node at (0,0) {$\times$};
\node at (1,0){$\circ$};
\draw (0,.3) .. controls +(0,.5) and +(0,.5) .. +(1,0);
\node at (0,1.5) {$<$};
\node at (1,1.5){$>$};
\end{tikzpicture}&
\begin{tikzpicture}[anchorbase,scale=.5]
\node at (0,0) {$\times$};
\node at (1,0){$<$};
\node at (2,0){$(\circ)$};
\draw (0,.3) to [out=90,in=-90] +(1,.9);
\draw (0,-.3) .. controls +(0,-.5) and +(0,-.5) .. +(2,0);
\draw (2,.3) -- +(0,.9);
\node at (0,1.5) {$<$};
\node at (1,1.5){$\times$};
\node at (2,1.5){$(\circ)$};
\end{tikzpicture}\\
\end{array}}
\\
\label{GST2b}
&{\arraycolsep=2pt
\begin{array}{r|c|c|c|c|c|c}
\begin{tikzpicture}[anchorbase,scale=.5]
\node at (-1.5,-.15) {$\mu$};
\node at (-1.5,1.6) {$\lambda$};
\end{tikzpicture} & 
\begin{tikzpicture}[anchorbase,scale=.5]
\node at (0,0) {$\circ$};
\node at (1,0){$\up$};
\draw (1,.3) to [out=90,in=-90] +(-1,.9);
\node at (0,1.5) {$\up$};
\node at (1,1.5){$\circ$};
\end{tikzpicture} & 
\begin{tikzpicture}[anchorbase,scale=.5]
\node at (0,0) {$\circ$};
\node at (1,0){$\times$};
\node at (2,0){$ (\up)$};
\draw (2,.3) to [out=90,in=-90] +(-2,.9);
\draw (1,1.25) .. controls +(0,-.5) and +(0,-.5) .. +(1,0);
\node at (0,1.5) {$\up$};
\node at (1,1.5){$\down$};
\node at (2,1.5){$ (\up)$};
\end{tikzpicture}&
\begin{tikzpicture}[anchorbase,scale=.5]
\node at (0,0) {$\circ$};
\node at (1,0){$\times$};
\node at (2,0){$(\up$\phantom{)}};
\node at (3,0){\phantom{(}$\up)$};
\draw (2,.3) to [out=90,in=-90] +(-2,.9);
\draw (1,1.25) .. controls +(0,-.5) and +(0,-.5) .. +(1,0);
\draw (2,-.3) .. controls +(0,-.5) and +(0,-.5) .. +(1,0);
\fill (2.5,-.67) circle(3.5pt);
\draw (3,.3) -- +(0,.9);
\node at (0,1.5) {$\up$};
\node at (1,1.5){$\down$};
\node at (2,1.5){$(\up$\phantom{)}};
\node at (3,1.5){\phantom{(}$\up)$};
\end{tikzpicture}&
\begin{tikzpicture}[anchorbase,scale=.5]
\node at (0,0) {$\circ$};
\node at (1,0){$\up$};
\node at (2,0){$ (\up)$};
\fill (1.5,-.67) circle(3.5pt);\draw (1,.3) to [out=90,in=-90] +(-1,.9);
\draw (1,-.3) .. controls +(0,-.5) and +(0,-.5) .. +(1,0);
\draw (2,.3) -- +(0,.9);
\node at (0,1.5) {${\up}$};
\node at (1,1.5){$\circ$};
\node at (2,1.5){$ (\up)$};
\end{tikzpicture}&
\begin{tikzpicture}[anchorbase,scale=.5]
\node at (0,0) {$\down$};
\node at (1,0){$\up$};
\draw (0,.3) .. controls +(0,.5) and +(0,.5) .. +(1,0);
\node at (0,1.5) {$\times$};
\node at (1,1.5){$\circ$};
\end{tikzpicture}&
\begin{tikzpicture}[anchorbase,scale=.5]
\node at (0,0) {$\down$};
\node at (1,0){$\times$};
\node at (2,0){$(\up)$};
\draw (0,.3) to [out=90,in=-90] +(1,.9);
\draw (0,-.3) .. controls +(0,-.5) and +(0,-.5) .. +(2,0);
\draw (2,.3) -- +(0,.9);
\node at (0,1.5) {$\times$};
\node at (1,1.5){$\down$};
\node at (2,1.5){$(\up)$};
\end{tikzpicture}\\
\end{array}}
\end{eqnarray}
On the other hand, note that since $a_{i_0}>0$ we have $a_{i_0} = \mu_{i_0}^t - i_0 + \frac{\de}{2} = j_0 - 1 - i_0 + \frac{\de}{2}$. This implies $\cS(\pmu)_{j_0} = \frac{\de}{2}+j_0-\mu_{i_0}-1 = a_{i_0} + 1 > 0$. Adding the box will decreased this by $1$ (since $\mu_{i_0}$ is increased by $1$). Thus we have the symbol $\up$ or $\times$ at position $a_{i_0}+1$ with $\up$ moved to the left. Furthermore $\cS(\pmu)_j < \cS(\pmu)_{j_0} - 1$ since $\mu_j > \mu_{j_0}$ for $j<j_0$, which in turn implies that at position $a_{i_0}$ there is either the symbol $\down$ or $\circ$. Again, in all cases neither the tail length nor the number of dotted cups changes, thus all fake cups are unchanged. Additionally if an $\up$ in $\pmu^\infty$ is frozen and moved in $\pla^\infty$ it will be frozen in $\pla^\infty$ as well.\\

For the remaining cases we have to distinguish between $r$ being odd or even.\\

\noindent \textit{$\triangleright$ \textbf{Case} $\osp(2m+1|2n)$:} We distinguish three possibilities: adding the box exactly above the shifted diagonal, adding the box exactly below the shifted diagonal, and adding the box on the diagonal.

\subsubsection*{The additional box is added directly above the diagonal} We add the box in position $(j_0,i_0)$ directly above the diagonal. Thus, $i_0-j_0 = \frac{\delta}{2}+\frac{1}{2}$. In addition $b_{j_0} = \frac{1}{2}$ and it will be increased by $1$, while all other $a$'s and $b$'s are left unchanged. Thus a symbol $>$ is moved from position $\frac{1}{2}$ to position $\frac{3}{2}$. If this symbol is not part of a symbol $\times$ then the arguments are the same as for adding a box far above the diagonal and we refer to that case. If on the other hand it is part of a $\times$ this implies that the indicator is $(+)$ since $a_{i_0-1}=\frac{1}{2}$. Thus the $\times$ at position $\frac{1}{2}$ is not coloured and then the two possible situations are displayed in \eqref{GST4} on the left and image under $\op{T}$ is displayed on the right. 
\abovedisplayskip0.5em
\belowdisplayskip0.5em
\begin{equation}
\begin{array}{r|c|c}
\begin{tikzpicture}[anchorbase,scale=.5]
\node at (-1.5,-.15) {$\mu$};
\node at (-1.5,1.6) {$\lambda$};
\end{tikzpicture} & 
\begin{tikzpicture}[anchorbase,scale=.5]
\node at (0,0) {$\times$};
\node at (1,0){$\circ$};
\draw (0,.3) .. controls +(0,.5) and +(0,.5) .. +(1,0);
\draw (0,-.3) .. controls +(0,-.5) and +(0,-.5) .. +(1,0);
\node at (0,1.5) {$<$};
\node at (1,1.5){$>$};
\end{tikzpicture} & 
\begin{tikzpicture}[anchorbase,scale=.5]
\node at (0,0) {$\times$};
\node at (1,0){$<$};
\node at (2,0){$(\circ)$};
\draw (0,.3) to [out=90,in=-90] +(1,.9);
\draw (0,-.3) .. controls +(0,-.5) and +(0,-.5) .. +(2,0);
\draw (2,.3) -- +(0,.9);
\node at (0,1.5) {$<$};
\node at (1,1.5){$\times$};
\node at (2,1.5){$(\circ)$};
\end{tikzpicture}\\
\end{array}
\qquad \qquad 
\begin{array}{r|c|c}
\begin{tikzpicture}[anchorbase,scale=.5]
\node at (-1.5,-.15) {$\mu$};
\node at (-1.5,1.6) {$\lambda$};
\end{tikzpicture} & 
\begin{tikzpicture}[anchorbase,scale=.5]
\node at (0,0) {$\down$};
\node at (1,0){$\up$};
\draw (0,.3) .. controls +(0,.5) and +(0,.5) .. +(1,0);
\draw (0,-.3) .. controls +(0,-.5) and +(0,-.5) .. +(1,0);
\node at (0,1.5) {$\circ$};
\node at (1,1.5){$\times$};
\end{tikzpicture} & 
\begin{tikzpicture}[anchorbase,scale=.5]
\node at (0,0) {$\down$};
\node at (1,0){$\circ$};
\node at (2,0){$(\up)$};
\draw (0,.3) to [out=90,in=-90] +(1,.9);
\draw (0,-.3) .. controls +(0,-.5) and +(0,-.5) .. +(2,0);
\draw (2,.3) -- +(0,.9);
\node at (0,1.5) {$\circ$};
\node at (1,1.5){$\down$};
\node at (2,1.5){$(\up)$};
\end{tikzpicture}\\
\end{array}
\label{GST4}
\end{equation}
On the other hand note that $\cS(\pmu)_{j_0}=-\frac{1}{2}$ which will be decreased to $-\frac{3}{2}$ by adding the box. Again neither tail length nor number of dotted cups changes. The claim follows in this case.

\subsubsection*{The additional box is added directly below the diagonal} We add the box in position $(j_0,i_0)$ directly below the diagonal, thus $i_0-j_0 = \frac{\delta}{2}-\frac{3}{2}$. In addition $a_{i_0} = \frac{1}{2}$ and it will be increased by $1$, while all other $a$'s and $b$'s are left unchanged. Thus a symbol $<$ is moved from position $\frac{1}{2}$ to position $\frac{3}{2}$. If this symbol is not part of a symbol $\times$ then the arguments are the same as for adding a box far below the diagonal and we refer to that case. As in that case the indicator is $(+)$, since $a_{i_0}=\frac{1}{2}$, and we obtain the two possibilities displayed on the left below:
\abovedisplayskip0.5em
\belowdisplayskip0.5em
\begin{equation*}
\begin{array}{r|c|c}
\begin{tikzpicture}[anchorbase,scale=.5]
\node at (-1.5,-.15) {$\mu$};
\node at (-1.5,1.6) {$\lambda$};
\end{tikzpicture} & 
\begin{tikzpicture}[anchorbase,scale=.5]
\node at (0,0) {$\times$};
\node at (1,0){$\circ$};
\draw (0,.3) .. controls +(0,.5) and +(0,.5) .. +(1,0);
\draw (0,-.3) .. controls +(0,-.5) and +(0,-.5) .. +(1,0);
\node at (0,1.5) {$>$};
\node at (1,1.5){$<$};
\end{tikzpicture} & 
\begin{tikzpicture}[anchorbase,scale=.5]
\node at (0,0) {$\times$};
\node at (1,0){$>$};
\node at (2,0){$(\circ)$};
\draw (0,.3) to [out=90,in=-90] +(1,.9);
\draw (0,-.3) .. controls +(0,-.5) and +(0,-.5) .. +(2,0);
\draw (2,.3) -- +(0,.9);
\node at (0,1.5) {$>$};
\node at (1,1.5){$\times$};
\node at (2,1.5){$(\circ)$};
\end{tikzpicture}\\
\end{array}
\qquad \qquad 
\begin{array}{r|c|c}
\begin{tikzpicture}[anchorbase,scale=.5]
\node at (-1.5,-.15) {$\mu$};
\node at (-1.5,1.6) {$\lambda$};
\end{tikzpicture} & 
\begin{tikzpicture}[anchorbase,scale=.5]
\node at (0,0) {$\down$};
\node at (1,0){$\up$};
\draw (0,.3) .. controls +(0,.5) and +(0,.5) .. +(1,0);
\draw (0,-.3) .. controls +(0,-.5) and +(0,-.5) .. +(1,0);
\node at (0,1.5) {$\times$};
\node at (1,1.5){$\circ$};
\end{tikzpicture} & 
\begin{tikzpicture}[anchorbase,scale=.5]
\node at (0,0) {$\down$};
\node at (1,0){$\times$};
\node at (2,0){$(\up)$};
\draw (0,.3) to [out=90,in=-90] +(1,.9);
\draw (0,-.3) .. controls +(0,-.5) and +(0,-.5) .. +(2,0);
\draw (2,.3) -- +(0,.9);
\node at (0,1.5) {$\times$};
\node at (1,1.5){$\down$};
\node at (2,1.5){$(\up)$};
\end{tikzpicture}\\
\end{array}
\end{equation*}
Applying $\op{T}$ givves the weight diagrams displayed on the right. On the other hand, note that $\cS(\pmu)_{j_0}=\frac{3}{2}$ which will be decreased to $\frac{1}{2}$. The rest of the argument is the same as before.

\subsubsection*{The additional box is added on the diagonal} We add the box in position $(j_0,i_0)$ on the diagonal. Thus $i_0-j_0 = \frac{\delta}{2}-\frac{1}{2}$. In addition $a_{i_0} = -\frac{1}{2}$ which will be increased by $1$, while all other $a$'s and $b$'s are left unchanged.
The $\frac{1}{2}$ position for $\mu$ contains only the symbol $\times$ and the indicator is $(-)$ since $a_{i_0} = -\frac{1}{2}$. Thus in the GS-combinatorics adding the box on the diagonal does not change the diagrammatic weight itself but the indicator from $(-)$ to $(+)$. Which decreases the tail length by $1$. On the other hand in this case $\cS(\pmu)_{j_0}=\frac{1}{2}$. It will be decreased to $-\frac{1}{2}$ and thus the first cup gets changed from a dotted cup to an undotted cup (while preserving all frozen variables).\\

\noindent \textit{$\triangleright$ \textbf{Case} $\osp(2m|2n)$:} Again we distinguish three scenarios as above.

\subsubsection*{The additional box is added directly above the diagonal} We add the box in position $(j_0,i_0)$ directly above the diagonal. Thus, $i_0-j_0 = \frac{\delta}{2}+1$. In addition $b_{j_0} = 1$ and it will be increased by $1$, while all other $a$'s and $b$'s are left unchanged. Then we can argue as in the case of adding a box far above the diagonal.

\subsubsection*{The additional box is added directly below the diagonal} We add the box in position $(j_0,i_0)$ directly below the diagonal. Thus, $i_0-j_0 = \frac{\delta}{2}-1$. In addition $a_{i_0} = 0$ and it will be increased by $1$, while all other $a$'s and $b$'s are left unchanged.

Note that $b_{j_0-1} > 0$. Furthermore the rest of the diagonal to the lower right is empty, implying $a_i=0$ for $i > i_0$ and $b_j = 0$ for $j > j_0-1$, which implies that in the tail we have exactly once the symbol $>$ and possibly some $\times$. The $\times$ are distributed to obtain the coloured diagram, leaving $>$ at position zero. This leads to the following configurations (at positions zero and $1$, the rest is unchanged) displayed in \eqref{GST3a} with the image under $\op{T}$ displayed in \eqref{GST3b}:
\begin{eqnarray}
\label{GST3a}
&\begin{array}{r|c|c|c|c}
\begin{tikzpicture}[anchorbase,scale=.5]
\node at (-1.5,-.15) {$\mu$};
\node at (-1.5,1.6) {$\lambda$};
\end{tikzpicture} & 
\begin{tikzpicture}[anchorbase,scale=.5]
\node at (0,0) {$>$};
\node at (1,0){$\circ$};
\draw (1,.3) to [out=90,in=-90] +(-1,.9);
\node at (0,1.5) {$\circ$};
\node at (1,1.5){$>$};
\end{tikzpicture} & 
\begin{tikzpicture}[anchorbase,scale=.5]
\node at (0,0) {$>$};
\node at (1,0){$<$};
\node at (2,0){$ (\circ)$};
\draw (2,.3) to [out=90,in=-90] +(-2,.9);
\draw (1,1.25) .. controls +(0,-.5) and +(0,-.5) .. +(1,0);
\node at (0,1.5) {$\circ$};
\node at (1,1.5){$\times$};
\node at (2,1.5){$ (\circ)$};
\end{tikzpicture}&
\begin{tikzpicture}[anchorbase,scale=.5]
\node at (0,0) {$>$};
\node at (1,0){$<$};
\node at (2,0){$(\color{myred}{\boldsymbol \otimes}$\phantom{)}};
\node at (3,0){\phantom{(}$\circ)$};
\draw (2,.3) to [out=90,in=-90] +(-2,.9);
\draw (1,1.25) .. controls +(0,-.5) and +(0,-.5) .. +(1,0);
\draw (2,-.3) .. controls +(0,-.5) and +(0,-.5) .. +(1,0);
\draw (3,.3) -- +(0,.9);
\node at (0,1.5) {$\color{myred}{\boldsymbol \otimes}$};
\node at (1,1.5){$\times$};
\node at (2,1.5){$(\circ$\phantom{)}};
\node at (3,1.5){\phantom{(}$\circ)$};
\end{tikzpicture}&
\begin{tikzpicture}[anchorbase,scale=.5]
\node at (0,0) {$>$};
\node at (1,0){$\color{myred}{\boldsymbol \otimes}$};
\node at (2,0){$ (\circ)$};
\draw (1,.3) to [out=90,in=-90] +(-1,.9);
\draw (1,-.3) .. controls +(0,-.5) and +(0,-.5) .. +(1,0);
\draw (2,.3) -- +(0,.9);
\node at (0,1.5) {$\color{myred}{\boldsymbol \otimes}$};
\node at (1,1.5){$>$};
\node at (2,1.5){$ (\circ)$};
\end{tikzpicture}\\
\end{array}
\\
\label{GST3b}
&\begin{array}{r|c|c|c|c}
\begin{tikzpicture}[anchorbase,scale=.5]
\node at (-1.5,-.15) {$\mu$};
\node at (-1.5,1.6) {$\lambda$};
\end{tikzpicture} & 
\begin{tikzpicture}[anchorbase,scale=.5]
\node at (0,0) {$\circ$};
\node at (1,0){$\up$};
\draw (1,.3) to [out=90,in=-90] +(-1,.9);
\node at (0,1.5) {$\Diamond$};
\node at (1,1.5){$\circ$};
\end{tikzpicture} & 
\begin{tikzpicture}[anchorbase,scale=.5]
\node at (0,0) {$\circ$};
\node at (1,0){$\times$};
\node at (2,0){$ (\up)$};
\draw (2,.3) to [out=90,in=-90] +(-2,.9);
\draw (1,1.25) .. controls +(0,-.5) and +(0,-.5) .. +(1,0);
\node at (0,1.5) {$\Diamond$};
\node at (1,1.5){$\down$};
\node at (2,1.5){$ (\up)$};
\end{tikzpicture}&
\begin{tikzpicture}[anchorbase,scale=.5]
\node at (0,0) {$\circ$};
\node at (1,0){$\times$};
\node at (2,0){$(\up$\phantom{)}};
\node at (3,0){\phantom{(}$\up)$};
\draw (2,.3) to [out=90,in=-90] +(-2,.9);
\draw (1,1.25) .. controls +(0,-.5) and +(0,-.5) .. +(1,0);
\draw (2,-.3) .. controls +(0,-.5) and +(0,-.5) .. +(1,0);
\draw (3,.3) -- +(0,.9);
\fill (2.5,-.67) circle(3.5pt);
\node at (0,1.5) {$\Diamond$};
\node at (1,1.5){$\down$};
\node at (2,1.5){$(\up$\phantom{)}};
\node at (3,1.5){\phantom{(}$\up)$};
\end{tikzpicture}&
\begin{tikzpicture}[anchorbase,scale=.5]
\node at (0,0) {$\circ$};
\node at (1,0){$\up$};
\node at (2,0){$ (\up)$};
\fill (1.5,-.67) circle(3.5pt);
\draw (1,.3) to [out=90,in=-90] +(-1,.9);
\draw (1,-.3) .. controls +(0,-.5) and +(0,-.5) .. +(1,0);
\draw (2,.3) -- +(0,.9);
\node at (0,1.5) {${\Diamond}$};
\node at (1,1.5){$\circ$};
\node at (2,1.5){$ (\up)$};
\end{tikzpicture}\\
\end{array}
\end{eqnarray}

On the other hand $\cS(\pmu)_{j_0} = 1$ which implies that at position $1$ there is either an $\up$ or a $\times$. In addition, since $\mu_j > \mu_{j_0}$ for $j < j_0$ it holds that $\cS(\pmu)_j \leq -1$ for $j < j_0$ (the case $=-1$ giving us the symbol $\times$ at position $1$) and since $\mu_j \leq \mu_{j_0}$ for $j > j_0$ it holds $\cS(\pmu)_j \geq 2$ for $j > j_0$.  Again, neither tail length nor number of dotted cups changes. 

\subsubsection*{The additional box is added on the diagonal} We add the box in position $(j_0,i_0)$ on the diagonal and thus $i_0-j_0 = \frac{\delta}{2}$. In addition $b_{j_0} = 0$ which will be increased by $1$, while all other $a$'s and $b$'s are left unchanged. It holds $a_{i_0}=0$ as well as $a_i=0$ for $i > i_0$ and $b_j=0$ for $j>j_0$.  This means that the diagram with a tail has only the symbol $\times$ at position zero (possibly multiple times), with all but one being distributed when forming the coloured diagram. Since we add a box on the diagonal all $b_j > 1$ for $j<j_0$, and all $a_i \geq 1$ for $i < i_0$. This implies that we have the symbol $\circ$ or $>$ at position $1$, giving us the following configurations displayed on the left (showing positions zero and $1$) with the image under $\op{T}$ displayed on the right hand side.
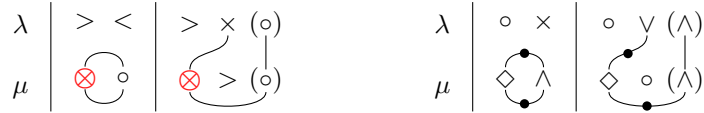
\begin{equation*}
\begin{array}{r|c|c}
\begin{tikzpicture}[anchorbase,scale=.5]
\node at (-1.5,-.15) {$\mu$};
\node at (-1.5,1.6) {$\lambda$};
\end{tikzpicture} & 
\begin{tikzpicture}[anchorbase,scale=.5]
\node at (0,0) {$\color{myred}{\boldsymbol \otimes}$};
\node at (1,0){$\circ$};
\draw (0,.3) .. controls +(0,.5) and +(0,.5) .. +(1,0);
\draw (0,-.3) .. controls +(0,-.5) and +(0,-.5) .. +(1,0);
\node at (0,1.5) {$>$};
\node at (1,1.5){$<$};
\end{tikzpicture} & 
\begin{tikzpicture}[anchorbase,scale=.5]
\node at (0,0) {$\color{myred}{\boldsymbol \otimes}$};
\node at (1,0){$>$};
\node at (2,0){$(\circ)$};
\draw (0,.3) to [out=90,in=-90] +(1,.9);
\draw (0,-.3) .. controls +(0,-.5) and +(0,-.5) .. +(2,0);
\draw (2,.3) -- +(0,.9);
\node at (0,1.5) {$>$};
\node at (1,1.5){$\times$};
\node at (2,1.5){$(\circ)$};
\end{tikzpicture}\\
\end{array}
\qquad \qquad 
\begin{array}{r|c|c}
\begin{tikzpicture}[anchorbase,scale=.5]
\node at (-1.5,-.15) {$\mu$};
\node at (-1.5,1.6) {$\lambda$};
\end{tikzpicture} & 
\begin{tikzpicture}[anchorbase,scale=.5]
\node at (0,-.1) {$\Diamond$};
\node at (1,0){$\up$};
\fill (.5,.67) circle(3.5pt);
\draw (0,.3) .. controls +(0,.5) and +(0,.5) .. +(1,0);
\fill (.5,-.67) circle(3.5pt);
\draw (0,-.3) .. controls +(0,-.5) and +(0,-.5) .. +(1,0);
\node at (0,1.5) {$\circ$};
\node at (1,1.5){$\times$};
\end{tikzpicture} & 
\begin{tikzpicture}[anchorbase,scale=.5]
\node at (0,-.1) {$\Diamond$};
\node at (1,0){$\circ$};
\node at (2,0){$(\up)$};
\fill (.5,.72) circle(3.5pt);
\draw (0,.3) to [out=90,in=-90] +(1,.9);
\draw (0,-.3) .. controls +(0,-.5) and +(0,-.5) .. +(2,0);
\draw (2,.3) -- +(0,.9);
\fill (1,-.67) circle(3.5pt);
\node at (0,1.5) {$\circ$};
\node at (1,1.5){$\down$};
\node at (2,1.5){$(\up)$};
\end{tikzpicture}\\
\end{array}
\end{equation*}
On the other hand, $\mu_{j_0}= j_0 + \frac{\de}{2}-1$, thus $\cS(\pmu)_{j_0} = 0$, which gets decreased by $1$. Furthermore $\cS(\pmu)_j < -1$ for $j<j_0$ since $\mu_j > \mu_{j_0}$ for $j<j_0$ and $\cS(\pmu)_j \geq 1$ for $j>j_0$. This gives the symbol $\Diamond$ at position zero and either $\up$ or $\circ$ at position $1$.

The tail length decreases by $1$, but we also loose the decoration on the first dotted cup from the left or the dotted cup altogether; all fake cups and corresponding frozen vertices remain unchanged. Since the cup diagrams agree, their leftmost label determines whether they are coloured (in the sense of \cite{GS2}) or dotted in our sense, hence the statement follows from \eqref{T2}.
The proposition follows.
\end{proof}

\begin{definition}  \label{def:consistlabels}
Assume $\eta\in X^+(\mg)$.
Let $D$ be the {\rm GS}-cup diagram associated with $\eta$. Then a {\it consistent labelling} of $D$ is a labelling of the vertices with $>$, $<$, $\times$, $\circ$ such that the core symbols match with the core symbols of $\mathrm{GS}(\eta)$ and each cup is labelled by precisely one $\times$ and one $\circ$. 
For a tailless weight $\nu\in X^+(\mg)$ let $A(\eta,\nu)=1$ if $\op{GS}(\nu)$ is a consistent labelling of $D$ and  $A(\eta,\nu)=0$ otherwise. In case of a consistent labelling, let $x(\eta,\nu)$ be the number of coloured cups, and $y(\eta,\nu)$ the number of coloured cups labelled $\circ$ and $\times$ in this order. Set $z(\eta,\nu)=1$ if there exists a cup with left vertex at position zero labelled with $\times$ via $\nu$ in the even case, respectively a cups with left vertex at position $\tfrac{1}{2}$ labelled with $\times$ via $\nu$ in the odd case and indicator $(+)$ and let $z(\eta,\nu)=0$ otherwise.
\end{definition}

\section{Counting formulas}
We present now several dimension formulas for homomorphism spaces.

\subsection{Dimensions of morphism spaces: alternating formula}
We commence with the following dimension formula deduced from the results of Gruson and Serganova, see \cite[Theorems 1 to 4]{GS2}\footnote{It might help to say that the first condition is  only implicitly contained in \cite{GS2}.}.
\begin{prop}
\label{prop:alternating}
Let $\la,\mu\in X^+(\mg)$. 
\begin{enumerate}
\item If $r=2m$ and $a_m>0>a_{m}'$ or $a_m<0<a_{m}'$, as in \eqref{oje}, \eqref{muinab}, then 
\abovedisplayskip0.25em
\belowdisplayskip0.25em
\[
\op{Hom}_{\cF'}(P^\mg(\la),P^\mg(\mu))=\{0\},
\]
\item otherwise
\begin{eqnarray}
\label{aa}
\op{dim}\op{Hom}_{\cF'}(P^\mg(\la),P^\mg(\mu))=\sum a(\la,\nu)a(\mu,\nu)
\end{eqnarray}
where the sum  runs through all tailless dominant weights $\nu$ and 
\begin{eqnarray}
\label{a}
a(\eta,\nu)&=&(-1)^{x(\eta,\nu)+y(\eta,\nu)+z(\eta,\nu)}A(\eta,\nu) 
\end{eqnarray}
  for any $\eta\in X^+(\mg)$.
  \end{enumerate}
\end{prop}
\begin{remark}
\label{alter}
{\rm
Note that $a(\eta,\nu)\in\{-1,0,1\}$. In particular, the numbers are not always non-negative, and the above sum might have some (non-trivial) cancellations. In the framework of Gruson and Serganova, the $a(\eta,\nu)$ are coefficients expressing  the so-called {\it Euler classes} $\cE^\mg(\nu)$ in terms of simple modules, i.e. we have in the Grothendieck group
$[\cE^\mg(\nu)]=\sum_\la a(\la,\nu)[L^\mg(\la)]$. 
}
\end{remark}
\begin{proof}
We have $\op{dim}\op{Hom}_{\cF'}(P^\mg(\la),P^\mg(\mu))=[P^\mg(\mu):L^\mg(\la)]$, where $[-:-]$ denotes the Jordan-H\"older multiplicity of $L^\mg(\la)$ in $P^\mg(\mu)$ or alternatively the coefficient of the class $[L^\mg(\la)]$ in the Grothendieck group when the class $[P^\mg(\la)]$  is  expressed in the classes of the simple modules. On the other hand, the  classes  $[\cE^\mg(\nu)]$  of the Euler characteristics (for tailless dominant $\nu$) are linearly independent in the Grothendieck group and 
\abovedisplayskip0.5em
\belowdisplayskip0.5em
\[
[P^\mg(\la)]=\sum_{\nu} a(\la,\nu) [\cE^\mg(\nu)]\quad \text{and} \quad [\cE^\mg(\nu)]=\sum_{\mu} a(\mu,\nu) [L^\mg(\mu)]
\]
by \cite[Lemma 3, Theorem 1]{GS2}, so~\eqref{aa} holds. Formula~\eqref{a} is just a concise reformulation of \cite[Theorem 2, Theorem 3, Theorem 4]{GS2}.
\end{proof}

\subsection{Dimensions of morphism spaces: positive formula}
We first show that the cancellations addressed in Remark~\ref{alter} appear precisely if the corresponding space of homomorphisms vanishes completely. This allows us to get the following explicit dimension formula.
\begin{prop}
\label{cancel}
Let $\la,\mu\in X^+(\mg)$ with $a_m\geq0$ and $a_m'\geq 0$ in the notation from \eqref{oje} respectively \eqref{muinab}, and let $\pla,\pmu$ be the corresponding hook partitions.
Then there exist $\epsilon \in \op{Stab}_\sigma(\lambda)$ and $\epsilon^\prime \in \op{Stab}_\sigma(\mu)$ such that the following are equivalent 
\begin{enumerate}[(I)]
\item \label{equI} $\op{Hom}_{\cF'}(P^\mg(\la),P^\mg(\mu))\not=0$,
\item \label{equII} the circle diagram $C=\underline{(\pla^\owedge,\epsilon)}\overline{(\pmu^\owedge,\epsilon^\prime)}$ is not nuclear and every component has an  even number of dots.
\end{enumerate}
If \eqref{equI} and \eqref{equII} are satisfied, then the following holds
\begin{enumerate}
\item $a(\la,\nu)a(\mu,\nu)\in\{0,1\}$ for any tailless $\nu\in X^+(\mg)$,
\item $\op{dim}\op{Hom}_{\cF'}(P^\mg(\la),P^\mg(\mu))=2^{c+c^\prime}$, where $c$ is the number of closed components of $c$. In case $a_m=0=a_m^\prime$, or equivalently if the zero vertex in $C$ is contained in a cup and a cap, then $c^\prime=1$, otherwise $c^\prime=0$.
\end{enumerate}
\end{prop}

\begin{proof}
For $\eqref{equI}\Rightarrow \eqref{equII}$ it is enough to show that $\op{Hom}_{\cF'}(P^\mg(\la),P^\mg(\mu))\not=0$ implies we have no non-propagating line and that each closed component has an even number of dots, since then only the leftmost line is allowed to carry dots and we can choose $\epsilon$ and $\epsilon^\prime$ such that by definition of $(\pla^\owedge, \epsilon)$ and $(\pmu^\owedge,\epsilon^\prime)$ the total number of dots is even.
By \eqref{aa} we have $\op{dim}\op{Hom}_{\cF'}(P^\mg(\la),P^\mg(\mu))
=\sum_\nu a(\la,\nu)a(\mu,\nu)$ and $a(\la,\nu)$ is non-zero if putting the GS-weight $\op{GS}(\nu)$ on top of the cup diagram $D$ associated with $\op{GS}(\la)$ results in a picture where the labels $>$ and $<$ in $\la$ and $\nu$ agree and each cup has the two symbols $\circ$, $\times$ in any order at its two endpoints. Clearly $a(\la,\nu)=0$ if there is a non-propagating line, since the line must have $\circ$'s at the end, but has an odd total number of cups and caps.

If $a(\la,\nu) \neq 0$, then $a(\la,\nu)=(-1)^{x+z+y}$, where $x=x(\la,\nu)$ , $y=y(\la,\nu)$, and $z=z(\la,\nu)$. In particular, $x$ does not depend on $\nu$. Let $K$ be a closed component of the circle diagram $C$. If $a(\la,\nu)a(\mu,\nu)\not=0$ then we can find a weight $\nu'$ such that $\op{GS}(\nu')$  agrees with $\op{GS}(\nu)$ at all vertices not contained in $K$, but the symbols $\times$ and $\circ$ swapped for the vertices contained in $K$. Assume now that $K$ has an odd total number of dots. If $K$ does not contain the vertex $\frac{1}{2}$ or zero then $(-1)^{y(\la,\nu)+y(\mu,\nu)}=-(-1)^{y(\la,\nu')+y(\mu,\nu')}$, hence $a(\la,\nu)a(\mu,\nu)=-a(\la,\nu')a(\mu,\nu')$ and so the two contributions cancel.

The same holds if $K$ does contain the the vertex $\tfrac{1}{2}$ but with the same indicator $(+)$ or $(-)$ in $\la$ and $\mu$. If it contains the vertex $\tfrac{1}{2}$ and the indicators differ then we have an even number of coloured cups and caps in $K$, hence
\begin{eqnarray*}
(-1)^{y(\la,\nu)+y(\mu,\nu)}&=&\phantom{-}(-1)^{y(\la,\nu')+y(\mu,\nu')} \quad \text{ and }\\
(-1)^{z(\la,\nu)+z(\mu,\nu)}&=&-(-1)^{z(\la,\nu')+z(\mu,\nu')}
\end{eqnarray*}
implying again $a(\la,\nu)a(\mu,\nu)=-a(\la,\nu')a(\mu,\nu')$. Hence each closed component requires an even number of dots and so \eqref{equI} implies \eqref{equII}.

Note that one must take more care in case the zero vertex is contained in a component of $C$ with both $a_m=0$ and $a_m^\prime=0$. In this case the vertices of this component are fixed in the $\op{GS}$-weight $\op{GS}(\nu)$ but in the combinatorics of $\op{GS}$ the possible weights all come with a sign $[+]$ or $[-]$ giving the needed factor $2$ for this component.

For the converse note that \eqref{equII} implies that the circle diagram $C$ is orientable, each line in a unique way and each closed component in exactly two ways. The same holds if we remove the dots. After applying $\op{T}$, any such orientation gives an allowed labelling $\nu$ in the sense of Definition~\ref{def:consistlabels}. We claim that the corresponding value $\alpha=a(\la,\nu)a(\mu,\nu)$ is equal to $1$. By definition 
\[
\alpha=(-1)^{x(\la)+x(\mu)+z(\la,\nu)+z(\mu,\nu)}
(-1)^{y(\la,\nu)+y(\mu,\nu)}.
\]
If $C$ is a small circle, i.e. only contains a single cup and cap, then it has either no dots, hence no coloured cups and there is nothing to check. Or two dots and two coloured cups (and the same indicator) and the statement is clear as well. Otherwise, if $C$ contains a kink without coloured cups and caps then we can remove the kink to obtain a new $\la$ and $\mu$ with the same value $\alpha$ attached. So we assume there is no such kink, but then it contains a configuration of the form (dashed lines indicate the colouring)
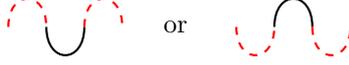
\begin{eqnarray*}
\begin{tikzpicture}[thick,scale=1]
\draw[color=red,dashed] (0,0) .. controls +(0,.5) and +(0,.5) .. +(.5,0);
\draw (.5,0) .. controls +(0,-.5) and +(0,-.5) .. +(.5,0);
\draw[color=red,dashed] (1,0) .. controls +(0,.5) and +(0,.5) .. +(.5,0);
\node at (2.2,0) {or};
\draw[color=red,dashed] (3,0) .. controls +(0,-.5) and +(0,-.5) .. +(.5,0);
\draw (3.5,0) .. controls +(0,.5) and +(0,.5) .. +(.5,0);
\draw[color=red,dashed] (4,0) .. controls +(0,-.5) and +(0,-.5) .. +(.5,0);
\end{tikzpicture}
\end{eqnarray*}
Removing the colouring and also the newly created uncoloured kink changes $\la$ and $\mu$, but not the corresponding value $\alpha$. Hence it must be equal to $1$.

Note that if $C$ is a component containing zero in case $a_m=0$ and $a_m^\prime=0$, then in the $\op{GS}$ combinatorics there is only one allowed orientation, but it comes equipped with a sign $[\pm]$. Thus any possible weight that orients all other components and uses the unique orientation on the component containing zero is counted twice. If the component is a circle that gives again the total number of orientation of the circle diagram $C$ and each contributing $1$ to $\op{dim}\op{Hom}_{\cF'}(P^\mg(\la),P^\mg(\mu))
=\sum_\nu a(\la,\nu)a(\mu,\nu)$. If the component is a line and $a_m=0$ and $a_m^\prime=0$ they are all doubled up thus giving the additional factor of $2$ times the number of orientations. Therefore \eqref{equII} implies \eqref{equI} and the remaining statements follow.
\end{proof}

\subsection{The Dimension Formula}
We finally use Proposition~\ref{cancel} and Lemma~\ref{lemplusminus} to deduce the dimension formula, also establishing Theorem~\ref{mainprop}.

\begin{theorem}[Dimension formula]
\label{maintheorem}
Consider $G=\OSPrn$ and $\la,\mu\in X^+(G)$. Let $c(\la,\mu)$ be the number of possible orientations of $\underline{\la}\overline{\mu}$. Then
\[
\mathrm{dim} \Hom_\cF(P(\la),P(\mu)) = \left\lbrace \begin{array}{ll} c(\la,\mu) & \text{ if } \underline{\la}\overline{\mu} \text{ is not nuclear},\\
0 & \text{ if } \underline{\la}\overline{\mu} \text{ is nuclear}.
\end{array}\right.
\] 
\end{theorem}

\begin{proof}
Note that if $\la$ and $\mu$ do not have the same core diagram then $\underline{\la}\overline{\mu}$ is not orientable and on the other hand by Propositions~\ref{blocksindiagramsodd} and~\ref{blocksindiagramseven} it follows that $P(\la)$ and $P(\mu)$ are not in the same block hence the morphism space is zero, thus the claim is automatic in this case. From now on assume that the core diagrams of both $\la$ and $\mu$ match.

For $\la\in X^+(G)$ let $\la^\mg\in X^+(\mg)$ such that $a_m\geq0$ in the notation from \eqref{oje} (with $\la$ replaced by $\la^\mg$) with the same underlying hook partitions as $\la$.   Now consider $H=\Hom_\cF(P(\la),P(\mu))$ as in the theorem. We will freely use Proposition~\ref{Cartanmatrix} to swap the roles of $\la$ and $\mu$. 

{\it Assume first that $G=\OSPrn$ with $r$ odd}. If $\la=(\la^\mg,+)$ and $\mu=(\mu^\mg,-)$ (or the reversed signs), then $H=\{0\}$ by Remark~\ref{rem:plusminusodd} and $\underline{\la}\overline{\mu}$ is not orientable by Proposition~\ref{Homzero}. If the signs agree then again by Remark~\ref{rem:plusminusodd} and Proposition~\ref{cancel} the claim follows.

{\it Assume now that $G=\OSPrn$ with $r$ even}. Let first $\la=(\la^\mg,\pm)$ and $\mu=(\mu^\mg)^G$. Then by Proposition~\ref{blocksO} we have
$\op{dim}H=\op{dim} \Hom_{\cF'}(P^\mg(\la^\mg),P^\mg (\mu^\mg))$
which is given by Proposition~\ref{cancel}. The claim follows by comparing Proposition~\ref{cancel} with Proposition~\ref{nosignsign}.

Let now $\la=(\la^\mg)^G$ and $\mu=(\mu^\mg)^G$. Again by Proposition~\ref{blocksO} we have $\op{dim}H=\op{dim} \Hom_{\cF'}(P^\mg(\la^\mg),P^\mg (\mu^\mg))$. Moreover by Proposition~\ref{Voila} $\underline{\la}$ and $\underline{\mu}$ have a (dotted) ray at zero.  In particular, the circle diagram $\underline{\la}\overline{\mu}$ has, apart from a straight line $L$ passing through zero and built from two dotted rays, only closed components or rays containing no dots at all (since they are to the right of the propagating line). Hence by Remark~\ref{ornuclear} the diagram is orientable if and only if every component has an even number of dots. The number of orientations is then obviously equal to $2^c$, where $c$ is the number of closed components. Thus the claim follows with Proposition~\ref{cancel}. 

Finally we have the case of signs for both weights which is the most involved case. Let $\la=(\la^\mg,\epsilon)$ and $\mu=(\mu^\mg,\epsilon^\prime)$. If $\la$ nor $\mu$ contain zero in their core diagram then it follows from Proposition~\ref{Homzero} that $\underline{\la}\overline{\mu}$ is only orientable if $\epsilon=\epsilon'$. On the other hand Lemma~\ref{lem:orientableandlines} implies that we are in case (1) of Lemma~\ref{lemplusminus} and the morphism space for different signs is equal to zero. Now Propositions~\ref{blocksO} and~\ref{cancel} imply that the dimension of the morphism space for equal signs is given by the number of orientations.

If both, $\la$ and $\mu$, contain zero in their core diagram then position zero is contained in a line or in a circle in the diagram $\underline{\la}\overline{\mu}$. Assume it is contained in a circle then Lemma~\ref{lem:orientableandlines} implies that we are in case (1) of Lemma~\ref{lemplusminus}, but $\epsilon$ and $\epsilon^\prime$ do not need to be equal and the dimension of the morphism space is zero in case the signs are chosen such that the circle containing the zero vertex is not orientable and thus, by Proposition~\ref{blocksO}, $\op{dim}H=\op{dim} \Hom_{\cF'}(P^\mg(\la^\mg),P^\mg (\mu^\mg))$ if they are chosen such that the circle is orientable. Again the claim follows from Proposition~\ref{cancel}.
To treat the case with position zero contained in a line in $\underline{\la}\overline{\mu}$, note that Lemma~\ref{lem:orientableandlines} implies that we are in case (2) of Lemma~\ref{lemplusminus} and for all choices of $\epsilon$ and $\epsilon^\prime$ it holds
$\op{dim}H=\tfrac{1}{2}\op{dim}\Hom_{\cF'}(P^\mg(\la^\mg),P^\mg (\mu^\mg))$
which is again equal to the number of orientations of the diagram $\underline{\la}\overline{\mu}$ by Proposition~\ref{cancel}.
\end{proof}

\begin{lemma} \label{lem:orientableandlines}
Let $(\la,\epsilon),(\mu,\epsilon^\prime) \in X^+(G)$. 
\begin{enumerate}
\item 
If the circle diagram $C=\underline{(\pla^{\owedge},\epsilon)}\overline{(\pmu^{\owedge},\epsilon^\prime)}$ has no line passing containing  the zero vertex and is orientable then $\op{dim}\op{Hom}_\cF(P(\la,\epsilon),P(\mu,\epsilon')) \neq 0$, whereas the morphism space vanishes if one of the two signs is changed.
\item 
If $C$ is orientable and contains a line passing through the zero vertex, then $\op{dim}\op{Hom}_\cF(P(\la,\pm\epsilon),P(\mu,\pm\epsilon')) \neq 0$ for all possible sign choices.
\end{enumerate}
\end{lemma}
\begin{proof}
This follows from the classification theorem in \cite{CH} of indecomposable summands in $V^{\otimes d}$ and will be proved in Part II.
\end{proof}

\section{Examples} \label{sec:examples}

\subsection{The classical case: \texorpdfstring{$\OSP(r|0)$}{OSp(r,0)}} \label{sec:classical}

We start with the case $\OSP(3|0)$. The irreducible modules in $\cF$ are labelled by $(0,1)$-hook partitions, that means partitions which fit into one column, all with an attached sign, see Proposition~\ref{def:labelssimples1} and Lemma~\ref{AB}. The tail length is always zero, see Definition~\ref{def:tail}. The table in Figure~\eqref{klassisch} shows in the first and third column the partitions together with their signs and next to it (on the right) the corresponding weight diagrams from Definition~\ref{superweight}. Since all non-core symbols are frozen, the associated cup diagram consists only of rays, namely as follows
\abovedisplayskip0.5em
\belowdisplayskip0.5em
\begin{equation}
\label{klassisch}
 \arraycolsep=6pt\def\arraystretch{1.8}
\begin{array}[t]{c|c||c|c|c}
(\varnothing,+) & \circ\up\down\down\down\down\down\ldots & (\varnothing,-) & \circ\down\down\down\down\down\down\ldots &
\begin{minipage}{.7cm} \scalebox{.4}{\yng(3)} \end{minipage}\\
\hline
\left(\begin{minipage}[][][c]{.3cm} \scalebox{.4}{\yng(1)} \end{minipage},+\right) & \down\circ\down\down\down\down\down\ldots & \left(\begin{minipage}[][][c]{.3cm} \scalebox{.4}{\yng(1)} \end{minipage},-\right) & \up\circ\down\down\down\down\down\ldots &
\begin{minipage}{.7cm} \scalebox{.4}{\yng(2)} \end{minipage}\\
\hline
\left(\begin{minipage}[][][c]{.3cm} \scalebox{.4}{\yng(1,1)} \end{minipage},+\right) & \up\down\circ\down\down\down\down\ldots & \left(\begin{minipage}[][][c]{.3cm} \scalebox{.4}{\yng(1,1)} \end{minipage},-\right) & \down\down\circ\down\down\down\down\ldots &
\begin{minipage}{.7cm} \scalebox{.4}{\yng(2,1)} \end{minipage}\\
\hline
\left(\begin{minipage}[][][c]{.3cm} \scalebox{.4}{\yng(1,1,1)} \end{minipage},+\right) & \down\down\down\circ\down\down\down\ldots & \left(\begin{minipage}[][][c]{.3cm} \scalebox{.4}{\yng(1,1,1)} \end{minipage},-\right) & \up\down\down\circ\down\down\down\ldots &
\begin{minipage}{.7cm} \scalebox{.4}{\yng(2,1,1)} \end{minipage}\\
\hline
\begin{minipage}[c][1cm][c]{1.2cm} $\left(1^a,+\right)\,$ \\ $a$ even
\end{minipage}
& \up\down\ldots\down\!\!\overset{a+1}{\circ}\!\!\down\ldots & 
\begin{minipage}[c][][c]{1.2cm} $\left(1^a,-\right)\,$ \\ $a$ even
\end{minipage}
 & \down\down\ldots\down\!\!\overset{a+1}{\circ}\!\!\down\ldots &
(2,1^{a-1})\\
\hline
\begin{minipage}[c][1cm][c]{1.2cm} $\left(1^a,+\right)$ \\ $a$ odd 
\end{minipage}
& \down\down\ldots\down\!\!\overset{a+1}{\circ}\!\!\down\ldots & 
\begin{minipage}[c][][c]{1.2cm} $\left(1^a,-\right)\,$ \\ $a$ odd
\end{minipage}
& \up\down\ldots\down\!\!\overset{a+1}{\circ}\!\!\down\ldots &
(2,1^{a-1})\\
\end{array}
\end{equation}

\begin{remark}[Generalizing Weyl's notion of associated partitions]\hfill
\begin{enumerate}
\item Additionally, the last column in \eqref{klassisch} shows the unique partitions $\pga$ such that $\pga^\infty$, obtained via  $\cS(\pga)$, is the weight diagram in the same row where the first $\down$ is changed to an $\up$. Note that the partitions in the first and last column together form a pair of {\it associated partitions} in the sense of Weyl, i.e. their first rows are of length less or equal to $r=3$ and together sum up to $r=3$, and the partitions coincide otherwise. Associated partitions correspond to irreducible $\OSP(3|0)$ representations that are isomorphic when restricted to $\SOSP(3|0)$. For more details on associated partitions see e.g. \cite[$\S$ 19.5]{FH}. Hence our diagrammatics could be seen as extending Weyl's notion of associated partitions.
\item  The same is true more generally for $\OSP(2m+1|0)$. The two weight diagrams attached to the two different signs for a given partition differ precisely at the first symbol. Changing this first symbol from an $\up$ to a $\down$ produces the associated partition. A pair of associated partitions corresponds to representations that differ by taking the tensor product with the sign representation, see \cite[Exercise 19.23]{FH}, which agrees with Proposition~\ref{def:labelssimples1}.
\item In the case $\OSP(2m|0)$ the irreducible modules are labelled by partitions that have at most $m$ columns and they have a sign iff the partition has strictly less than $m$ columns, in which case the two partitions are associated in Weyl's sense. In case of a partition with $m$ columns the partition is associated to itself, which corresponds to the fact that the weight diagram starts with the symbol $\Diamond$ at position zero and therefore it does not obtain a sign in our convention.
\end{enumerate}
\end{remark}
\subsection{The smallest non-semisimple case: \texorpdfstring{$\OSP(3|2)$}{OSp(3,2)}} \label{sec:osp32}
Let us come back to the example from Section \ref{32}, the category $\cF(\OSP(3|2))$. The various diagrammatic weights are listed in the table below.  We first list the $(1,1)$-hook partitions, then the sequence $\cS(\pla)$ and the corresponding diagrammatic weights and cup diagrams.
\abovedisplayskip0.5em
\belowdisplayskip0.5em
\begin{equation*} 
\scalebox{.85}{\arraycolsep=6pt\def\arraystretch{2.2}
\begin{array}[t]{c|c|c|c|c}
\pla & \cS(\pla) & \pla^\infty & (\pla^\owedge,+) & (\pla^\owedge,-) \\ \hline
\varnothing&
(\frac{1}{2},\frac{3}{2},\frac{5}{2},\frac{7}{2},\frac{9}{2},\scriptstyle{\ldots})&
\up\up\owedge\owedge\owedge & 
\begin{tikzpicture}[thick,anchorbase,scale=0.8]
\node at (0,.1) {$\scriptstyle \mathbf{\up}$};
\node at (.5,.1) {$\scriptstyle \mathbf{\up}$};
\node at (1,.1) {$\scriptstyle \mathbf{\up}$};
\node at (1.5,.1) {$\scriptstyle \mathbf{\down}$};
\node at (2,.1) {$\scriptstyle \mathbf{\down}$};
\node at (2.5,.1) {$\mathbf{\scriptstyle{\cdots}}$};
\draw (0,0) .. controls +(0,-.5) and +(0,-.5) .. +(.5,0);
\fill (.25,-.365) circle(2pt);
\draw (1,0) -- +(0,-.5);
\fill (1,-.3) circle(2pt);
\draw (1.5,0) -- +(0,-.5);
\draw (2,0) -- +(0,-.5);
\end{tikzpicture} & 
\begin{tikzpicture}[thick,anchorbase,scale=0.8]
\node at (0,.1) {$\scriptstyle \mathbf{\up}$};
\node at (.5,.1) {$\scriptstyle \mathbf{\up}$};
\node at (1,.1) {$\scriptstyle \mathbf{\down}$};
\node at (1.5,.1) {$\scriptstyle \mathbf{\down}$};
\node at (2,.1) {$\mathbf{\scriptstyle \down}$};
\node at (2.5,.1) {$\mathbf{\scriptstyle{\cdots}}$};
\draw (0,0) .. controls +(0,-.5) and +(0,-.5) .. +(.5,0);
\fill (.25,-.365) circle(2pt);
\draw (1,0) -- +(0,-.5);
\draw (1.5,0) -- +(0,-.5);
\draw (2,0) -- +(0,-.5);
\end{tikzpicture} \\ 
\hline 
\scalebox{.4}{\yng(1)}&
({\scriptstyle{-}}\frac{1}{2},\frac{3}{2},\frac{5}{2},\frac{7}{2},\frac{9}{2},\scriptstyle{\ldots})&
\down\up\owedge\owedge\owedge & 
\begin{tikzpicture}[thick,anchorbase,scale=0.8]
\node at (0,.1) {$\scriptstyle \mathbf{\down}$};
\node at (.5,.1) {$\scriptstyle \mathbf{\up}$};
\node at (1,.1) {$\scriptstyle \mathbf{\down}$};
\node at (1.5,.1) {$\scriptstyle \mathbf{\down}$};
\node at (2,.1) {$\scriptstyle \mathbf{\down}$};
\node at (2.5,.1) {$\mathbf{\scriptstyle{\cdots}}$};
\draw (0,0) .. controls +(0,-.5) and +(0,-.5) .. +(.5,0);
\draw (1,0) -- +(0,-.5);
\draw (1.5,0) -- +(0,-.5);
\draw (2,0) -- +(0,-.5);
\end{tikzpicture}
& 
\begin{tikzpicture}[thick,anchorbase,scale=0.8]
\node at (0,.1) {$\scriptstyle \mathbf{\down}$};
\node at (.5,.1) {$\scriptstyle \mathbf{\up}$};
\node at (1,.1) {$\scriptstyle \mathbf{\up}$};
\node at (1.5,.1) {$\scriptstyle \mathbf{\down}$};
\node at (2,.1) {$\scriptstyle \mathbf{\down}$};
\node at (2.5,.1) {$\mathbf{\scriptstyle{\cdots}}$};
\draw (0,0) .. controls +(0,-.5) and +(0,-.5) .. +(.5,0);
\draw (1,0) -- +(0,-.5);
\fill (1,-.3) circle(2pt);
\draw (1.5,0) -- +(0,-.5);
\draw (2,0) -- +(0,-.5);
\end{tikzpicture}
\\ 
\hline 
\scalebox{.4}{\yng(2)}&
({\scriptstyle{-}}\frac{3}{2},\frac{3}{2},\frac{5}{2},\frac{7}{2},\frac{9}{2},\scriptstyle{\ldots})&
\circ\times\owedge\owedge\owedge & 
\begin{tikzpicture}[thick,anchorbase,scale=0.8]
\node at (0,.1) {$\scriptstyle \mathbf{\circ}$};
\node at (.5,.1) {$\scriptstyle \mathbf{\times}$};
\node at (1,.1) {$\scriptstyle \mathbf{\up}$};
\node at (1.5,.1) {$\scriptstyle \mathbf{\down}$};
\node at (2,.1) {$\scriptstyle \mathbf{\down}$};
\node at (2.5,.1) {$\mathbf{\scriptstyle{\cdots}}$};
\draw (1,0) -- +(0,-.5);
\fill (1,-.3) circle(2pt);
\draw (1.5,0) -- +(0,-.5);
\draw (2,0) -- +(0,-.5);
\end{tikzpicture}
& 
\begin{tikzpicture}[thick,anchorbase,scale=0.8]
\node at (0,.1) {$\scriptstyle \mathbf{\circ}$};
\node at (.5,.1) {$\scriptstyle \mathbf{\times}$};
\node at (1,.1) {$\scriptstyle \mathbf{\down}$};
\node at (1.5,.1) {$\scriptstyle \mathbf{\down}$};
\node at (2,.1) {$\scriptstyle \mathbf{\down}$};
\node at (2.5,.1) {$\mathbf{\scriptstyle{\cdots}}$};
\draw (1,0) -- +(0,-.5);
\draw (1.5,0) -- +(0,-.5);
\draw (2,0) -- +(0,-.5);
\end{tikzpicture}
\\ 
\hline 
\scalebox{.4}{\yng(1,1)}&
({\scriptstyle{-}}\frac{1}{2},\frac{1}{2},\frac{5}{2},\frac{7}{2},\frac{9}{2},\scriptstyle{\ldots})&
\times\circ\owedge\owedge\owedge & 
\begin{tikzpicture}[thick,anchorbase,scale=0.8]
\node at (0,.1) {$\scriptstyle \mathbf{\times}$};
\node at (.5,.1) {$\scriptstyle \mathbf{\circ}$};
\node at (1,.1) {$\scriptstyle \mathbf{\up}$};
\node at (1.5,.1) {$\scriptstyle \mathbf{\down}$};
\node at (2,.1) {$\scriptstyle \mathbf{\down}$};
\node at (2.5,.1) {$\mathbf{\scriptstyle{\cdots}}$};
\draw (1,0) -- +(0,-.5);
\fill (1,-.3) circle(2pt);
\draw (1.5,0) -- +(0,-.5);
\draw (2,0) -- +(0,-.5);
\end{tikzpicture}
& 
\begin{tikzpicture}[thick,anchorbase,scale=0.8]
\node at (0,.1) {$\scriptstyle \mathbf{\times}$};
\node at (.5,.1) {$\scriptstyle \mathbf{\circ}$};
\node at (1,.1) {$\scriptstyle \mathbf{\down}$};
\node at (1.5,.1) {$\scriptstyle \mathbf{\down}$};
\node at (2,.1) {$\scriptstyle \mathbf{\down}$};
\node at (2.5,.1) {$\mathbf{\scriptstyle{\cdots}}$};
\draw (1,0) -- +(0,-.5);
\draw (1.5,0) -- +(0,-.5);
\draw (2,0) -- +(0,-.5);
\end{tikzpicture}
\\ 
\hline 
\scalebox{.4}{\yng(2,1)}&({\scriptstyle{-}}\frac{3}{2},\frac{1}{2},\frac{5}{2},\frac{7}{2},\frac{9}{2},\scriptstyle{\ldots})&
\owedge\down\up\owedge\owedge & 
\begin{tikzpicture}[thick,anchorbase,scale=0.8]
\node at (0,.1) {$\scriptstyle \mathbf{\down}$};
\node at (.5,.1) {$\scriptstyle \mathbf{\down}$};
\node at (1,.1) {$\scriptstyle \mathbf{\up}$};
\node at (1.5,.1) {$\scriptstyle \mathbf{\down}$};
\node at (2,.1) {$\scriptstyle \mathbf{\down}$};
\node at (2.5,.1) {$\mathbf{\scriptstyle{\cdots}}$};
\draw (0,0) -- +(0,-.5);
\draw (.5,0) .. controls +(0,-.5) and +(0,-.5) .. +(.5,0);
\draw (1.5,0) -- +(0,-.5);
\draw (2,0) -- +(0,-.5);
\end{tikzpicture}
& 
\begin{tikzpicture}[thick,anchorbase,scale=0.8]
\node at (0,.1) {$\scriptstyle \mathbf{\up}$};
\node at (.5,.1) {$\scriptstyle \mathbf{\down}$};
\node at (1,.1) {$\scriptstyle \mathbf{\up}$};
\node at (1.5,.1) {$\scriptstyle \mathbf{\down}$};
\node at (2,.1) {$\scriptstyle \mathbf{\down}$};
\node at (2.5,.1) {$\mathbf{\scriptstyle{\cdots}}$};
\draw (0,0) -- +(0,-.5);
\fill (0,-.3) circle(2pt);
\draw (.5,0) .. controls +(0,-.5) and +(0,-.5) .. +(.5,0);
\draw (1.5,0) -- +(0,-.5);
\draw (2,0) -- +(0,-.5);
\end{tikzpicture}
\\ 
\hline
\scalebox{.4}{\yng(3,1,1)}&
({\scriptstyle{-}}\frac{5}{2},\frac{1}{2},\frac{3}{2},\frac{7}{2},\frac{9}{2},\scriptstyle{\ldots})&
\owedge\owedge\down\up\owedge & 
\begin{tikzpicture}[thick,anchorbase,scale=0.8]
\node at (0,.1) {$\scriptstyle \mathbf{\down}$};
\node at (.5,.1) {$\scriptstyle \mathbf{\down}$};
\node at (1,.1) {$\scriptstyle \mathbf{\down}$};
\node at (1.5,.1) {$\scriptstyle \mathbf{\up}$};
\node at (2,.1) {$\scriptstyle \mathbf{\down}$};
\node at (2.5,.1) {$\mathbf{\scriptstyle{\cdots}}$};
\draw (0,0) -- +(0,-.5);
\draw (.5,0) -- +(0,-.5);
\draw (1,0) .. controls +(0,-.5) and +(0,-.5) .. +(.5,0);
\draw (2,0) -- +(0,-.5);
\end{tikzpicture}
& 
\begin{tikzpicture}[thick,anchorbase,scale=0.8]
\node at (0,.1) {$\scriptstyle \mathbf{\up}$};
\node at (.5,.1) {$\scriptstyle \mathbf{\down}$};
\node at (1,.1) {$\scriptstyle \mathbf{\down}$};
\node at (1.5,.1) {$\scriptstyle \mathbf{\up}$};
\node at (2,.1) {$\scriptstyle \mathbf{\down}$};
\node at (2.5,.1) {$\mathbf{\scriptstyle{\cdots}}$};
\draw (0,0) -- +(0,-.5);
\fill (0,-.3) circle(2pt);
\draw (.5,0) -- +(0,-.5);
\draw (1,0) .. controls +(0,-.5) and +(0,-.5) .. +(.5,0);
\draw (2,0) -- +(0,-.5);
\end{tikzpicture}
\\ 
\hline
\scalebox{.4}{\yng(4,1,1,1)}&
({\scriptstyle{-}}\frac{7}{2},\frac{1}{2},\frac{3}{2},\frac{5}{2},\frac{9}{2},\scriptstyle{\ldots})&
\owedge\owedge\owedge\down\up & 
\begin{tikzpicture}[thick,anchorbase,scale=0.8]
\node at (0,.1) {$\scriptstyle \mathbf{\down}$};
\node at (.5,.1) {$\scriptstyle \mathbf{\down}$};
\node at (1,.1) {$\scriptstyle \mathbf{\down}$};
\node at (1.5,.1) {$\scriptstyle \mathbf{\down}$};
\node at (2,.1) {$\scriptstyle \mathbf{\up}$};
\node at (2.5,.1) {$\mathbf{\scriptstyle{\cdots}}$};
\draw (0,0) -- +(0,-.5);
\draw (.5,0) -- +(0,-.5);
\draw (1,0) -- +(0,-.5);
\draw (1.5,0) .. controls +(0,-.5) and +(0,-.5) .. +(.5,0);
\end{tikzpicture}
& 
\begin{tikzpicture}[thick,anchorbase,scale=0.8]
\node at (0,.1) {$\scriptstyle \mathbf{\up}$};
\node at (.5,.1) {$\scriptstyle \mathbf{\down}$};
\node at (1,.1) {$\scriptstyle \mathbf{\down}$};
\node at (1.5,.1) {$\scriptstyle \mathbf{\down}$};
\node at (2,.1) {$\scriptstyle \mathbf{\up}$};
\node at (2.5,.1) {$\mathbf{\scriptstyle{\cdots}}$};
\draw (0,0) -- +(0,-.5);
\fill (0,-.3) circle(2pt);
\draw (.5,0) -- +(0,-.5);
\draw (1,0) -- +(0,-.5);
\draw (1.5,0) .. controls +(0,-.5) and +(0,-.5) .. +(.5,0);
\end{tikzpicture}
\end{array}
}
\end{equation*}
\smallskip

From the weight diagram one can read of (using Proposition~\ref{blocksindiagramsodd}) the blocks and the cup diagrams (including those from Section \ref{32}) for the indecomposable projective modules in a given block. Using now Theorem~\ref{mainprop}, Theorem~\ref{main} and the multiplication rule for circle diagrams from \cite{ESperv}, one deduces the shape and relations for the quiver from Theorem~\ref{exs}.  The block containing $L(\varnothing,-)$ is equivalent to the block $\mathcal{B}$ containing $L(\varnothing,+)$. All other blocks are obviously semisimple (and of atypicality $0$).

\begin{remark}
\label{obsVera}
Although the category $\cF=\cF(\OSP(3|2))$ decomposes as $\cF^+\oplus \cF^-$ with the summands equivalent to $\cF(\SOSP(3|2))$, we still prefer to work with the whole $\cF$ due to its connection to Deligne's category, see \cite{Deligne2}, \cite{CH} and to the Brauer algebras, in particular because \eqref{surj} is not surjective for $\SOSP(3|2)$. To see this observe that 
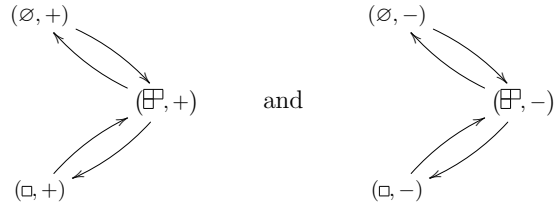
\abovedisplayskip0.25em
\belowdisplayskip0.5em
\begin{eqnarray*}
\scalebox{.8}{
\begin{xy}
  \xymatrix{
  (\varnothing,+)\ar@/^/[dr]&& & \left(\varnothing,-\right)\ar@/^/[dr]\\
  &\left(\scalebox{.4}{\yng(2,1)},+\right) \ar@/^/[ul]\ar@/^/[dl]&\text{\large and}&&\left(\scalebox{.4}{\yng(2,1)},-\right)\ar@/^/[ul]\ar@/^/[dl]\\
 \left(\scalebox{.4}{\yng(1)},+\right)\ar@/^/[ur]&&&\left(\scalebox{.4}{\yng(1)},-\right)\ar@/^/[ur]
}
\end{xy}}
\end{eqnarray*}
show pieces of the quiver corresponding to the two summands $\cF^\pm$. On the vertices one can see the labellings of the indecomposable projective modules $P(\la)$ and the corresponding associated partition.  The number of boxes in the partitions or (if it exists) the associated partition equals the tensor power $d$ such that $P(\la)$ appears as a summand in $V^{\otimes d}$. Observe that these numbers are always even for the quiver on the left and odd for the quiver on the right, in agreement with Remark~\ref{theds}. If one now restricts to $G'$, then $\op{res}P(\la,+)\cong\op{res}P(\la,-)$ and of course all non-trivial homomorphisms stay non-trivial.  Therefore there are non-trivial morphism from $V^{\otimes d}$ to $V^{\otimes d'}$ for some $d$, $d'$ such that $d\not\equiv d'\op{mod} 2$. These morphisms cannot be controlled by the Brauer or Deligne category.
\end{remark}

\subsection{The smallest even case: \texorpdfstring{$\OSP(2|2)$}{OSp(2,2)} and \texorpdfstring{$\OSP$}{OSp} vs \texorpdfstring{$\SOSP$}{SOSp}} \label{sec:osp22}
We chose now one of the most basic non-classical cases, to showcase the differences between the $\OSP$ and the $\SOSP$ situation.

In case of $\SOSP(2|2)$ the block containing the trivial representation $L^\mathfrak{g}(0)$ contains all irreducible representations of the form $L^\mathfrak{g}(\pm a\varepsilon_1 + a \delta_1)$. Abbreviating the $L^\mathfrak{g}(\pm a\varepsilon_1 + a \delta_1)$ by $(\pm a|a)$ we obtain for it the quiver
\small
\abovedisplayskip0.25em
\belowdisplayskip0.5em
\begin{eqnarray*}
\begin{xy}
  \xymatrix{
  \cdots \ar@/^/[r]^{f_{-3}}
  &(-2|2)\ar@/^/[l]^{g_{-3}}\ar@/^/[r]^{f_{-2}}
  &(-1|1)\ar@/^/[l]^{g_{-2}}\ar@/^/[r]^{f_{-1}}
  &(0|0)\ar@/^/[l]^{g_{-1}}\ar@/^/[r]^{f_0}
  &(1|1)\ar@/^/[l]^{g_0}\ar@/^/[r]^{f_1}
  &(2|2)\ar@/^/[l]^{g_1}\ar@/^/[r]^{f_2}
  &\cdots\ar@/^/[l]^{g_2}
  }
\end{xy}
\end{eqnarray*}
\normalsize
subject to the relations $f_{i+1} \circ f_i = 0 = g_i \circ g_{i-1}$ and $g_i \circ f_i = g_{i-1} \circ f_{i-1}$.

The shape of the quiver and the relations follow from Proposition~\ref{cancel}. Alternatively one can also use translation functors studied in \cite{GS2}.  

Switching to $\OSP(2|2)$ corresponds here to taking, in a suitable sense, the smash product of the original path algebra with the group $\mZ/2\mZ$ generated by the involution $\sigma$ and consider the corresponding category of modules, see e.g. \cite[Example 2.1]{RR} for an analogous situation. More precisely we obtain the following: 
the representation $L^\mathfrak{g}(0)$ is doubled up to $L(0,+)$ and $L(0,-)$ while $L^\mathfrak{g}(a\varepsilon_1 + a \delta_1)$ and $L^\mathfrak{g}(- a\varepsilon_1 + a \delta_1)$ give the same representation $L((a|a)^G)$, see Definition~\ref{def:labelssimples2}. The results is that the following quiver describes the principal block of $\cF$ (where we used the elements from $X^+(G)$ as labels for the vertices),
\small
\abovedisplayskip0.25em
\belowdisplayskip0.5em
\begin{eqnarray*}
\scalebox{.95}{
\begin{xy}
  \xymatrix{
  (0,+)\ar@/^/[dr]^{f_+}\\
  &(1|1)^G \ar@/^/[ul]^{g_+}\ar@/^/[dl]^{g_-}\ar@/^/[r]^{f_1}
  &(2|2)^G \ar@/^/[l]^{g_1}\ar@/^/[r]^{f_2}
  &(3|3)^G \ar@/^/[l]^{g_2}\ar@/^/[r]^{f_3}
  &(4|4)^G \ar@/^/[l]^{g_3}\cdots\\
 (0,-)\ar@/^/[ur]^{f_-}}
\end{xy}
}
\end{eqnarray*}
\normalsize
subject to the same kind of zero relations as above. Moreover, the induced grading via Corollary~\ref{grading} corresponds exactly to the grading given by the path lengths.
Observe that the trivial block of atypicality $1$ here is equivalent to the blocks of atypicality $1$ for $\op{OSp}(3|2)$ (in contrast to the case of $\op{SOSp}(3|2)$).

\subsection{Illustration of the Dimension Formula for \texorpdfstring{$\OSP(4|4)$}{OSp(4,4)}} \label{sec:osp44}
In this section we apply Theorem~\ref{mainprop} respectively the Dimension Formula (Theorem~\ref{maintheorem}) to calculate the (graded) dimensions of the morphism spaces between certain projective indecomposable modules in the principal block for $\OSP(4|4)$. In \eqref{cupdiag} one can find a list of cup diagrams, whose weight sequences are all diagrammatically linked and in the same block as the trivial representation with sign $+$.

\abovedisplayskip0.15cm
\belowdisplayskip0.15cm
\begin{equation} \label{cupdiag}
\begin{tikzpicture}[thick,anchorbase,scale=0.75]
\node at (-1,-.3) {$\lambda_0^+ = $};
\node at (0,.1) {$\scriptstyle \mathbf{\Diamond}$};
\node at (.5,.1) {$\scriptstyle \mathbf{\up}$};
\node at (1,.1) {$\scriptstyle \mathbf{\up}$};
\node at (1.5,.1) {$\scriptstyle \mathbf{\up}$};
\node at (2,.1) {$\scriptstyle \mathbf{\up}$};
\node at (2.5,.1) {$\scriptstyle \mathbf{\down}$};
\node at (3,.1) {$\mathbf{\scriptstyle{\cdots}}$};
\draw (0,0) .. controls +(0,-.5) and +(0,-.5) .. +(.5,0);
\fill (.25,-.365) circle(2pt);
\draw (1,0) .. controls +(0,-.5) and +(0,-.5) .. +(.5,0);
\fill (1.25,-.365) circle(2pt);
\draw (2,0) -- +(0,-.5);
\fill (2,-.3) circle(2pt);
\draw (2.5,0) -- +(0,-.5);

\begin{scope}[xshift=5cm]
\node at (-1,-.3) {$\lambda_1^+ = $};
\node at (0,.1) {$\scriptstyle \mathbf{\Diamond}$};
\node at (.5,.1) {$\scriptstyle \mathbf{\down}$};
\node at (1,.1) {$\scriptstyle \mathbf{\up}$};
\node at (1.5,.1) {$\scriptstyle \mathbf{\up}$};
\node at (2,.1) {$\scriptstyle \mathbf{\up}$};
\node at (2.5,.1) {$\scriptstyle \mathbf{\down}$};
\node at (3,.1) {$\mathbf{\scriptstyle{\cdots}}$};
\draw (0,0) .. controls +(0,-.75) and +(0,-.75) .. +(1.5,0);
\draw (.5,0) .. controls +(0,-.5) and +(0,-.5) .. +(.5,0);
\draw (2,0) -- +(0,-.5);
\fill (2,-.3) circle(2pt);
\draw (2.5,0) -- +(0,-.5);
\end{scope}

\begin{scope}[xshift=10cm]
\node at (-1,-.3) {$\lambda_2^+ = $};
\node at (0,.1) {$\scriptstyle \mathbf{\Diamond}$};
\node at (.5,.1) {$\scriptstyle \mathbf{\up}$};
\node at (1,.1) {$\scriptstyle \mathbf{\down}$};
\node at (1.5,.1) {$\scriptstyle \mathbf{\up}$};
\node at (2,.1) {$\scriptstyle \mathbf{\up}$};
\node at (2.5,.1) {$\scriptstyle \mathbf{\down}$};
\node at (3,.1) {$\mathbf{\scriptstyle{\cdots}}$};
\draw (0,0) .. controls +(0,-.5) and +(0,-.5) .. +(.5,0);
\draw (1,0) .. controls +(0,-.5) and +(0,-.5) .. +(.5,0);
\draw (2,0) -- +(0,-.5);
\fill (2,-.3) circle(2pt);
\draw (2.5,0) -- +(0,-.5);
\end{scope}

\begin{scope}[yshift=-1.5cm]
\node at (-1,-.3) {$\lambda_0^- = $};
\node at (0,.1) {$\scriptstyle \mathbf{\Diamond}$};
\node at (.5,.1) {$\scriptstyle \mathbf{\up}$};
\node at (1,.1) {$\scriptstyle \mathbf{\up}$};
\node at (1.5,.1) {$\scriptstyle \mathbf{\up}$};
\node at (2,.1) {$\scriptstyle \mathbf{\down}$};
\node at (2.5,.1) {$\scriptstyle \mathbf{\down}$};
\node at (3,.1) {$\mathbf{\scriptstyle{\cdots}}$};
\draw (0,0) .. controls +(0,-.5) and +(0,-.5) .. +(.5,0);
\draw (1,0) .. controls +(0,-.5) and +(0,-.5) .. +(.5,0);
\fill (1.25,-.365) circle(2pt);
\draw (2,0) -- +(0,-.5);
\draw (2.5,0) -- +(0,-.5);
\end{scope}

\begin{scope}[yshift=-1.5cm,xshift=5cm]
\node at (-1,-.3) {$\lambda_1^- = $};
\node at (0,.1) {$\scriptstyle \mathbf{\Diamond}$};
\node at (.5,.1) {$\scriptstyle \mathbf{\down}$};
\node at (1,.1) {$\scriptstyle \mathbf{\up}$};
\node at (1.5,.1) {$\scriptstyle \mathbf{\up}$};
\node at (2,.1) {$\scriptstyle \mathbf{\down}$};
\node at (2.5,.1) {$\scriptstyle \mathbf{\down}$};
\node at (3,.1) {$\mathbf{\scriptstyle{\cdots}}$};
\draw (0,0) .. controls +(0,-.75) and +(0,-.75) .. +(1.5,0);
\fill (.75,-.55) circle(2pt);
\draw (.5,0) .. controls +(0,-.5) and +(0,-.5) .. +(.5,0);
\draw (2,0) -- +(0,-.5);
\draw (2.5,0) -- +(0,-.5);
\end{scope}

\begin{scope}[yshift=-1.5cm,xshift=10cm]
\node at (-1,-.3) {$\lambda_2^- = $};
\node at (0,.1) {$\scriptstyle \mathbf{\Diamond}$};
\node at (.5,.1) {$\scriptstyle \mathbf{\up}$};
\node at (1,.1) {$\scriptstyle \mathbf{\down}$};
\node at (1.5,.1) {$\scriptstyle \mathbf{\up}$};
\node at (2,.1) {$\scriptstyle \mathbf{\down}$};
\node at (2.5,.1) {$\scriptstyle \mathbf{\down}$};
\node at (3,.1) {$\mathbf{\scriptstyle{\cdots}}$};
\draw (0,0) .. controls +(0,-.5) and +(0,-.5) .. +(.5,0);
\fill (.25,-.365) circle(2pt);
\draw (1,0) .. controls +(0,-.5) and +(0,-.5) .. +(.5,0);
\draw (2,0) -- +(0,-.5);
\draw (2.5,0) -- +(0,-.5);
\end{scope}

\begin{scope}[yshift=-3cm]
\node at (-1,-.3) {$\lambda_3^+ = $};
\node at (0,.1) {$\scriptstyle \mathbf{\Diamond}$};
\node at (.5,.1) {$\scriptstyle \mathbf{\up}$};
\node at (1,.1) {$\scriptstyle \mathbf{\up}$};
\node at (1.5,.1) {$\scriptstyle \mathbf{\down}$};
\node at (2,.1) {$\scriptstyle \mathbf{\up}$};
\node at (2.5,.1) {$\scriptstyle \mathbf{\down}$};
\node at (3,.1) {$\mathbf{\scriptstyle{\cdots}}$};
\draw (0,0) .. controls +(0,-.5) and +(0,-.5) .. +(.5,0);
\draw (1,0) -- +(0,-.5);
\fill (1,-.3) circle(2pt);
\draw (1.5,0) .. controls +(0,-.5) and +(0,-.5) .. +(.5,0);
\draw (2.5,0) -- +(0,-.5);
\end{scope}

\begin{scope}[yshift=-4.5cm]
\node at (-1,-.3) {$\lambda_3^- = $};
\node at (0,.1) {$\scriptstyle \mathbf{\Diamond}$};
\node at (.5,.1) {$\scriptstyle \mathbf{\up}$};
\node at (1,.1) {$\scriptstyle \mathbf{\down}$};
\node at (1.5,.1) {$\scriptstyle \mathbf{\down}$};
\node at (2,.1) {$\scriptstyle \mathbf{\up}$};
\node at (2.5,.1) {$\scriptstyle \mathbf{\down}$};
\node at (3,.1) {$\mathbf{\scriptstyle{\cdots}}$};
\draw (0,0) .. controls +(0,-.5) and +(0,-.5) .. +(.5,0);
\fill (.25,-.365) circle(2pt);
\draw (1,0) -- +(0,-.5);
\draw (1.5,0) .. controls +(0,-.5) and +(0,-.5) .. +(.5,0);
\draw (2.5,0) -- +(0,-.5);
\end{scope}

\begin{scope}[yshift=-3cm,xshift=5cm]
\node at (-1,-.3) {$\lambda_4^+ = $};
\node at (0,.1) {$\scriptstyle \mathbf{\Diamond}$};
\node at (.5,.1) {$\scriptstyle \mathbf{\up}$};
\node at (1,.1) {$\scriptstyle \mathbf{\up}$};
\node at (1.5,.1) {$\scriptstyle \mathbf{\down}$};
\node at (2,.1) {$\scriptstyle \mathbf{\down}$};
\node at (2.5,.1) {$\scriptstyle \mathbf{\up}$};
\node at (3,.1) {$\mathbf{\scriptstyle{\cdots}}$};
\draw (0,0) .. controls +(0,-.5) and +(0,-.5) .. +(.5,0);
\draw (1,0) -- +(0,-.5);
\fill (1,-.3) circle(2pt);
\draw (1.5,0) -- +(0,-.5);
\draw (2,0) .. controls +(0,-.5) and +(0,-.5) .. +(.5,0);
\end{scope}

\begin{scope}[yshift=-4.5cm,xshift=5cm]
\node at (-1,-.3) {$\lambda_4^- = $};
\node at (0,.1) {$\scriptstyle \mathbf{\Diamond}$};
\node at (.5,.1) {$\scriptstyle \mathbf{\up}$};
\node at (1,.1) {$\scriptstyle \mathbf{\down}$};
\node at (1.5,.1) {$\scriptstyle \mathbf{\down}$};
\node at (2,.1) {$\scriptstyle \mathbf{\down}$};
\node at (2.5,.1) {$\scriptstyle \mathbf{\up}$};
\node at (3,.1) {$\mathbf{\scriptstyle{\cdots}}$};
\draw (0,0) .. controls +(0,-.5) and +(0,-.5) .. +(.5,0);
\fill (.25,-.365) circle(2pt);
\draw (1,0) -- +(0,-.5);
\draw (1.5,0) -- +(0,-.5);
\draw (2,0) .. controls +(0,-.5) and +(0,-.5) .. +(.5,0);
\end{scope}

\begin{scope}[yshift=-3cm,xshift=10cm]
\node at (-1,-.3) {$\lambda_5 = $};
\node at (0,.1) {$\scriptstyle \mathbf{\Diamond}$};
\node at (.5,.1) {$\scriptstyle \mathbf{\down}$};
\node at (1,.1) {$\scriptstyle \mathbf{\down}$};
\node at (1.5,.1) {$\scriptstyle \mathbf{\up}$};
\node at (2,.1) {$\scriptstyle \mathbf{\up}$};
\node at (2.5,.1) {$\scriptstyle \mathbf{\down}$};
\node at (3,.1) {$\mathbf{\scriptstyle{\cdots}}$};
\draw (0,0) -- +(0,-.5);
\fill (0,-.3) circle(2pt);
\draw (.5,0) .. controls +(0,-.75) and +(0,-.75) .. +(1.5,0);
\draw (1,0) .. controls +(0,-.5) and +(0,-.5) .. +(.5,0);
\draw (2.5,0) -- +(0,-.5);
\end{scope}

\begin{scope}[yshift=-4.5cm,xshift=10cm]
\node at (-1,-.3) {$\lambda_6 = $};
\node at (0,.1) {$\scriptstyle \mathbf{\Diamond}$};
\node at (.5,.1) {$\scriptstyle \mathbf{\down}$};
\node at (1,.1) {$\scriptstyle \mathbf{\up}$};
\node at (1.5,.1) {$\scriptstyle \mathbf{\down}$};
\node at (2,.1) {$\scriptstyle \mathbf{\up}$};
\node at (2.5,.1) {$\scriptstyle \mathbf{\down}$};
\node at (3,.1) {$\mathbf{\scriptstyle{\cdots}}$};
\draw (0,0) -- +(0,-.5);
\fill (0,-.3) circle(2pt);
\draw (.5,0) .. controls +(0,-.5) and +(0,-.5) .. +(.5,0);
\draw (1.5,0) .. controls +(0,-.5) and +(0,-.5) .. +(.5,0);
\draw (2.5,0) -- +(0,-.5);
\end{scope}
\end{tikzpicture}
\end{equation}
For $\la_0^\pm,\ldots,\la_4^\pm$ above, the hook partition underlying $\la_i^\pm$ is $(i,1^i)$, while the hook partition underlying $\la_5^\pm$ is $(2,2,2)$ and for $\la_6^\pm$ it is $(3,2,2,1)$. \\

{
\begin{figure}[!ht]
\abovedisplayskip0.25em
\belowdisplayskip0.5em
\begin{equation*} \arraycolsep=3pt\def\arraystretch{2}
\begin{array}{c||c;{2pt/2pt}c|c;{2pt/2pt}c|c;{2pt/2pt}c|c;{2pt/2pt}c|c;{2pt/2pt}c|c|c}
& \lambda_0^+ & \lambda_0^- & \lambda_1^+ & \lambda_1^- & \lambda_2^+ & \lambda_2^- & \lambda_3^+ & \lambda_3^- & \lambda_4^+ & \lambda_4^- & \lambda_5 & \lambda_6 \\\hline \hline
\lambda_0^+ & E(q) & 0 & q^2[2] & 0 & 0 & 0 & 0 & q^2[2] & 0 & 0 & 0 & q^2 \\ \hdashline[2pt/2pt]
\lambda_0^- & 0 & E(q) & 0 & q^2[2] & 0 & 0 & q^2[2] & 0 & 0 & 0 & 0 & q^2 \\ \hline
\lambda_1^+ & q^2[2] & 0 & E(q) & 0 & q^2[2] & 0 & q^2 & q^2 & 0 & 0 & q^2 & q^2[2] \\ \hdashline[2pt/2pt]
\lambda_1^- & 0 & q^2[2] & 0 & E(q) & 0 & q^2[2] & q^2 & q^2 & 0 & 0 & q^2 & q^2[2] \\ \hline
\lambda_2^+ & 0 & 0 & q^2[2] & 0 & E(q) & 0 & q^2[2] & 0 & 0 & 0 & q^2[2] & q^2 \\ \hdashline[2pt/2pt]
\lambda_2^- & 0 & 0 & 0 & q^2[2] & 0 & E(q) & 0 & q^2[2] & 0 & 0 & q^2[2] & q^2 \\ \hline
\lambda_3^+ & 0 & q^2[2] & q^2 & q^2 & q^2[2] & 0 & E(q) & 0 & q^2[2] & 0 & q^2 & q^2[2] \\ \hdashline[2pt/2pt]
\lambda_3^- & q^2[2] & 0 & q^2 & q^2 & 0 & q^2[2] & 0 & E(q) & 0 & q^2[2] & q^2 & q^2[2] \\ \hline
\lambda_4^+ & 0 & 0 & 0 & 0 & 0 & 0 & q^2[2] & 0 & E(q) & 0 & 0 & q^2 \\ \hdashline[2pt/2pt]
\lambda_4^- & 0 & 0 & 0 & 0 & 0 & 0 & 0 & q^2[2] & 0 & E(q) & 0 & q^2 \\ \hline
\lambda_5 & 0 & 0 & q^2 & q^2 & q^2[2] & q^2[2] & q^2 & q^2 & 0 & 0 & E(q) & q^2[2] \\ \hline
\lambda_6 & q^2 & q^2 & q^2[2] & q^2[2] & q^2 & q^2 & q^2[2] & q^2[2] & q^2 & q^2 & q^2[2] & E(q) \\ \hline
\end{array}
\end{equation*}
\abovecaptionskip0cm
\belowcaptionskip0cm
\caption{Hilbert-Poincar\'e polynomials for graded homomorphism spaces}
\label{bigtable}
\end{figure}}
\smallskip
By pairing the cup diagrams from \eqref{cupdiag}in all possible ways to obtain circle diagrams and checking the possible orientations and their degrees, one directly deduces the table in Figure~\ref{bigtable} for the Hilbert-Poincar\'e polynomials of the morphism spaces where we abbreviate $E(q)=q^2[2]^2$ and $[2]=q^{-1}+q$.  Note that the Hilbert-Poincar\'e polynomials of the endomorphism spaces are constant, namely equal to $E(q)$. This is a general phenomenon. By Lemma~\ref{lem:defect}, each block $\mathcal{B}$ has a well-defined defect $\op{def}(\mathcal{B})$ that is the number of cups in each cup diagram. Then by Theorem~\ref{main} and the definition of the diagram algebra, see \cite[Theorem 6.2 and Corollary 8.8]{ESperv}, we always have an isomorphism of algebras $\op{End}_\cB(P(\la))\cong\mC[X]/(X^2)^{\otimes \op{def}}$ with $\op{deg}(X)=2$.

{\intextsep0.4cm
\begin{figure}[!ht]
\abovedisplayskip0cm
\belowdisplayskip0.25cm
\begin{equation*} \arraycolsep=1pt\def\arraystretch{2}
\begin{array}{c||c|c||c|c}
&\multicolumn{2}{c}{\OSP(7|4)} & \multicolumn{2}{c}{\OSP(6|4)}\\ 
\hline
\pla & \cS(\pla) & { \arraycolsep=1pt\def\arraystretch{1}
\begin{array}{c}
(\pla,+) \\  (\pla,-) \end{array}}& \cS(\pla) & { \arraycolsep=1pt\def\arraystretch{1}
\begin{array}{c}
(\pla,+) \\  (\pla,-) \end{array}} \text{ resp. } \pla\\ \hline
\varnothing&
{\scriptstyle (\frac{3}{2},\frac{5}{2},\frac{7}{2},\frac{9}{2},\frac{11}{2},\ldots)}&
{ \arraycolsep=1pt\def\arraystretch{1}
\begin{array}{c}
\begin{tikzpicture}[thick,anchorbase,scale=0.8]
\node at (0,.1) {$\scriptstyle \mathbf{\circ}$};
\node at (.5,.1) {$\scriptstyle \mathbf{\up}$};
\node at (1,.1) {$\scriptstyle \mathbf{\up}$};
\node at (1.5,.1) {$\scriptstyle \mathbf{\up}$};
\node at (2,.1) {$\scriptstyle \mathbf{\up}$};
\node at (2.5,.1) {$\scriptstyle \mathbf{\up}$};
\node at (3,.1) {$\mathbf{\scriptstyle{\cdots}}$};
\draw (.5,0) .. controls +(0,-.5) and +(0,-.5) .. +(.5,0);
\fill (.75,-.365) circle(2pt);
\draw (1.5,0) .. controls +(0,-.5) and +(0,-.5) .. +(.5,0);
\fill (1.75,-.365) circle(2pt);
\draw (2.5,0) -- +(0,-.5);
\fill (2.5,-.3) circle(2pt);
\end{tikzpicture}\\ 
\begin{tikzpicture}[thick,anchorbase,scale=0.8]
\node at (0,.1) {$\scriptstyle \mathbf{\circ}$};
\node at (.5,.1) {$\scriptstyle \mathbf{\down}$};
\node at (1,.1) {$\scriptstyle \mathbf{\up}$};
\node at (1.5,.1) {$\scriptstyle \mathbf{\up}$};
\node at (2,.1) {$\scriptstyle \mathbf{\up}$};
\node at (2.5,.1) {$\scriptstyle \mathbf{\down}$};
\node at (3,.1) {$\mathbf{\scriptstyle{\cdots}}$};
\draw (.5,0) .. controls +(0,-.5) and +(0,-.5) .. +(.5,0);
\fill (.75,-.365) circle(2pt);
\draw (1.5,0) .. controls +(0,-.5) and +(0,-.5) .. +(.5,0);
\fill (1.75,-.365) circle(2pt);
\draw (2.5,0) -- +(0,-.5);
\end{tikzpicture}
\end{array}}
&
{\scriptstyle (1,2,3,4,5,\scriptstyle{\ldots})}&
{ \arraycolsep=1pt\def\arraystretch{1}
\begin{array}{c}
\begin{tikzpicture}[thick,anchorbase,scale=0.8]
\node at (0,.1) {$\scriptstyle \mathbf{\circ}$};
\node at (.5,.1) {$\scriptstyle \mathbf{\up}$};
\node at (1,.1) {$\scriptstyle \mathbf{\up}$};
\node at (1.5,.1) {$\scriptstyle \mathbf{\up}$};
\node at (2,.1) {$\scriptstyle \mathbf{\up}$};
\node at (2.5,.1) {$\scriptstyle \mathbf{\up}$};
\node at (3,.1) {$\mathbf{\scriptstyle{\cdots}}$};
\draw (.5,0) .. controls +(0,-.5) and +(0,-.5) .. +(.5,0);
\fill (.75,-.365) circle(2pt);
\draw (1.5,0) .. controls +(0,-.5) and +(0,-.5) .. +(.5,0);
\fill (1.75,-.365) circle(2pt);
\draw (2.5,0) -- +(0,-.5);
\fill (2.5,-.3) circle(2pt);
\end{tikzpicture}\\ 
\begin{tikzpicture}[thick,anchorbase,scale=0.8]
\node at (0,.1) {$\scriptstyle \mathbf{\circ}$};
\node at (.5,.1) {$\scriptstyle \mathbf{\up}$};
\node at (1,.1) {$\scriptstyle \mathbf{\up}$};
\node at (1.5,.1) {$\scriptstyle \mathbf{\up}$};
\node at (2,.1) {$\scriptstyle \mathbf{\up}$};
\node at (2.5,.1) {$\scriptstyle \mathbf{\down}$};
\node at (3,.1) {$\mathbf{\scriptstyle{\cdots}}$};
\draw (.5,0) .. controls +(0,-.5) and +(0,-.5) .. +(.5,0);
\fill (.75,-.365) circle(2pt);
\draw (1.5,0) .. controls +(0,-.5) and +(0,-.5) .. +(.5,0);
\fill (1.75,-.365) circle(2pt);
\draw (2.5,0) -- +(0,-.5);
\end{tikzpicture}
\end{array}}
\\ 
\hline
\scalebox{.4}{\yng(1)}&
{\scriptstyle (\frac{1}{2},\frac{5}{2},\frac{7}{2},\frac{9}{2},\frac{11}{2},\scriptstyle{\ldots})}&
{ \arraycolsep=1pt\def\arraystretch{1}
\begin{array}{c}
\begin{tikzpicture}[thick,anchorbase,scale=0.8]
\node at (0,.1) {$\scriptstyle \mathbf{\up}$};
\node at (.5,.1) {$\scriptstyle \mathbf{\circ}$};
\node at (1,.1) {$\scriptstyle \mathbf{\up}$};
\node at (1.5,.1) {$\scriptstyle \mathbf{\up}$};
\node at (2,.1) {$\scriptstyle \mathbf{\up}$};
\node at (2.5,.1) {$\scriptstyle \mathbf{\down}$};
\node at (3,.1) {$\mathbf{\scriptstyle{\cdots}}$};
\draw (0,0) .. controls +(0,-.5) and +(0,-.5) .. +(1,0);
\fill (.5,-.365) circle(2pt);
\draw (1.5,0) .. controls +(0,-.5) and +(0,-.5) .. +(.5,0);
\fill (1.75,-.365) circle(2pt);
\draw (2.5,0) -- +(0,-.5);
\end{tikzpicture}\\ 
\begin{tikzpicture}[thick,anchorbase,scale=0.8]
\node at (0,.1) {$\scriptstyle \mathbf{\up}$};
\node at (.5,.1) {$\scriptstyle \mathbf{\circ}$};
\node at (1,.1) {$\scriptstyle \mathbf{\up}$};
\node at (1.5,.1) {$\scriptstyle \mathbf{\up}$};
\node at (2,.1) {$\scriptstyle \mathbf{\up}$};
\node at (2.5,.1) {$\scriptstyle \mathbf{\up}$};
\node at (3,.1) {$\mathbf{\scriptstyle{\cdots}}$};
\draw (0,0) .. controls +(0,-.5) and +(0,-.5) .. +(1,0);
\fill (.5,-.365) circle(2pt);
\draw (1.5,0) .. controls +(0,-.5) and +(0,-.5) .. +(.5,0);
\fill (1.75,-.365) circle(2pt);
\draw (2.5,0) -- +(0,-.5);
\fill (2.5,-.3) circle(2pt);
\end{tikzpicture}
\end{array}}
&
{\scriptstyle (0,2,3,4,5,\scriptstyle{\ldots})}&
{ \arraycolsep=1pt\def\arraystretch{1}
\begin{array}{c}
\begin{tikzpicture}[thick,anchorbase,scale=0.8]
\node at (0,.1) {$\scriptstyle \mathbf{\Diamond}$};
\node at (.5,.1) {$\scriptstyle \mathbf{\circ}$};
\node at (1,.1) {$\scriptstyle \mathbf{\up}$};
\node at (1.5,.1) {$\scriptstyle \mathbf{\up}$};
\node at (2,.1) {$\scriptstyle \mathbf{\up}$};
\node at (2.5,.1) {$\scriptstyle \mathbf{\up}$};
\node at (3,.1) {$\mathbf{\scriptstyle{\cdots}}$};
\draw (0,0) .. controls +(0,-.5) and +(0,-.5) .. +(1,0);
\fill (.5,-.365) circle(2pt);
\draw (1.5,0) .. controls +(0,-.5) and +(0,-.5) .. +(.5,0);
\fill (1.75,-.365) circle(2pt);
\draw (2.5,0) -- +(0,-.5);
\fill (2.5,-.3) circle(2pt);
\end{tikzpicture}\\ 
\begin{tikzpicture}[thick,anchorbase,scale=0.8]
\node at (0,.1) {$\scriptstyle \mathbf{\Diamond}$};
\node at (.5,.1) {$\scriptstyle \mathbf{\circ}$};
\node at (1,.1) {$\scriptstyle \mathbf{\up}$};
\node at (1.5,.1) {$\scriptstyle \mathbf{\up}$};
\node at (2,.1) {$\scriptstyle \mathbf{\up}$};
\node at (2.5,.1) {$\scriptstyle \mathbf{\down}$};
\node at (3,.1) {$\mathbf{\scriptstyle{\cdots}}$};
\draw (0,0) .. controls +(0,-.5) and +(0,-.5) .. +(1,0);
\fill (.5,-.365) circle(2pt);
\draw (1.5,0) .. controls +(0,-.5) and +(0,-.5) .. +(.5,0);
\fill (1.75,-.365) circle(2pt);
\draw (2.5,0) -- +(0,-.5);
\end{tikzpicture}
\end{array}}
\\ 
\hline
\scalebox{.4}{\yng(2)}&
{\scriptstyle (-\frac{1}{2},\frac{5}{2},\frac{7}{2},\frac{9}{2},\frac{11}{2},\scriptstyle{\ldots})}&
{ \arraycolsep=1pt\def\arraystretch{1}
\begin{array}{c}
\begin{tikzpicture}[thick,anchorbase,scale=0.8]
\node at (0,.1) {$\scriptstyle \mathbf{\down}$};
\node at (.5,.1) {$\scriptstyle \mathbf{\circ}$};
\node at (1,.1) {$\scriptstyle \mathbf{\up}$};
\node at (1.5,.1) {$\scriptstyle \mathbf{\up}$};
\node at (2,.1) {$\scriptstyle \mathbf{\up}$};
\node at (2.5,.1) {$\scriptstyle \mathbf{\up}$};
\node at (3,.1) {$\mathbf{\scriptstyle{\cdots}}$};
\draw (0,0) .. controls +(0,-.5) and +(0,-.5) .. +(1,0);
\draw (1.5,0) .. controls +(0,-.5) and +(0,-.5) .. +(.5,0);
\fill (1.75,-.365) circle(2pt);
\draw (2.5,0) -- +(0,-.5);
\fill (2.5,-.3) circle(2pt);
\end{tikzpicture}\\ 
\begin{tikzpicture}[thick,anchorbase,scale=0.8]
\node at (0,.1) {$\scriptstyle \mathbf{\down}$};
\node at (.5,.1) {$\scriptstyle \mathbf{\circ}$};
\node at (1,.1) {$\scriptstyle \mathbf{\up}$};
\node at (1.5,.1) {$\scriptstyle \mathbf{\up}$};
\node at (2,.1) {$\scriptstyle \mathbf{\up}$};
\node at (2.5,.1) {$\scriptstyle \mathbf{\down}$};
\node at (3,.1) {$\mathbf{\scriptstyle{\cdots}}$};
\draw (0,0) .. controls +(0,-.5) and +(0,-.5) .. +(1,0);
\draw (1.5,0) .. controls +(0,-.5) and +(0,-.5) .. +(.5,0);
\fill (1.75,-.365) circle(2pt);
\draw (2.5,0) -- +(0,-.5);
\end{tikzpicture}
\end{array}}
&
{\scriptstyle (-1,2,3,4,5,\scriptstyle{\ldots})}&
{ \arraycolsep=1pt\def\arraystretch{1}
\begin{array}{c}
\begin{tikzpicture}[thick,anchorbase,scale=0.8]
\node at (0,.1) {$\scriptstyle \mathbf{\circ}$};
\node at (.5,.1) {$\scriptstyle \mathbf{\down}$};
\node at (1,.1) {$\scriptstyle \mathbf{\up}$};
\node at (1.5,.1) {$\scriptstyle \mathbf{\up}$};
\node at (2,.1) {$\scriptstyle \mathbf{\up}$};
\node at (2.5,.1) {$\scriptstyle \mathbf{\down}$};
\node at (3,.1) {$\mathbf{\scriptstyle{\cdots}}$};
\draw (.5,0) .. controls +(0,-.5) and +(0,-.5) .. +(.5,0);
\draw (1.5,0) .. controls +(0,-.5) and +(0,-.5) .. +(.5,0);
\fill (1.75,-.365) circle(2pt);
\draw (2.5,0) -- +(0,-.5);
\end{tikzpicture}\\ 
\begin{tikzpicture}[thick,anchorbase,scale=0.8]
\node at (0,.1) {$\scriptstyle \mathbf{\circ}$};
\node at (.5,.1) {$\scriptstyle \mathbf{\down}$};
\node at (1,.1) {$\scriptstyle \mathbf{\up}$};
\node at (1.5,.1) {$\scriptstyle \mathbf{\up}$};
\node at (2,.1) {$\scriptstyle \mathbf{\up}$};
\node at (2.5,.1) {$\scriptstyle \mathbf{\up}$};
\node at (3,.1) {$\mathbf{\scriptstyle{\cdots}}$};
\draw (.5,0) .. controls +(0,-.5) and +(0,-.5) .. +(.5,0);
\draw (1.5,0) .. controls +(0,-.5) and +(0,-.5) .. +(.5,0);
\fill (1.75,-.365) circle(2pt);
\draw (2.5,0) -- +(0,-.5);
\fill (2.5,-.3) circle(2pt);
\end{tikzpicture}
\end{array}}
\\  
\hline
\scalebox{.4}{\yng(2,1)}&
{\scriptstyle (-\frac{1}{2},\frac{3}{2},\frac{7}{2},\frac{9}{2},\frac{11}{2},\scriptstyle{\ldots})}&
{ \arraycolsep=1pt\def\arraystretch{1}
\begin{array}{c}
\begin{tikzpicture}[thick,anchorbase,scale=0.8]
\node at (0,.1) {$\scriptstyle \mathbf{\down}$};
\node at (.5,.1) {$\scriptstyle \mathbf{\up}$};
\node at (1,.1) {$\scriptstyle \mathbf{\circ}$};
\node at (1.5,.1) {$\scriptstyle \mathbf{\up}$};
\node at (2,.1) {$\scriptstyle \mathbf{\up}$};
\node at (2.5,.1) {$\scriptstyle \mathbf{\down}$};
\node at (3,.1) {$\mathbf{\scriptstyle{\cdots}}$};
\draw (0,0) .. controls +(0,-.5) and +(0,-.5) .. +(.5,0);
\draw (1.5,0) .. controls +(0,-.5) and +(0,-.5) .. +(.5,0);
\fill (1.75,-.365) circle(2pt);
\draw (2.5,0) -- +(0,-.5);
\end{tikzpicture}\\ 
\begin{tikzpicture}[thick,anchorbase,scale=0.8]
\node at (0,.1) {$\scriptstyle \mathbf{\down}$};
\node at (.5,.1) {$\scriptstyle \mathbf{\up}$};
\node at (1,.1) {$\scriptstyle \mathbf{\circ}$};
\node at (1.5,.1) {$\scriptstyle \mathbf{\up}$};
\node at (2,.1) {$\scriptstyle \mathbf{\up}$};
\node at (2.5,.1) {$\scriptstyle \mathbf{\up}$};
\node at (3,.1) {$\mathbf{\scriptstyle{\cdots}}$};
\draw (0,0) .. controls +(0,-.5) and +(0,-.5) .. +(.5,0);
\draw (1.5,0) .. controls +(0,-.5) and +(0,-.5) .. +(.5,0);
\fill (1.75,-.365) circle(2pt);
\draw (2.5,0) -- +(0,-.5);
\fill (2.5,-.3) circle(2pt);
\end{tikzpicture}
\end{array}}
&
{\scriptstyle (-1,1,3,4,5,\scriptstyle{\ldots})}&
{ \arraycolsep=1pt\def\arraystretch{1}
\begin{array}{c}
\begin{tikzpicture}[thick,anchorbase,scale=0.8]
\node at (0,.1) {$\scriptstyle \mathbf{\circ}$};
\node at (.5,.1) {$\scriptstyle \mathbf{\times}$};
\node at (1,.1) {$\scriptstyle \mathbf{\circ}$};
\node at (1.5,.1) {$\scriptstyle \mathbf{\up}$};
\node at (2,.1) {$\scriptstyle \mathbf{\up}$};
\node at (2.5,.1) {$\scriptstyle \mathbf{\down}$};
\node at (3,.1) {$\mathbf{\scriptstyle{\cdots}}$};
\draw (1.5,0) .. controls +(0,-.5) and +(0,-.5) .. +(.5,0);
\fill (1.75,-.365) circle(2pt);
\draw (2.5,0) -- +(0,-.5);
\end{tikzpicture}\\ 
\begin{tikzpicture}[thick,anchorbase,scale=0.8]
\node at (0,.1) {$\scriptstyle \mathbf{\circ}$};
\node at (.5,.1) {$\scriptstyle \mathbf{\times}$};
\node at (1,.1) {$\scriptstyle \mathbf{\circ}$};
\node at (1.5,.1) {$\scriptstyle \mathbf{\up}$};
\node at (2,.1) {$\scriptstyle \mathbf{\up}$};
\node at (2.5,.1) {$\scriptstyle \mathbf{\up}$};
\node at (3,.1) {$\mathbf{\scriptstyle{\cdots}}$};
\draw (1.5,0) .. controls +(0,-.5) and +(0,-.5) .. +(.5,0);
\fill (1.75,-.365) circle(2pt);
\draw (2.5,0) -- +(0,-.5);
\fill (2.5,-.3) circle(2pt);
\end{tikzpicture}
\end{array}}
\\ 
\hline
\scalebox{.4}{\yng(2,2)}&
{\scriptstyle (-\frac{1}{2},\frac{1}{2},\frac{7}{2},\frac{9}{2},\frac{11}{2},\scriptstyle{\ldots})}&
{ \arraycolsep=1pt\def\arraystretch{1}
\begin{array}{c}
\begin{tikzpicture}[thick,anchorbase,scale=0.8]
\node at (0,.1) {$\scriptstyle \mathbf{\times}$};
\node at (.5,.1) {$\scriptstyle \mathbf{\circ}$};
\node at (1,.1) {$\scriptstyle \mathbf{\circ}$};
\node at (1.5,.1) {$\scriptstyle \mathbf{\up}$};
\node at (2,.1) {$\scriptstyle \mathbf{\up}$};
\node at (2.5,.1) {$\scriptstyle \mathbf{\up}$};
\node at (3,.1) {$\mathbf{\scriptstyle{\cdots}}$};
\draw (1.5,0) .. controls +(0,-.5) and +(0,-.5) .. +(.5,0);
\fill (1.75,-.365) circle(2pt);
\draw (2.5,0) -- +(0,-.5);
\fill (2.5,-.3) circle(2pt);
\end{tikzpicture}\\ 
\begin{tikzpicture}[thick,anchorbase,scale=0.8]
\node at (0,.1) {$\scriptstyle \mathbf{\times}$};
\node at (.5,.1) {$\scriptstyle \mathbf{\circ}$};
\node at (1,.1) {$\scriptstyle \mathbf{\circ}$};
\node at (1.5,.1) {$\scriptstyle \mathbf{\up}$};
\node at (2,.1) {$\scriptstyle \mathbf{\up}$};
\node at (2.5,.1) {$\scriptstyle \mathbf{\down}$};
\node at (3,.1) {$\mathbf{\scriptstyle{\cdots}}$};
\draw (1.5,0) .. controls +(0,-.5) and +(0,-.5) .. +(.5,0);
\fill (1.75,-.365) circle(2pt);
\draw (2.5,0) -- +(0,-.5);
\end{tikzpicture}
\end{array}}
&
{\scriptstyle (-1,0,3,4,5,\scriptstyle{\ldots})}&
{ \arraycolsep=1pt\def\arraystretch{1}
\begin{array}{c}
\begin{tikzpicture}[thick,anchorbase,scale=0.8]
\node at (0,.1) {$\scriptstyle \mathbf{\Diamond}$};
\node at (.5,.1) {$\scriptstyle \mathbf{\down}$};
\node at (1,.1) {$\scriptstyle \mathbf{\circ}$};
\node at (1.5,.1) {$\scriptstyle \mathbf{\up}$};
\node at (2,.1) {$\scriptstyle \mathbf{\up}$};
\node at (2.5,.1) {$\scriptstyle \mathbf{\down}$};
\node at (3,.1) {$\mathbf{\scriptstyle{\cdots}}$};
\draw (0,0) .. controls +(0,-.75) and +(0,-.75) .. +(2,0);
\fill (1,-.565) circle(2pt);
\draw (0.5,0) .. controls +(0,-.5) and +(0,-.5) .. +(1,0);
\draw (2.5,0) -- +(0,-.5);
\end{tikzpicture}\\ 
\begin{tikzpicture}[thick,anchorbase,scale=0.8]
\node at (0,.1) {$\scriptstyle \mathbf{\Diamond}$};
\node at (.5,.1) {$\scriptstyle \mathbf{\down}$};
\node at (1,.1) {$\scriptstyle \mathbf{\circ}$};
\node at (1.5,.1) {$\scriptstyle \mathbf{\up}$};
\node at (2,.1) {$\scriptstyle \mathbf{\up}$};
\node at (2.5,.1) {$\scriptstyle \mathbf{\up}$};
\node at (3,.1) {$\mathbf{\scriptstyle{\cdots}}$};
\draw (0,0) .. controls +(0,-.75) and +(0,-.75) .. +(2,0);
\fill (1,-.565) circle(2pt);
\draw (0.5,0) .. controls +(0,-.5) and +(0,-.5) .. +(1,0);
\draw (2.5,0) -- +(0,-.5);
\fill (2.5,-.3) circle(2pt);
\end{tikzpicture}
\end{array}}
\\ 
\hline
\scalebox{.4}{\yng(3,3)}&
{\scriptstyle (-\frac{3}{2},-\frac{1}{2},\frac{7}{2},\frac{9}{2},\frac{11}{2},\scriptstyle{\ldots})}&
{ \arraycolsep=1pt\def\arraystretch{1}
\begin{array}{c}
\begin{tikzpicture}[thick,anchorbase,scale=0.8]
\node at (0,.1) {$\scriptstyle \mathbf{\down}$};
\node at (.5,.1) {$\scriptstyle \mathbf{\down}$};
\node at (1,.1) {$\scriptstyle \mathbf{\circ}$};
\node at (1.5,.1) {$\scriptstyle \mathbf{\up}$};
\node at (2,.1) {$\scriptstyle \mathbf{\up}$};
\node at (2.5,.1) {$\scriptstyle \mathbf{\up}$};
\node at (3,.1) {$\mathbf{\scriptstyle{\cdots}}$};
\draw (0,0) .. controls +(0,-.75) and +(0,-.75) .. +(2,0);
\draw (0.5,0) .. controls +(0,-.5) and +(0,-.5) .. +(1,0);
\draw (2.5,0) -- +(0,-.5);
\fill (2.5,-.3) circle(2pt);
\end{tikzpicture}\\ 
\begin{tikzpicture}[thick,anchorbase,scale=0.8]
\node at (0,.1) {$\scriptstyle \mathbf{\down}$};
\node at (.5,.1) {$\scriptstyle \mathbf{\down}$};
\node at (1,.1) {$\scriptstyle \mathbf{\circ}$};
\node at (1.5,.1) {$\scriptstyle \mathbf{\up}$};
\node at (2,.1) {$\scriptstyle \mathbf{\up}$};
\node at (2.5,.1) {$\scriptstyle \mathbf{\down}$};
\node at (3,.1) {$\mathbf{\scriptstyle{\cdots}}$};
\draw (0,0) .. controls +(0,-.75) and +(0,-.75) .. +(2,0);
\draw (0.5,0) .. controls +(0,-.5) and +(0,-.5) .. +(1,0);
\draw (2.5,0) -- +(0,-.5);
\end{tikzpicture}
\end{array}}
&
{\scriptstyle (-2,-1,3,4,5,\scriptstyle{\ldots})}&
{ \arraycolsep=1pt\def\arraystretch{1}
\begin{array}{c}
\begin{tikzpicture}[thick,anchorbase,scale=0.8]
\node at (0,.1) {$\scriptstyle \mathbf{\circ}$};
\node at (.5,.1) {$\scriptstyle \mathbf{\down}$};
\node at (1,.1) {$\scriptstyle \mathbf{\down}$};
\node at (1.5,.1) {$\scriptstyle \mathbf{\up}$};
\node at (2,.1) {$\scriptstyle \mathbf{\up}$};
\node at (2.5,.1) {$\scriptstyle \mathbf{\up}$};
\node at (3,.1) {$\mathbf{\scriptstyle{\cdots}}$};
\draw (.5,0) .. controls +(0,-.75) and +(0,-.75) .. +(1.5,0);
\draw (1,0) .. controls +(0,-.5) and +(0,-.5) .. +(.5,0);
\draw (2.5,0) -- +(0,-.5);
\fill (2.5,-.3) circle(2pt);
\end{tikzpicture}\\ 
\begin{tikzpicture}[thick,anchorbase,scale=0.8]
\node at (0,.1) {$\scriptstyle \mathbf{\circ}$};
\node at (.5,.1) {$\scriptstyle \mathbf{\down}$};
\node at (1,.1) {$\scriptstyle \mathbf{\down}$};
\node at (1.5,.1) {$\scriptstyle \mathbf{\up}$};
\node at (2,.1) {$\scriptstyle \mathbf{\up}$};
\node at (2.5,.1) {$\scriptstyle \mathbf{\down}$};
\node at (3,.1) {$\mathbf{\scriptstyle{\cdots}}$};
\draw (.5,0) .. controls +(0,-.75) and +(0,-.75) .. +(1.5,0);
\draw (1,0) .. controls +(0,-.5) and +(0,-.5) .. +(.5,0);
\draw (2.5,0) -- +(0,-.5);
\end{tikzpicture}
\end{array}}
\\ 
\hline
\scalebox{.4}{\yng(2,2,1)}&
{\scriptstyle (-\frac{1}{2},\frac{1}{2},\frac{5}{2},\frac{9}{2},\frac{11}{2},\scriptstyle{\ldots})}&
{ \arraycolsep=1pt\def\arraystretch{1}
\begin{array}{c}
\begin{tikzpicture}[thick,anchorbase,scale=0.8]
\node at (0,.1) {$\scriptstyle \mathbf{\times}$};
\node at (.5,.1) {$\scriptstyle \mathbf{\circ}$};
\node at (1,.1) {$\scriptstyle \mathbf{\up}$};
\node at (1.5,.1) {$\scriptstyle \mathbf{\circ}$};
\node at (2,.1) {$\scriptstyle \mathbf{\up}$};
\node at (2.5,.1) {$\scriptstyle \mathbf{\down}$};
\node at (3,.1) {$\mathbf{\scriptstyle{\cdots}}$};
\draw (1,0) .. controls +(0,-.5) and +(0,-.5) .. +(1,0);
\fill (1.5,-.365) circle(2pt);
\draw (2.5,0) -- +(0,-.5);
\end{tikzpicture}\\ 
\begin{tikzpicture}[thick,anchorbase,scale=0.8]
\node at (0,.1) {$\scriptstyle \mathbf{\times}$};
\node at (.5,.1) {$\scriptstyle \mathbf{\circ}$};
\node at (1,.1) {$\scriptstyle \mathbf{\up}$};
\node at (1.5,.1) {$\scriptstyle \mathbf{\circ}$};
\node at (2,.1) {$\scriptstyle \mathbf{\up}$};
\node at (2.5,.1) {$\scriptstyle \mathbf{\up}$};
\node at (3,.1) {$\mathbf{\scriptstyle{\cdots}}$};
\draw (1,0) .. controls +(0,-.5) and +(0,-.5) .. +(1,0);
\fill (1.5,-.365) circle(2pt);
\draw (2.5,0) -- +(0,-.5);
\fill (2.5,-.3) circle(2pt);
\end{tikzpicture}
\end{array}}
&
{\scriptstyle (-1,0,2,4,5,\scriptstyle{\ldots})}&
{ \arraycolsep=1pt\def\arraystretch{1}
\begin{array}{c}
\begin{tikzpicture}[thick,anchorbase,scale=0.8]
\node at (0,.1) {$\scriptstyle \mathbf{\Diamond}$};
\node at (.5,.1) {$\scriptstyle \mathbf{\down}$};
\node at (1,.1) {$\scriptstyle \mathbf{\up}$};
\node at (1.5,.1) {$\scriptstyle \mathbf{\circ}$};
\node at (2,.1) {$\scriptstyle \mathbf{\up}$};
\node at (2.5,.1) {$\scriptstyle \mathbf{\down}$};
\node at (3,.1) {$\mathbf{\scriptstyle{\cdots}}$};
\draw (0,0) .. controls +(0,-.75) and +(0,-.75) .. +(2,0);
\fill (1,-.565) circle(2pt);
\draw (.5,0) .. controls +(0,-.5) and +(0,-.5) .. +(.5,0);
\draw (2.5,0) -- +(0,-.5);
\end{tikzpicture}\\ 
\begin{tikzpicture}[thick,anchorbase,scale=0.8]
\node at (0,.1) {$\scriptstyle \mathbf{\Diamond}$};
\node at (.5,.1) {$\scriptstyle \mathbf{\down}$};
\node at (1,.1) {$\scriptstyle \mathbf{\up}$};
\node at (1.5,.1) {$\scriptstyle \mathbf{\circ}$};
\node at (2,.1) {$\scriptstyle \mathbf{\up}$};
\node at (2.5,.1) {$\scriptstyle \mathbf{\up}$};
\node at (3,.1) {$\mathbf{\scriptstyle{\cdots}}$};
\draw (0,0) .. controls +(0,-.75) and +(0,-.75) .. +(2,0);
\fill (1,-.565) circle(2pt);
\draw (.5,0) .. controls +(0,-.5) and +(0,-.5) .. +(.5,0);
\draw (2.5,0) -- +(0,-.5);
\fill (2.5,-.3) circle(2pt);
\end{tikzpicture}
\end{array}}
\\ 
\hline
\scalebox{.4}{\yng(3,3,3)} \;\;&
{\scriptstyle (-\frac{3}{2},-\frac{1}{2},\frac{1}{2},\frac{9}{2},\frac{11}{2},\scriptstyle{\ldots})}&
{ \arraycolsep=1pt\def\arraystretch{1}
\begin{array}{c}
\begin{tikzpicture}[thick,anchorbase,scale=0.8]
\node at (0,.1) {$\scriptstyle \mathbf{\times}$};
\node at (.5,.1) {$\scriptstyle \mathbf{\down}$};
\node at (1,.1) {$\scriptstyle \mathbf{\circ}$};
\node at (1.5,.1) {$\scriptstyle \mathbf{\circ}$};
\node at (2,.1) {$\scriptstyle \mathbf{\up}$};
\node at (2.5,.1) {$\scriptstyle \mathbf{\up}$};
\node at (3,.1) {$\mathbf{\scriptstyle{\cdots}}$};
\draw (0.5,0) .. controls +(0,-.5) and +(0,-.5) .. +(1.5,0);
\draw (2.5,0) -- +(0,-.5);
\fill (2.5,-.3) circle(2pt);
\end{tikzpicture}\\ 
\begin{tikzpicture}[thick,anchorbase,scale=0.8]
\node at (0,.1) {$\scriptstyle \mathbf{\times}$};
\node at (.5,.1) {$\scriptstyle \mathbf{\down}$};
\node at (1,.1) {$\scriptstyle \mathbf{\circ}$};
\node at (1.5,.1) {$\scriptstyle \mathbf{\circ}$};
\node at (2,.1) {$\scriptstyle \mathbf{\up}$};
\node at (2.5,.1) {$\scriptstyle \mathbf{\down}$};
\node at (3,.1) {$\mathbf{\scriptstyle{\cdots}}$};
\draw (0.5,0) .. controls +(0,-.5) and +(0,-.5) .. +(1.5,0);
\draw (2.5,0) -- +(0,-.5);
\end{tikzpicture}
\end{array}}
&
{\scriptstyle (-2,-1,0,4,5,\scriptstyle{\ldots})}&
\begin{tikzpicture}[thick,anchorbase,scale=0.8]
\node at (0,.1) {$\scriptstyle \mathbf{\Diamond}$};
\node at (.5,.1) {$\scriptstyle \mathbf{\down}$};
\node at (1,.1) {$\scriptstyle \mathbf{\down}$};
\node at (1.5,.1) {$\scriptstyle \mathbf{\circ}$};
\node at (2,.1) {$\scriptstyle \mathbf{\up}$};
\node at (2.5,.1) {$\scriptstyle \mathbf{\up}$};
\node at (3,.1) {$\mathbf{\scriptstyle{\cdots}}$};
\draw (.5,0) .. controls +(0,-.75) and +(0,-.75) .. +(2,0);
\draw (1,0) .. controls +(0,-.5) and +(0,-.5) .. +(1,0);
\draw (0,0) -- +(0,-.5);
\end{tikzpicture}
\end{array}
\end{equation*}
\caption{$\OSP(7|4)$ and $\OSP(6|4)$}
\label{hugeexample}
\end{figure}}

\subsection{Some higher rank examples: \texorpdfstring{$\OSP(7|4)$}{OSp(7,4)} and \texorpdfstring{$\OSP(6|4)$}{OSp(6,4)}} \label{sec:higher}
Finally we calculate the weight and cup diagrams for the two special cases of $\OSP(7|4)$, with $\frac{\de}{2}=\frac{3}{2}$, and $\OSP(6|4)$, with $\frac{\de}{2}=1$.

Here the first column in Figure~\ref{hugeexample} lists the $(3,2)$-hook partition, follows by two columns showing first the sequence $\cS(\pla)$ and then the associated weight and cup diagrams (two if it includes a sign). One can see that in these examples again defects greater than $1$ occur. We leave it to the reader to check that the blocks are all equivalent to blocks which we have seen already (namely to those with the same atypicality).

\bibliographystyle{alpha}
\bibliography{references}
\end{document}